\documentclass[a4paper]{amsart}

\pdfoutput=1

\usepackage[utf8]{inputenc}
\usepackage[T1]{fontenc}
\usepackage{lmodern}
\usepackage{amsthm, amssymb, amsmath, amsfonts, mathrsfs}
\usepackage[colorlinks=true, pdfstartview=FitV, linkcolor=blue, citecolor=blue, urlcolor=blue,pagebackref=false]{hyperref}
\usepackage{booktabs} 
\usepackage{stmaryrd} 

\usepackage{microtype}

\newtheorem{thm}{Theorem}[section]
\newtheorem{prop}[thm]{Proposition}
\newtheorem{lem}[thm]{Lemma}
\newtheorem{cor}[thm]{Corollary}
\theoremstyle{remark}
\newtheorem{rem}[thm]{Remark}
\theoremstyle{definition}
\newtheorem{defi}[thm]{Definition}

\renewcommand{\le}{\leqslant}
\renewcommand{\leq}{\leqslant}
\renewcommand{\ge}{\geqslant}
\renewcommand{\geq}{\geqslant}
\newcommand{\ls}{\lesssim}
\renewcommand{\subset}{\subseteq}
\newcommand{\mcl}{\mathcal}
\newcommand{\msf}{\mathsf}
\newcommand{\mfk}{\mathfrak}
\newcommand{\msc}{\mathscr}
\newcommand{\N}{\mathbb{N}}
\newcommand{\na}{\nabla}
\newcommand{\Ll}{\left}
\newcommand{\Rr}{\right}
\newcommand{\lhs}{left-hand side}
\newcommand{\rhs}{right-hand side}
\newcommand{\1}{\mathbf{1}}
\newcommand{\R}{\mathbb{R}}

\newcommand{\Z}{\mathbb{Z}}
\renewcommand{\P}{\mathbb{P}}
\newcommand{\ov}{\overline}
\newcommand{\un}{\underline}
\newcommand{\td}{\tilde}
\newcommand{\wh}{\widehat}
\newcommand{\eps}{\varepsilon}
\renewcommand{\d}{{\mathrm{d}}}
\newcommand{\dr}{\partial}
\newcommand{\ZM}{\cZ_{M}}
\newcommand{\tZM}{\tilde{\cZ}_{M}}

\newcommand{\al}{\alpha}
\newcommand{\be}{\beta}
\newcommand{\ga}{\gamma}
\newcommand{\de}{\delta}
\newcommand{\ze}{\zeta}
\renewcommand{\th}{\theta}
\newcommand{\ka}{\kappa}
\newcommand{\la}{\lambda}
\newcommand{\si}{\sigma}
\newcommand{\cZ}{\mathsf{Z}}
\newcommand{\DM}{\mathscr{D}_M^{:n:}}

\newcommand{\B}{\mathcal{B}}
\newcommand{\F}{\msc{F}}
\newcommand{\G}{\mcl{G}}
\newcommand{\Gg}{\G^\th_c}

\newcommand{\Bb}{\B^{\al,\mu}_{p,q}}

\newcommand{\Bpq}{\B_{p,q}}

\newcommand{\tBb}{\td{\B}^{\al,M}_{p,q}}
\newcommand{\tBB}[1]{\td{\B}^{#1,M}_{\infty}}
\newcommand{\tB}{\td{\B}}
\newcommand{\hB}{\widehat{\B}}
\newcommand{\dk}{\delta_k}
\newcommand{\Lm}{L_\mu}
\newcommand{\tL}{\td{L}_M}
\newcommand{\hL}{\widehat{L}_\sigma}
\newcommand{\wm}{w_\mu}
\newcommand{\lan}{\left\langle}
\newcommand{\ran}{\right\rangle_M}

\newcommand{\ol}{\varolessthan}
\newcommand{\oo}{\varodot}
\newcommand{\og}{\varogreaterthan}
\newcommand{\dXi}{\delta_k X_M^{(\text{in})}}
\newcommand{\dXo}{\delta_k X_M^{(\text{out})}}

\newcommand{\Zdd}{Z^{( 2 )}}
\newcommand{\Ztt}{Z^{( 3 )}}

\newcommand{\ZZ}{{\widehat{\msc{Z}}^{\sigma}_p}}
\newcommand{\ZZt}{{\widehat{\msc{Z}}^{\sigma}_{3p}}}

\newcommand{\tn}{|\!|\!|} 
\newcommand{\BK}{\mathscr{B}_{T^\star} }
\newcommand{\MT}{\mathscr{M}_{T^\star} }

\DeclareMathOperator{\supp}{Supp}
\newcommand{\Pl}{\mathcal{P}_\ell}
\newcommand{\Plp}{\mathcal{P}_\ell^{\perp}}

\newcommand{\Hhk}{\Hh^{(k)}}
\newcommand{\tZ}{\tilde{\cZ}}
\newcommand{\tc}{\widetilde{\chi}}
\newcommand{\te}{\widetilde{\eta}}
\newcommand{\E}{\mathbb{E}}
\newcommand{\Hh}{\mathcal{H}}
\newcommand{\Ff}{\mathcal{F}}
\newcommand{\ph}{\phi}
\newcommand{\Ss}{\mathcal{S}}

\numberwithin{equation}{section}

\title{Global well-posedness of the dynamic $\Phi^4$ model in the plane}
\author{Jean-Christophe Mourrat, Hendrik Weber}

\address[Jean-Christophe Mourrat]{Ecole normale supérieure de Lyon, CNRS, Lyon, France}
\email{jean-christophe.mourrat@ens-lyon.fr}

\address[Hendrik Weber]{University of Warwick, Coventry, United Kingdom}
\email{hendrik.weber@warwick.ac.uk}

\begin{document}

\begin{abstract}

We show global well-posedness of the dynamic $\Phi^4$ model in the plane. The model is a non-linear stochastic PDE that can only be interpreted in a ``renormalised'' sense. Solutions take values in suitable weighted Besov spaces of negative regularity.

\bigskip


\medskip


\end{abstract}
\maketitle
%
%
%
%
%
%
%
%
\section{Introduction}
\label{s:intro}
The aim of this paper is to show global-in-time well-posedness for the stochastic quantisation equation 
\begin{equation}
\label{e:eqX}
\Ll\{
\begin{array}{ll}
\dr_t X = \Delta X - X^{:3:} + a X +  \xi, \qquad \text{on } \R_+ \times \R^2, \\
X(0,\cdot) = X_0,
\end{array}
\Rr.
\end{equation}
on the full space $\R^2$ and in the probabilistically strong sense. Here $\xi$ denotes a  white noise over $\R \times \R^2$, $a$ is a real parameter and $X^{: 3 :}$ denotes a renormalised cubic power. This cubic power is sometimes referred to as a ``Wick power'' or ``normally ordered power'', and sometimes written in the suggestive form $X^3 - 3  \infty  X$, where ``$\infty$'' stands for a divergent constant that appears in the renormalisation procedure. The initial condition $X_0$ is assumed to take values in a certain Besov space of negative regularity with weights.

Equation \eqref{e:eqX} describes the natural reversible dynamics for the Euclidean $\Phi^4_2$ quantum field theory. It is given by a Gibbs measure on $\mathcal{S}^{\prime}(\R^2)$ which is formally proportional to
\begin{equation}\label{e:phi4}
\exp\Ll(- \int_{\R^2}\Ll[\frac14 X^{: 4 :}  -\frac{a}{2} X^{: 2 :} \Rr]\Rr)  \, \d \nu(X),
\end{equation}
where $\nu$ is the law of a Gaussian free field. This measure was constructed and investigated intensively in the seventies and eighties (see \cite{GJ81} and the references therein). In 1981, Parisi and Wu \cite{ParisiWu} proposed to construct solutions to \eqref{e:eqX} as a means to obtain samples from \eqref{e:phi4} via an MCMC procedure. In this article, we fully perform the construction of solutions to \eqref{e:eqX}. In our companion article \cite{JCH}, we show that solutions of \eqref{e:eqX} on the two-dimensional torus arise as scaling limits for the Glauber dynamics of a ferromagnetic Ising-Kac model near criticality. 

\smallskip

Parisi and Wu's article \cite{ParisiWu} received a lot of attention over the years, and the construction of solutions to ``renormalised'' SPDE has been a recurring theme in the stochastic analysis literature. First results were due to Jona-Lasinio and Mitter 
 \cite{JLM}. Using the Girsanov theorem, they constructed solutions to a modified equation 
 \begin{align}\label{e:JLM}
 \dr_t X = (-\Delta+1)^{-\eps} \big(  \Delta X - X^{:3:} + a X\big) + (-\Delta+1)^{-\frac{\eps}{2}} \xi
 \end{align}
 for  $\frac{9}{10} <\eps <1$, on the two-dimensional torus. Note that (at least formally), this equation also defines reversible dynamics with respect to the $\Phi^4$ measure \eqref{e:phi4}.
 
 In the early nineties, Albeverio and R\"ockner \cite{AlbRock} studied \eqref{e:eqX} using Dirichlet forms. They could show that the Dirichlet form for \eqref{e:eqX} is closable, and thus construct weak solutions to \eqref{e:eqX}. Weak uniqueness for solutions on the torus was shown in \cite{RockShen}.  In \cite{MiRo}, Mikulevicius and Rozovskii developed an alternative approach and constructed martingale solutions to \eqref{e:JLM} on the torus for any value of $\eps \in [0,1)$, and shown uniqueness in law for $\eps>0$. Uniqueness in law for solutions to the original equation  \eqref{e:eqX} on the full space remained open in all of these approaches. 
 
 A breakthrough result was obtained by da Prato and Debussche in \cite{dPD}. They considered \eqref{e:eqX} on the two-dimensional torus and showed short time existence and uniqueness in the probabilistically strong sense, via a fixed point argument in a suitable Besov space. Using the reversible measure \eqref{e:phi4}, they also showed non-explosion for almost every (with respect to this measure) initial datum. 

\smallskip

Our argument builds on this result and extends the method developed in \cite{dPD}. The strategy is similar in spirit to the one-dimensional construction performed in  \cite{Iwa}. 
  Following \cite{dPD}, we first construct periodic solutions for a short time on a torus of arbitrary size, using a fixed point argument. Deviating from \cite{dPD}, we  derive a priori estimates that are strong enough to imply non-explosion on the torus for an arbitrary initial condition in a natural Besov space. We then show that, as the torus grows larger, the family of solutions remains in a compact subset of a suitable Besov space with polynomial weights. This implies the existence of solutions by extracting a converging subsequence.

The proof of uniqueness comes with a twist. The nature of the equation does not allow for a standard Gronwall argument in Besov spaces with polynomial weights. Instead, we ``unfold'' the information of boundedness in such a space into a scale of bounds in Besov spaces with (stretched) exponential weights. We then perform a Gronwall-type argument using this infinite scale of bounds.

\smallskip

Recently, the three-dimensional version of \eqref{e:eqX} has received a lot of attention. In \cite{Martin1}, Hairer developed a theory of ``regularity structures''. The construction of solutions to the three dimensional version of \eqref{e:eqX}  on the torus and for short times was one of the key applications given in \cite{Martin1}. In \cite{MassimilianoPHI4}, Catellier and Chouk presented another method to derive an equivalent result. Their argument is based on the method of  ``modelled distributions'' which was put forward by Gubinelli, Imkeller and Perkowski in \cite{Gubi}. Yet another method to construct short time solutions to the three-dimensional version of \eqref{e:eqX}, using the renormalisation group, was proposed by Kuppiainen in \cite{Kupi}. Four dimensional versions of \eqref{e:eqX} and \eqref{e:phi4} are not expected to exist  \cite{Aizenman}.

The problem addressed in all of these works is somewhat orthogonal to the problem discussed in this article. More precisely, all of these works develop methods to understand the behaviour of a large class of stochastic equations (including \eqref{e:eqX}) on \emph{small} scales. These methods apply to nonlinear stochastic equations satisfying a certain scaling property (called \emph{subcriticality} in \cite{Martin1}) which permits to view the solution of the nonlinear equation as a perturbation of a linearised stochastic equation. It is not expected that a single general theory can give \emph{global} in time non-explosion results for all of these equations in finite or infinite volume; such results will rather have to be obtained case by case. Yet, we hope that the present article will serve as a first step towards proving global-in-time well-posedness for the three-dimensional version of~\eqref{e:eqX}. 

\smallskip

In \cite{Cyril0} and in the forthcoming article  \cite{Cyril}, Hairer and Labb\'e obtain global well-posedness for a  \emph{parabolic Anderson model} on all of $\R_+ \times \R^d$ for $d=2,3$. This model is given by a renormalised version of the stochastic PDE
\begin{align}\label{e:PAM}
\partial_t u = \Delta u + u\, \eta \, ,
\end{align}
where $\eta$ is a white noise \emph{in space}. To this end, they significantly extend the theory of regularity structures to include weights, and they are able to replace some of the $L^{\infty}$-type assumptions from \cite{Martin1} by a more general $L^p$ structure. This part of their work is similar in spirit to our treatment of $L^p$-based Besov spaces with weights below. But their method differs from ours in an important way: Hairer and Labb\'e can directly view \eqref{e:PAM} as a fixed point problem for an operator which is \emph{globally} Lipschitz continuous on some (complicated) space. Their result thus follows by a Picard iteration. We do not expect that such an approach would work  for \eqref{e:eqX}, due to the non-linearity. We are thus forced to adopt the more indirect strategy described above. 

However, our uniqueness argument (see Section~\ref{s:uniqueness}) seems related to the method employed in \cite{Cyril}. Indeed, once the necessary a priori estimates are established, our argument reduces to a uniqueness statement for the heat equation with irregular potential, very much akin to \eqref{e:PAM}.

%
%

\subsection{Statement of the main results}\label{s:mainResult}

Let $\xi$ be space-time white noise on $\R_+ \times \R^2$, and let $a \in \R$. 
%
%
%
We denote by $Z$ be the solution of the stochastic heat equation
\begin{equation}
\label{e:eqZ}
\Ll\{
\begin{array}{ll}
\dr_t Z = \Delta Z  + \xi, \qquad \text{on } \R_+ \times \R^2, \\
Z(0,\cdot) = X_0,
\end{array}
\Rr.
\end{equation}
and denote by $Z^{:2:}, Z^{:3:}$  its Wick powers.
These Wick powers can, for example, be defined by approximation. Let $\rho$ be a mollifying kernel, i.e.\ a compactly supported, non-negative smooth function from $\R \times \R^2$ to $\R$ with $\int \rho =1$. 
For $\delta >0$, set 
\begin{equation*}
\rho_\delta(t,x) = \delta^{-4} \rho\Big(\frac{t}{\delta^2}, \frac{x}{\delta}\Big)\;.
\end{equation*}
Let $Z_\delta$ be the solution of \eqref{e:eqZ} with $\xi$ replaced by the regularised noise $\xi_\delta =\xi \ast \rho_\delta$. There exist constants $\mathfrak{c}_\delta$, which diverge logarithmically as $\delta$ goes to zero, such that
\begin{align*}
Z_\delta^2 - \mathfrak{c}_\delta \qquad Z_\delta^3 - 3 \mathfrak{c}_\delta Z_\delta
\end{align*}
converge to non-trivial limits, which we denote by  $Z^{:2:}$ and $Z^{:3:}$. Such a construction is given, for example, in \cite{dPD,Martin1}. Below, in Section~\ref{sec:Probability} (see \eqref{e:DefZN}) we give an alternative, more direct construction of $Z^{:2:}$ and $Z^{:3:}$. In particular, in Theorem~\ref{thm:BoundsOnZ} and Corollary~\ref{cor:BoundsOnZ2}, we show that for every initial condition $X_0$ in a suitable weighted Besov space, the $Z^{: n :}$ can be realised as random continuous (in time) functions taking values in a weighed Besov space of negative regularity.

Motivated by \cite{dPD}, we then say that $X$ solves the equation \eqref{e:eqX} if $X = Y + Z$, where $Y$ solves
\begin{equation}
\label{e:eqY}
\Ll\{
\begin{array}{ll}
\dr_t Y = \Delta Y  + \Psi(Y,Z,Z^{:2:}, Z^{:3:}), \qquad \text{on } \R_+ \times \R^2, \\
Y(0,\cdot) = 0,
\end{array}
\Rr.
\end{equation}
with
\begin{equation}\label{e:defPsi1}
\Psi(Y,Z,Z^{(2)}, Z^{(3)}) = -Y^3 - 3 Y^2 Z - 3 Y Z^{(2)} - Z^{(3)} + a(Y+Z).
\end{equation}
It turns out that $Y$ is a continuous (in time) function taking values in a Besov space of sufficient positive regularity. Hence, the
non-linear terms in \eqref{e:defPsi1} can be interpreted through multiplicative inequalities in these spaces.
We interpret \eqref{e:eqY} in the mild sense, i.e.\ we say that $Y$ solves \eqref{e:eqY} if for every $t\ge 0$,
\begin{equation}
\label{e:eqY-mild}
Y_t = \int_0^t e^{(t-s) \Delta} \, \Psi(Y_s, \un{Z}_s) \, \d s,
\end{equation}
where $\un{Z} = (Z, Z^{:2:}, Z^{:3:})$.
Our main result states that there exists exactly one solution of \eqref{e:eqY} taking values in some weighted Besov space $\hB^{\be,\si}_{p,\infty}$. In short, the space $\hB^{\be,\si}_{p,\infty}$ is defined analogously to the usual Besov space with regularity index $\be$ and lower indices $p,\infty$, except that the integration is taken against a weight of the form $(1+|x|^2)^{-\si/2}$ (see \eqref{e:def:hB} for a precise definition).
\begin{thm}[Existence and uniqueness of solutions]
\label{t:main} 
Let $\be < 2$, $\si > 2$, $\al > 0$ be sufficiently small and $p <\infty$ be sufficiently large. Let $X_0 \in \hB^{-\al,\si}_{3p,\infty}$, $Z$ be the solution of \eqref{e:eqZ}, and $Z^{:2:}, Z^{:3:}$ denote its Wick powers. With probability one, there exists a unique $Y \in \mcl C(\R_+,\hB^{\be,\si}_{p/9,\infty})$ solving~\eqref{e:eqY}.
\end{thm}

\begin{rem}
As explained above, one of the key steps in the construction of solutions on the full space is to show global-in-time existence and uniqueness for solutions on the torus, for arbitrary initial condition in a natural Besov space of negative regularity. This 
improves on the result from \cite{dPD} where non-explosion is only shown for almost every (with respect to the invariant measure) initial datum. This result is stated in Theorem~\ref{t:global-torus}.
\end{rem}

\begin{rem}
We expect that our method of proof can be modified to imply that 
$(\| X_t \|_{B^{-\alpha, \si}_{p/9, \infty}})_{t \ge 0}$ is a tight family of random variables. By the Krylov-Bogolyubov method, this would give a dynamic construction of a $\Phi^4$ measure formally given by \eqref{e:phi4}. (For some values of the parameter $a$, there are several Gibbs measures formally given by \eqref{e:phi4}, see \cite{GlimmJaffe}.)
\end{rem}

\begin{rem}\label{rem:SmoothApproximation}
As in \cite{Martin1}, our solution $X$ could be obtained as the limit of an approximation procedure with diverging constants:  Let $\xi_\delta = \xi \star \rho_\delta$ be the regularised white noise defined above.
For any $\delta>0$, let $X_\delta$ be the solution of the equation
\begin{align*}
\Ll\{
\begin{array}{ll}
\dr_t X_\delta = \Delta X_\delta - (X_\delta^{3} - 3\mathfrak{c}_\delta X_\delta) + a X_\delta +  \xi_\delta, \\
X_\delta(0,\cdot) = X_0.
\end{array}
\Rr.
\end{align*}
Then  the solutions $X_\delta$ converge to $X$ as $\delta$ goes to zero. Such a result is fully within the scope of the method presented here --- indeed, essentially only some $\delta$-dependent bounds on $Z$ and its Wick powers in Section~\ref{sec:Probability} would  have to be added (see Remark~\ref{rem:RenConst} for a discussion). However, this analysis is not too different from the calculations performed in  \cite{Martin1}, and in order to keep the length of the paper within reason, we refrain from giving this construction. 
\end{rem}

\begin{rem}
As in \cite{dPD}, one could replace the term $-X^{: 3 :}$ by any Wick polynomial of odd degree with negative leading coefficient.
\end{rem}

\subsection{Organisation of the paper}
The first half of the paper is devoted to developing the necessary facts about the different scales of weighted Besov spaces. Weighted Besov spaces have already been studied extensively, see in particular \cite[Chapter~4]{EdmTri}, \cite[Chapter~6]{Tri3} and references therein. Since the precise results that we need are difficult to locate in the literature, and for the reader's convenience, we have chosen to make our paper essentially self-contained. In Section~\ref{s:ineq}, we provide some preliminary results about functions with compactly supported Fourier transform and Gevrey classes. In Section~\ref{s:besov}, we develop all the necessary properties of the weighted Besov spaces in the case of a stretched exponential weight. We conclude 
the discussion of weighted Besov spaces in Section~\ref{s:DifferentWeights}, by indicating how our results change for different choices of weights. 

The actual construction of solutions to \eqref{e:eqX} is performed in the remaining sections. In Section~\ref{sec:Probability}, we recall some probabilistic preliminaries, and give a construction of solutions to the stochastic heat equation \eqref{e:eqZ} and its Wick powers in different weighted Besov spaces. In Section~\ref{s:torus}, we show global-in-time well-posedness for \eqref{e:eqY} in the periodic case. As stated above, we use the strategy developed in \cite{dPD} to construct solutions for a short time, and then derive a priori bounds that are strong enough to show non-explosion for every initial datum in a Besov space of periodic distributions. In Section~\ref{s:full1}, we extend these a priori bounds to the more general case of solutions on the plane. These bounds imply that solutions to \eqref{e:eqY} on tori of diverging size remain in a compact subset of a weighted Besov space. In Section~\ref{s:exist-full}, we show that any limiting point is a mild solution of \eqref{e:eqY}.  Uniqueness is shown in Section~\ref{s:uniqueness}. 

Finally, some standard calculations for Gevrey functions are collected in Appendix~\ref{s:AppA} for the reader's convenience.

\subsection{Notation}

We denote by $\supp f$ the support of the function $f$, by $B(0,r)$ the open Euclidean ball of radius $r$. For $p \in [1,\infty]$, we write $L^p$ for the usual space $L^p(\R^d, \d x)$, whose norm we denote by $\|\cdot \|_{L^p}$. The space of infinitely differentiable functions with compact support is denoted by $C^\infty_c$. For $I = \N$ or $\N \cup \{-1\}$ and $q \in [1,\infty]$, we let 
$$
\|(u_n)_{n \in I}\|_{\ell^q} = \Ll( \sum_{n \in I} |u_n|^q \Rr) ^{1/q}, 
$$
with the usual understanding as a supremum when $q = \infty$. We let
\begin{equation}
\label{e:def:ellstar}
\ell^q = \Ll\{ (u_n)_{n \in I} : \|u\|_{\ell^q} < \infty \text{ and } \lim_{n \to \infty} u_n = 0 \Rr\} .
\end{equation}
(Note that this differs from the usual definition of the space $\ell^q$ only when $q = \infty$; our definition makes the space separable in every case.)

We denote by $\mcl S$ the Schwartz space of smooth functions with rapid decay at infinity, and we denote the dual space of Schwartz distributions by $\mcl S'$. We write $\F f$ or $\hat{f}$ for the Fourier transform of $f$, which is defined by
$$
\F f(\zeta) = \hat{f}(\zeta) = \int e^{-i x \cdot \zeta} f(x) \, \d x
$$
for $f \in L^1$, and can be extended to any $f \in \mcl S'$ by duality. We also write $\F^{-1} f$ for the inverse Fourier transform, which, for $f \in L^1$, takes the form
$$
\F^{-1} f(\ze) = \frac{1}{(2\pi)^d} \int e^{i x \cdot \ze} f(\ze) \, \d \ze.
$$

\medskip

\noindent{\bfseries Acknowledgements.} HW was supported by an EPSRC First Grant.

%

%
%
%
%
%
%
%
%
\section{Functions with compactly supported Fourier transform}
\label{s:ineq}
The goal of this section is to show that a function that has a compactly supported Fourier transform satisfies several regularity properties. For instance, for $p \ge q$, the $L^p$ norm of such a function is controlled by its $L^q$ norm, with a constant that depends only on the location of the support of the Fourier transform. 

While such results are classical for usual $L^p$ spaces, our subsequent analysis requires that we extend these regularity results to weighted spaces. 

\medskip

Three scales of weights will be used in this paper: stretched exponential weights, polynomial weights, and ``flat'' weights on finite cubes for periodic functions. We could possibly have come up with a general framework that contains these three scales as particular cases. However, we find it clearer to focus first on the case of stretched exponential weights, which is the most delicate. We will then indicate why the argument carries over with only notational change to the other cases.

%
%
%
%
%
%
%
\subsection{Gevrey classes}
\label{s:gevrey}
We begin with a brief reminder on Gevrey classes (see also \cite[Chapter 1]{rod}).
\begin{defi}
The Gevrey class of index $\theta \ge 1$, denoted by $\G^\theta$, is the set of infinitely differentiable functions $f : \R^d \to \R$ satisfying
\begin{equation*}
\begin{array}{c}
\text{for every compact } K\text{, there exists } C < \infty \text{ such that for every } n \in \N^d, \\
\sup_{K} \Ll|\dr^n f\Rr| \le C^{|n| +1} \, (n!)^{\theta},
\end{array}
\end{equation*}
where $n!$ stands for $n_1! \cdots n_d!$ and $|n| = n_1+\cdots+n_d$. We let $\Gg$ be the set of compactly supported functions in $\G^\th$.
\end{defi}
Gevrey classes interpolate between analytic functions ($\theta = 1$) and $C^\infty$ functions. They are stable under addition, multiplication and differentiation. (Stability under multiplication is given by Proposition~\ref{p:gevrey-mult} of the appendix; the other stability properties are easier to check.) We have
$$
\Gg = \{0\} \quad \Leftrightarrow \quad \th = 1.
$$
In order to show that $\Gg$ is non trivial for $\theta > 1$ (following \cite[Example 1.4.9]{rod}), we can first show that for $\ka = 1/(\th - 1)$, the function
$$
\ph : 
\Ll\{
\begin{array}{lll}
	\R & \to & \R \\
	x & \mapsto & \exp \Ll( -x^{-\ka} \Rr) \1_{x > 0}
\end{array}
\right.
$$
belongs to $\G^\theta$, and then observe that for any $r > 0$, the function
$$
\Phi :
\Ll\{
\begin{array}{lll}
	\R^d & \to & \R \\
	x=(x_1,\ldots,x_d) & \mapsto & \prod_{i = 1}^d \ph(r+x_i) \ph(r-x_i)
\end{array}
\right.
$$
belongs to $\Gg$. For a given compact $K \subset \R^d$, one can then construct a function in $\Gg$ that is constant equal to $1$ on $K$ and vanishes outside of a given neighbourhood of $K$. Indeed, it suffices to take the convolution of an indicator function with $\Phi$ (for a suitable choice of $r$) and renormalise by $\|\Phi\|_{L^1}$.

We will shortly introduce function spaces with weights that decay roughly as $e^{-|x|^\de}$ ($|x| \to \infty$) for some $\de \in (0,1)$. Functions in $\Gg$ will help us counter-balance the presence of these weights thanks to the following property (which we will in fact use in the ``reverse'' direction, to construct functions with fast decay at infinity with prescribed Fourier transform in $\Gg$).

\begin{prop}[Decay of the Fourier transform]
\label{p:decay}
If $f \in \Gg$, then there exists $c > 0$ and $C < \infty$ such that
\begin{equation}
	\label{e:decay}
	|\hat{f}(\ze)| \le C e^{-c|\ze|^{1/\th}}.
\end{equation}
\end{prop}
The proof is recalled in the appendix, see Proposition~\ref{p:decay-a}.

\subsection{Young and Bernstein inequalities}

A key tool in the derivation of the regularity results we alluded to is Young's inequality. Our starting point is a version thereof that allows for the presence of weights.

\begin{defi}
Let $v, w : \R^d \to \R_+$. We say that $w$ is $v$\emph{-moderate} if for every $x, y \in \R^d$,
\begin{equation}
\label{def:v-moderate}
w(x+y) \le v(x) w(y).
\end{equation}
\end{defi}

\begin{thm}[Weighted Young inequality] 
Let $w$ be $v$-moderate. For every $r,p,q \in [1,\infty]$ satisfying
\begin{equation}
\label{e:exponents-Young}
\frac{1}{r} + 1 = \frac{1}{p} + \frac{1}{q}
\end{equation}
and every measurable functions $f, g : \R^d \to \R$,
$$
\|(f \star g) \, w\|_{L^r} \le \|f \, v\|_{L^p} \, \|g \, w\|_{L^q}.
$$
\label{t:young}
\end{thm}
\begin{proof}
We observe that
\begin{eqnarray*}
|f \star g|(x) w(x) & \le & \int |f|(y) |g|(x-y) w(x) \, \d y \\
	& \stackrel{\eqref{def:v-moderate}}{\le} & \int |f|(y) |g|(x-y) v(y) w(x-y) \, \d y \\
	& \le & \Ll[(|f| v) \star (|g| w) \Rr](x).
\end{eqnarray*}
The result then follows from the classical Young inequality, see e.g.\ \cite[Lemma~1.4]{bcd}.
\end{proof}

We let 
$$
|x|_* = \sqrt{1+|x|^2}
$$
be a smoothened version of the norm $|\cdot|$. Naturally, $|\cdot|_*$ is no longer a norm, but it still satisfies the triangle inequality, since
$$
|x+y|^2_* \le 1 + |x|^2 + |y|^2 + 2 |x| \, |y| 
\le 2+|x|^2 + |y|^2 + 2 |x|_*|y|_*= (|x|_* + |y|_*)^2.
$$
\begin{defi}
Throughout the paper,
\begin{equation}
\label{e:fixdeth}
\emph{we fix } \delta \in (0,1) \emph{ and } \theta \in (1,1/\delta).
\end{equation}
We define
\begin{equation}
\label{def:wm}
\wm(x) = e^{-\mu |x|_*^\de} \qquad (\mu \in \R),
\end{equation}
and let $L^p_\mu$ be the space $L^p(\R^d, \wm(x) \, \d x)$. We denote by $(\, \cdot \,, \, \cdot \,)_{\mu}$ the scalar product in $L^2_\mu$.
\end{defi}
(Since $\delta$ is kept fixed throughout, we choose to leave the dependence on $\delta$ implicit in the notation; the number $\theta$ will come into play shortly.) We impose from now on that $\mu \ge 0$. Note that
\begin{multline}
\label{e:w-moderate}
\frac{\wm(x+y)}{\wm(y)} = \exp \mu \Ll(|y|_*^\de - |x+y|_*^\de\Rr) \\
\le \exp \mu \Ll((|x+y|_*+ |x|_*)^\de - |x+y|_*^\de\Rr) \le \exp \mu |x|_*^\de,
\end{multline}
since we assume $\delta < 1$. Hence, the weight $\wm$ is $w_{-\mu}$-moderate. 

\begin{lem}[Scaling property]
\label{l:scaling}
Let 
$\phi \in \Gg$, $\phi_\la = \phi(\cdot/\la)$ and $g_\la = \F^{-1} \phi_\la$. For every $p \in [1,\infty]$ and $\mu_0 > 0$, there exists $C < \infty$ such that uniformly over $\la \ge 1$ and $\mu \le \mu_0$,
$$
\|g_\la\|_{L_{-\mu}^p} \le C \la^{d/p'},
$$
where $p' \in [1,\infty]$ is the conjugate exponent of $p$, that is, $1/p + 1/p' = 1$.
\end{lem}
\begin{proof}
Since $g_\la = \la^d g_1(\la \cdot)$, the result is clear if $p = \infty$. Otherwise, a change of scale gives 
$$
\|g_\la\|_{L_{-\mu}^p}^p = \la^{d(p-1)} \int |g_1|^p(x) \, e^{\mu|x/\lambda|_*^\delta} \, \d x.
$$
By Proposition~\ref{p:decay}, the latter integral is bounded uniformly over $\mu \le \mu_0$ and $\lambda \ge 1$.
\end{proof}
\begin{rem}
\label{r:scaling-poly}
Similarly, for every $\ga \ge 0$, we have
$
\| |\la \cdot|^\ga \, g_\la \|_{L_{-\mu}^p} \lesssim \la^{d/p'}. 
$
\end{rem}


\begin{lem}[Bernstein's lemma]
Let $B$ be a ball. For every $\mu_0 > 0$, $k = (k_1,\ldots,k_d) \in \N^d$ and $p \ge q \in [1,+\infty]$, there exists $C < \infty$ such that for every $\mu \le \mu_0$ and $\lambda \ge 1$,
$$
\supp \hat{f} \subset \lambda B \quad \Rightarrow \quad \|\partial^k f\|_{L^p_\mu} \le C \,  \lambda^{|k|+d\Ll(\frac{1}{q}- \frac{1}{p}\Rr)} \,  \|f\|_{L^q_{\mu q/p}} \qquad (|k| = k_1+ \cdots + k_d).
$$
\label{l:bernstein}
\end{lem}
\begin{rem}
\label{r:range of f's}
Here and in the two lemmas below, we have not been precise concerning the range of allowed functions $f$. We will only use the lemma for functions belonging to the Schwartz space $\mcl S$ of smooth functions with rapid decay at infinity, so we will prove the result in this setting. It is straightforward to generalise the result to Schwartz distributions (i.e.\ elements of $\mcl S'$). This is not the most general class one can think of, since there are functions in $L^q_{\mu q/p}$ that fail to be intepretable as Schwartz distributions (due to a fast growth at infinity). The main issue then is on the interpretation of $\hat f$ when $f \in L^q_{\mu q/p}$ does not belong to $\mcl S'$. We prefer to not delve into this question (but note that if $f \in L^q_{\mu q/p}$, then one can make sense of $\hat f$ as an element of the dual of $\Gg$ by Proposition~\ref{p:decay}).
\end{rem}
\begin{proof}[Proof of Lemma~\ref{l:bernstein}]
We show the result for finite $p$ and $q$, the adaptation for the remaining cases being transparent. Let 
$\phi \in \Gg$ be such that $\phi = 1$ on $B$, and let $\phi_\la = \phi(\cdot/\la)$. We observe that
$$
f = \F^{-1} \Ll( \hat{f} \phi_\la \Rr) = g_{\la} \star f,
$$
where $g_\la = \F^{-1} \phi_\la = \la^d g_1(\la\cdot)$. Writing $g_\la^{(k)} := (\dr^k g_1)_{\la} = \la^d (\dr^k g_1)(\la \cdot)$, we have
$$
\dr^k f = \la^{|k|} \ g^{(k)}_\la \star f.
$$ 
By Theorem~\ref{t:young},
\begin{multline*}
\la^{-|k|} \|\dr^k f\|_{L^p_\mu} = \|(g^{(k)}_\la \star f) \wm^{1/p}\|_{L^p} \le \|g^{(k)}_\la \, w_{-\mu}^{1/p}\|_{L^r} \, \|f \, \wm^{1/p}\|_{L^q} \\
= \|g^{(k)}_\la \, w_{-\mu}^{1/p}\|_{L^r} \, \|f \|_{L^q_{\mu q/p}} ,
\end{multline*}
where $r$ is such that
\begin{equation}
\label{defr}
1 + \frac{1}{p} = \frac{1}{r} + \frac{1}{q}.
\end{equation}
Since $\|g^{(k)}_\la \, w_{-\mu}^{1/p}\|_{L^r} = \|g^{(k)}_\la \|_{L_{-\mu r/p}^r}$,
the result follows by Lemma~\ref{l:scaling} and \eqref{defr} once we notice that $g^{(k)}_1 = \dr^k g_1$ is the (inverse) Fourier transform of a function in $\Gg$.
\end{proof}

\subsection{Effect of the heat flow}
\label{ss:heat-flow}

We now derive two results that quantify the regularizing and continuity properties of the heat semi-group.

\begin{lem}[Smoothing of the heat flow]
\label{l:smooth-heat} 
Let $\mcl{C}$ be an annulus (i.e.\ a set of the form $\{r \le |x| \le R\}$ for some $0 < r < R$) and $\mu_0 > 0$. There exists $c > 0$ and $C < \infty$ such that for every $p \in [1,+\infty]$, $\mu \le \mu_0$, $t\ge 0$ and $\lambda \ge 1$,
$$
\supp  \hat{f} \subset \lambda \mcl{C} \quad \Rightarrow \quad \|e^{t \Delta} f\|_{L^p_\mu} \le C e^{-ct\lambda^2} \|f\|_{L^p_\mu}.
$$
\end{lem}
\begin{proof}[Proof of Lemma~\ref{l:smooth-heat}]
We show the result for finite $p$, the adaptation to $p = \infty$ being straightforward (and a classical result since the weights no longer matter in this case). 
Let $\phi \in \Gg$ be such that $\phi = 1$ on $\mcl{C}$ and with support in an annulus, and let $\phi_\lambda = \phi(\cdot/\lambda)$. We observe that
$$
e^{t \Delta}f = \F^{-1} \Ll(\hat{f} \, \phi_\lambda \, e^{-t|\cdot|^2} \Rr) = g_{\lambda,t} \star f,
$$
where 
\begin{eqnarray*}
g_{\lambda,t} (x) & = & \frac{1}{(2\pi)^d} \int e^{i x \cdot \ze} \phi_\lambda(\ze) \, e^{-t|\ze|^2} \, \d \ze \\
& = & \frac{\la^d}{(2\pi)^d} \int e^{i \la x \cdot \ze} \phi(\ze) \, e^{-t|\la \ze|^2} \, \d \ze.
\end{eqnarray*}
Let us write
$$
\ov{g}_{\lambda,t}(x) = \int e^{i x \cdot \ze} \phi(\ze) \, e^{-t|\la \ze|^2} \, \d \ze,
$$
so that $g_{\la,t}(x) = (2\pi)^{-d} \la^d \ \ov{g}_{\la,t}(\la x)$. By Proposition~\ref{p:gevrey-ann} of the appendix,
$$
|\ov{g}_{\lambda,t}(x)| \le C  e^{-c\la^2 t - c |x|^{1/\th}}.
$$
By Theorem~\ref{t:young}, 
$$
\|e^{t \Delta} f\|_{L^p_\mu} = \|(g_{\lambda,t} \star f) \, \wm^{1/p}\|_{L^p} \le \|g_{\lambda,t} w_{-\mu}^{1/p}\|_{L^1} \|f \|_{L^p_\mu},
$$
and moreover,
\begin{eqnarray*}
\|g_{\lambda,t} w_{-\mu}^{1/p}\|_{L^1} & \le & {C  e^{-c\la^2 t}} \la^d \int e^{-c|\la x|^{1/\th}} e^{\mu |x|_*^\de/p} \, \d x \\
& \le & {C  e^{-c\la^2 t}}  \int e^{-c|x|^{1/\th}} e^{\mu |x/\la|_*^\de/p} \, \d x  \\
& \stackrel{(\la \ge 1)}{\le} & {C  e^{-c\la^2 t}} \int e^{-c|x|^{1/\th}} e^{\mu |x|_*^\de/p} \, \d x  .
\end{eqnarray*}
The integral is bounded uniformly over $\mu \le \mu_0$, so the result is proved.
\end{proof}
\begin{lem}[time regularity of the heat flow]
	\label{l:time-reg-Lp}
	Let $B$ be a ball and $\mu_0 > 0$. There exists $C < \infty$ such that for every $p \in [1,+\infty]$, $\mu \le \mu_0$, $t\ge 0$ and $\lambda \ge 1$,
	$$
	\supp  \hat{f} \subset \lambda B \quad \Rightarrow \quad \|(1-e^{t \Delta}) f\|_{L^p_\mu} \le C (t\lambda^2 \wedge 1) \, \|f\|_{L^p_\mu}.
	$$
\end{lem}
\begin{proof}
Lemma~\ref{l:smooth-heat} makes it clear that showing
\begin{equation}
	\label{e:time-reg-Lp}
	\|(1-e^{t \Delta}) f\|_{L^p_\mu} \le C t\lambda^2 \, \|f\|_{L^p_\mu}
\end{equation}
is sufficient. Let $\phi \in \Gg$ be such that $\phi = 1$ on $B$. As before, we can write
$$
(1-e^{t\Delta}) f = g_{\la,t} \star f,
$$
but this time with 
$$
g_{\la,t}(x) = \frac{\la^d}{(2\pi)^d} \int e^{i\la x \cdot \ze} \phi(\ze) (1-e^{-t|\la\ze|^2}) \, \d \ze,
$$
which we can in turn decompose as $g_{\la,t}(x) = (2\pi)^{-d} \la^d \ \ov{g}_{\la,t}(\la x)$ for
$$
\ov{g}_{\la,t}(x) = \int e^{i x \cdot \ze} \phi(\ze) (1-e^{-t|\la\ze|^2}) \, \d \ze.
$$
A minor variation of the proof of Proposition~\ref{p:gevrey-ann} shows that 
$$
|\ov{g}_{\la,t}(x)| \le C t \la^2 e^{-c|x|^{1/\th}},
$$
from which \eqref{e:time-reg-Lp} follows as in the proof of Lemma~\ref{l:smooth-heat}.
\end{proof}

%
%
%
%
%
%
%
%
\section{Weighted Besov spaces}
\label{s:besov}
We now introduce weighted Besov spaces, and use the results of the preceding section to deduce several important properties of these spaces. We will study how they relate to each other via continuous (or compact) embeddings and interpolations, their duality properties, the smoothing effect of the heat flow. Besides, a fundamental feature of Besov spaces for our purpose is their multiplicative structure: one can extend the multiplication $(f,g) \to fg$ to a continuous map on suitable Besov spaces. In large measure, our results parallel those for unweighted Besov spaces (and our arguments are inspired by those of \cite[Chapter~2]{bcd}), although the proofs often become more subtle.

\subsection{Definition, continuous embeddings and interpolation}

For future reference, let us define the annulus
\begin{equation}
\label{e:def:mclC}
\mcl{C}^\star = B(0,8/3) \setminus B(0,3/4).
\end{equation}
It is straightforward to adapt the proof of \cite[Proposition~2.10]{bcd} (using Proposition~\ref{p:gevrey-mult}) to show that there exist $\td{\chi}, \chi \in \Gg$ taking values in $[0,1]$ and such that
\begin{equation}
\label{chi-prop1}
\supp \, \td{\chi} \subset B(0,4/3),  
\end{equation}
\begin{equation}
\label{chi-prop2}
\supp \, {\chi} \subset \mcl{C}^\star,
\end{equation}
\begin{equation}
\label{chi-prop3}
\forall \ze \in \R^d, \ \td{\chi}(\ze) + \sum_{k = 0}^{+\infty} \chi(\ze/2^k) = 1.
\end{equation}
We use this dyadic partition of unity to decompose any function $f \in C^\infty_c$ as a sum of functions with localized spectrum. More precisely, we let
\begin{equation}
\label{e:def:chik}
\chi_{-1} = \td{\chi}, \qquad \chi_k = \chi(\cdot/2^k) \quad (k \ge 0),
\end{equation}
and for $k \ge -1$ integer,
$$
\delta_k f = \F^{-1} \Ll(\chi_k \, \hat{f}\Rr), \qquad S_k f = \sum_{j < k} \delta_j f
$$
(where the sum runs over $j \ge -1$), so that at least formally, $S_k f \rightarrow f$ as $k$ tends to infinity. 
For any $\alpha \in \R$, $\mu > 0$, $p,q \in [1,+ \infty]$ and $f \in C^\infty_c$, we define 
\begin{equation*}
\| f \|_{\Bb}:= \Ll[ \sum_{k = -1}^{\infty} \Ll( 2^{\al k } \| \dk f \|_{L^p_\mu} \Rr)^q \Rr]^{\frac{1}{q}}\; = \Ll\| \Ll( 2^{\al k } \| \dk f \|_{L^p_\mu} \Rr)_{k \ge -1} \Rr\|_{\ell^q},
\end{equation*}
with the usual interpretation as a supremum for $q = \infty$. The Besov space $\Bb$ consists of the completion of $C^\infty_c$ with respect to this norm. We denote by $\B^{\al,0}_{p,q}$ the space obtained in the same way but with $L^p_\mu$ replaced by the ``flat'' space $L^p$.

\begin{rem}
\label{r:completion}
We depart from the habit of defining $\Bb$ as the space of distributions for which $\|f\|_{\Bb}$ is finite. Our definition coincides with the usual one as soon as both $p$ and $q$ are finite, but yields a strictly smaller space when at least one of these indices is $\infty$.
There are several advantages to this different definition. First, it slightly simplifies the proofs of some estimates, by letting us show the estimate for functions in $C^\infty_c$ and then using a density argument. It also ensures that the Besov spaces are separable, which has a number of advantages when considering probability measures thereon. But perhaps the most important reason is related to the presence of weights. As was already alluded to in Remark~\ref{r:range of f's}, there is no canonical embedding of $\Bb$ into the space $\mcl S'$ of Schwartz distributions, because elements of $\Bb$ are allowed to grow too fast at infinity. On the other hand, one can check that elements of $\Bb$ define linear forms on $C^\infty_c$. Yet, our definition of $\|f\|_{\Bb}$ does not make sense as it stands for general linear forms on $C^\infty_c$. We believe that it is possible to overcome this problem by making use of results such as Proposition~\ref{p:decay}, and carry this construction on a less standard space of distributions (we think of the dual of the space of functions whose Fourier transform belongs to $\Gg$, with a suitable topology). However, we find our approach technically simpler.
\end{rem}

In the definition, we are using implicitly the fact that $\|f\|_{\Bb}$ is finite for every $f \in C^\infty_c$; this can easily be checked. Before doing so, we introduce the notation
\begin{equation}
\label{e:def:etak}
\eta_k = \F^{-1}(\chi_k), \qquad \eta = \eta_0,
\end{equation}
so that for $k \ge 0$, $\eta_k = 2^{kd} \eta(2^k\cdot)$. 
\begin{lem}
\label{l:smooth}
Let $\al \in \R$ and $p,q \in [1,\infty]$. If $f \in C^\infty_c$, then $\|f\|_{\Bb}$ is finite.
\end{lem}
\begin{proof}
Observe that for $k \ge 0$,
$$
\de_k f(x) = 2^{kd} \int f(y) \eta(2^{k}(x-y)) \, \d y.
$$
The Fourier transform of $\eta$ vanishes in a neighbourhood of the origin. Hence, for every positive integer $n$, the function $\widetilde{\eta}_n := (-\Delta)^{-n} \eta$ is well-defined and in $L^1$, and moreover,
\begin{equation}
\label{bound-dek-smooth}
\de_k f(x) = 2^{k(d-2n)} \int (-\Delta)^n f(y) \  \widetilde{\eta}_n(2^k(x-y)) \, \d y.
\end{equation}
By the (unweighted) Young inequality,
\begin{equation}
\label{e:bound-Sobolev}
\|\de_k f\|_{L^p} \lesssim  2^{k(d-1-2n)} \| (-\Delta)^n f\|_{L^p}.
\end{equation}
In particular, for any given $\al \in \R$, the sequence $(\|\de_k f\|_{L^p})_{k \in \N}$ decays faster than $2^{-\al k}$ as $k$ tends to infinity.
\end{proof}
\begin{rem}
\label{r:besov-mu}
If $\al_1 \le \al_2$, then uniformly over $\mu, p,q$,
\begin{equation}
\label{e:r1}
\|f\|_{\B^{\al_1,\mu}_{p,q}} \lesssim \|f\|_{\B^{\al_2,\mu}_{p,q}},
\end{equation}
where here and throughout, we understand that the inequality above holds for every $f \in \B^{\al_2,\mu}_{p,q}$. Indeed, this is clear if $f \in C^\infty_c$, and then we argue by density.
Similarly, if $q_1 \ge q_2$, then 
\begin{equation}\label{e:pandq}
\|f\|_{\B^{\al,\mu}_{p,q_1}} \le \|f\|_{\B^{\al,\mu}_{p,q_2}},
\end{equation}
while if $p_1 \le p_2$, then thee exists $C < \infty$ such that uniformly over $\al$, $\mu \neq 0$ and $q$,
\begin{equation}
\label{e:r2}
\|f\|_{\B^{\al,\mu}_{p_1,q}} \le C  \mu^{-\frac{d}{\de}\Ll( \frac{1}{p_1} - \frac{1}{p_2}    \Rr)} \|f\|_{\B^{\al,\mu}_{p_2,q}}	.
\end{equation}
\end{rem}
\begin{rem}
\label{r:embed-q}
Since
\begin{equation}
\label{e:estimdk}
\|\dk f\|_{\Lm^p} \le \frac{\| f \|_{\Bb} }{2^{\al k}},
\end{equation}
it follows that for every $\be < \al$,
$$
\|f\|_{\B^{\be,\mu}_{p,1}} = \sum_{k = -1}^{+\infty} 2^{\be k} \|\dk f\|_{\Lm^p} \le \frac{2^{(\al-\be)}}{1-2^{-(\al-\be)}} \|f\|_{\Bb}.
$$
Hence, for every $q,q' \in [1,\infty]$, we have $\|f\|_{\B^{\be,\mu}_{p,q'}} \lesssim \|f\|_{\Bb}$ uniformly over $\mu$. 
\end{rem}
\begin{rem}
\label{r:embed-lp}
The space $\B^{0,\mu}_{p,1}$ is continuously embedded in $\Lm^p$. Indeed, 
$$
\|f\|_{\Lm^p} \le \sum_{k =-1}^{+\infty} \|\dk f\|_{\Lm^p} = \| f \|_{\B^{0,\mu}_{p,1}}.
$$
\end{rem}
\begin{rem}
\label{r:lp-embed}
Conversely, $\Lm^p$ is continuously embedded in $\B^{0,\mu}_{p,\infty}$. Indeed, by Theorem~\ref{t:young},
$$
\|f\|_{\B^{0,\mu}_{p,\infty}} = \sup_{k \ge -1} \|\dk f\|_{\Lm^p} \le \sup_{k \ge -1} \|\eta_k\|_{L^1_{-\mu}} \, \|f\|_{\Lm^p},
$$
with $\sup \|\eta_k\|_{L^1_{-\mu}} < \infty$ by Lemma~\ref{l:scaling}. Moreover, for each given $\mu_0$, the inequality holds uniformly over $\mu \le \mu_0$.
\end{rem}

We also have

\begin{prop}[Besov embedding]
\label{p:embed} Let $\al \le \be \in \R$ and $p \ge r \in [1,\infty]$ be such that 
$$
\be = \alpha + d\Ll(\frac{1}{r} - \frac{1}{p}\Rr),
$$
and let $\mu_0 > 0$. There exists $C < \infty$ such that for every $q \in [1,\infty]$ and $\mu \le \mu_0$, 
$$
\|f \|_{\Bb} \le C \| f \|_{\B^{\be,\mu r/p}_{r,q}}.
$$
\end{prop}
\begin{proof}
It suffices to show the result for $f \in C^\infty_c$. By Lemma~\ref{l:bernstein}, 
$$
\|\dk f \|_{L^p_\mu} \lesssim 2^{k d \Ll(\frac{1}{r} - \frac{1}{p}\Rr)} \, \|\dk f\|_{L^r_{\mu r /p}},
$$
from which the result follows.
\end{proof}

Another notable consequence of Bernstein's lemma is the following.
\begin{prop}[Effect of derivatives]
	\label{p:derivatives}
	Let $\al \in \R$, $k = (k_1,\ldots,k_d) \in \N^d$, $p,q \in [1,\infty]$ and $\mu_0 > 0$. There exists $C < \infty$ such that for every $\mu \le \mu_0$,
	\begin{equation}
	\label{e:derivatives}
	\| \dr^k f \|_{\Bpq^{\al-|k|,\mu}} \le C \|f\|_{\Bb} \qquad (|k| = k_1 + \cdots + k_d).
	\end{equation}
\end{prop}
\begin{rem}
\label{r:derivatives}
We have not given a meaning to $\dr^k f$ for a general $f \in \Bb$. But once \eqref{e:derivatives} is established for arbitrary $f \in C^\infty_c$, we can define $\dr^k f \in \Bpq^{\al-|k|,\mu}$ for any $f \in \Bb$ by means of an approximating sequence in $C^\infty_c$, and \eqref{e:derivatives} is then automatically satisfied for every $f \in \Bb$.
\end{rem}
\begin{proof}[Proof of Proposition~\ref{p:derivatives}]
As for Proposition~\ref{p:embed}, the result is a direct consequence of Lemma~\ref{l:bernstein}, since the latter ensures that
$$
\|\de_l (\dr^k f) \|_{\Lm^p} = \|\dr^k (\de_l  f) \|_{\Lm^p} \lesssim 2^{l |k|} \| \de_l f\|_{\Lm^p}.
$$
\end{proof}

We now turn to interpolation inequalities between Besov spaces. 
\begin{prop}[Interpolation inequalities]
Let $\al_0, \al_1 \in \R$, $p_0,q_0,p_1,q_1 \in [1,\infty]$ and $\nu \in [0,1]$. Defining $\al = (1-\nu)\al_0 + \nu \al_1$ and $p,q \in [1,\infty]$ such that
$$
\frac{1}{p} = \frac{1-\nu}{p_0} + \frac{\nu}{p_1} \quad \text{ and } \quad \frac{1}{q} = \frac{1-\nu}{q_0} + \frac{\nu}{q_1},
$$
we have
$$
\|f\|_{\Bb} \le \|f\|^{1-\nu}_{\B^{\al_0,\mu}_{p_0,q_0}} \, \|f\|^\nu_{\B^{\al_1,\mu}_{p_1,q_1}}.
$$
\label{p:interpol}
\end{prop}
\begin{proof}
We provide a brief proof for the reader's convenience. Recall that we denote by $(\, \cdot \,, \, \cdot \,)_{\mu}$ the scalar product in $L^2_\mu$. We observe that for $f \in C^\infty_c$,
$$
\|f\|_{\Bb} = \lim_{N \to \infty} \sup \sum_{k = -1}^N 2^{\al k} \, \frac{(\dk f,g_k)_\mu \, u_k}{\|g_k\|_{L^{p'}} \, \|u\|_{\ell^{q'}}},
$$
where $p'$ and $q'$ are the conjugate exponents of $p$ and $q$ respectively, and the supremum is taken over all sequences $u$ and all simple functions $g_{-1}, \ldots,g_{N}$ (i.e.\ measurable functions that take a finite number of different values). For a complex number $z$ with $\mathsf{Re}(z) \in [0,1]$, we define
$$
g_{k,z} = |g_k|^{p \Ll[ \frac{1-z}{p_0} + \frac{z}{p_1} \Rr]} \frac{g_k}{|g_k|},
$$
where we understand the \rhs\ as being $0$ on the set $\{g_k = 0\}$. Similarly, we let $u_z$ be the sequence defined by
$$
u_{z,k} = |u_k|^{q \Ll[ \frac{1-z}{q_0} + \frac{z}{q_1} \Rr]} \frac{u_k}{|u_k|},
$$
and 
$$
\psi(z) = \sum_{k = -1}^N 2^{ k\Ll[ (1-z)\al_0 + z \al_1 \Rr]} \, \frac{(\dk f,g_{k,z})_\mu \, u_{z,k}}{\|g_k\|_{L^{p'}} \, \|u\|_{\ell^{q'}}}.
$$
The function $\psi$ is holomorphic and bounded on $\{\mathsf{Re}(z) \in [0,1]\}$, so by the three-lines lemma (see e.g.\ \cite[Appendix to IX.4]{rs2}),
$$
\psi(\nu) \le \sup_{\mathsf{Re}(z) = 0} |\psi(z)|^{1-\nu} \ \sup_{\mathsf{Re}(z) = 1} |\psi(z)|^\nu,
$$
and the result follows by checking that
$$
\sup_{\mathsf{Re}(z) = 0} |\psi(z)| \le \Ll\| \Ll( 2^{\al_0 k} \| \dk f \|_{L^{p_0}_\mu} \Rr)_{-1\le k \le N}  \Rr\|_{\ell^{q_0}} 
$$
and the similar inequality on $\mathsf{Re}(z) = 1$.
\end{proof}

\subsection{Effect of the heat flow}

The smoothing effect of the heat flow, as measured in Besov spaces, takes the following form.
\begin{prop}[Smoothing of the heat flow in Besov spaces]
\label{p:smooth-besov}
Let $\al \ge\be \in \R$, $\mu_0 > 0$ and $p,q \in [1,\infty]$. There exists $C < \infty$ such that uniformly over $\mu \le \mu_0$ and $t > 0$,
$$
\|e^{t\Delta} f\|_{\Bb} \le C \, t^{\frac{\be-\al}{2}} \, \|f\|_{\Bpq^{\be,\mu}}.
$$
\end{prop}
\begin{proof}
Since $\Delta$ acts by multiplication in Fourier space, we have $\dk \Ll(e^{t\Delta} f\Rr) = e^{t\Delta} \Ll(\dk f\Rr)$. By Lemma~\ref{l:smooth-heat},
$$
\|e^{t\Delta} \Ll(\dk f\Rr)\|_{\Lm^p} \lesssim e^{-ct2^{2k}}\, \|\dk f\|_{\Lm^p},
$$
so
$$
2^{\al k} \|\dk \Ll(e^{t\Delta} f\Rr)\|_{\Lm^p} \lesssim t^{\frac{\be-\al}{2}} \,\Ll[(t 2^{2k})^{\frac{\al-\be}{2}} e^{-ct2^{2k}} \Rr] 2^{\be k} \| \dk f\|_{\Lm^p}.
$$
The term between square brackets is bounded uniformly over $t$ and $k$. Taking the $\ell^q$ norm of both sides of the inequality, we get the result.
\end{proof}
\begin{prop}[Time regularity of the heat flow in Besov spaces]
	Let $\al \le \be \in \R$ be such that $\be - \al \le 2$, $\mu_0 > 0$ and $p,q \in [1,\infty]$. There exists $C < \infty$ such that uniformly over $\mu \le \mu_0$ and $t > 0$,
	$$
	\|(1-e^{t\Delta}) f \|_{\Bb} \le C t^{\frac{\be-\al}{2}} \|f\|_{\Bpq^{\be,\mu}}.
	$$
	\label{p:time-reg-flow}
\end{prop}
\begin{proof}
By Lemma~\ref{l:time-reg-Lp},
\begin{align*}
	2^{\al k} \|\dk \Ll(\Ll(1-e^{t\Delta}\Rr) f\Rr)\|_{\Lm^p} & \lesssim 2^{\al k}\Ll(t 2^{2k} \wedge 1\Rr)  \| \dk f\|_{\Lm^p} \\
	& \lesssim t^{\frac{\be-\al}{2}} \ \Ll[\Ll(t 2^{2k} \Rr)^{\frac{\al-\be}{2}} \Ll(t 2^{2k} \wedge 1\Rr) \Rr] \ 2^{\be k} \| \dk f\|_{\Lm^p}.
\end{align*}
The result follows since the term between square brackets is bounded uniformly over $t$ and $k$.
\end{proof}
\begin{rem}
\label{r:continuity}
By the same reasoning, we obtain that if $f \in \Bb$, then $t \mapsto e^{t\Delta} f$ is continuous in $\Bb$. Indeed, it suffices to check the continuity at time $0$. By Lemma~\ref{l:time-reg-Lp}, uniformly over $k$,
$$
\|\dk \Ll(\Ll(1-e^{t\Delta}\Rr) f\Rr)\|_{\Lm^p} \lesssim \|\dk f\|_{\Lm^p},
$$
and for every $k$, the left-hand side above tends to $0$ as $t$ tends to $0$. If $f \in \Bb$, then this ensures that
$$
\Ll\|(1-e^{t\Delta})f\Rr\|_{\Bb} = \Ll\|\Ll( 2^{\al k} \Ll\|\dk \Ll(\Ll(1-e^{t\Delta}\Rr) f\Rr)\Rr\|_{\Lm^p} \Rr)_{k \ge -1} \Rr\|_{\ell^q} \xrightarrow[k \to \infty]{}0,
$$
since $2^{\al k} \|\dk f\|_{\Lm^p} \xrightarrow[k \to \infty]{}0$ even when $q = \infty$, see \eqref{e:def:ellstar} and Remark~\ref{r:completion}.
\end{rem}
\subsection{Multiplicative structure}

The following lemma provides a convenient way to check whether a function belongs to a Besov space. We refer to \cite[Lemma~2.69]{bcd} for a proof (which is an application of Young's inequality). 
\begin{lem}[Series criterion I]
\label{l:crit}
Let $\al \in \R$, $p,q \in [1,\infty]$ and let $\mcl{C}$ be an annulus. There exists $C < \infty$ such that the following holds uniformly over $\mu$. If $(f_k)_{k \in \N}$ is a sequence of functions in $\mcl S$ such that
$$
\supp \hat{f}_k \subset 2^k \mcl{C} \quad \text{and} \quad \Ll( 2^{\al k } \| f_k \|_{L^p_\mu} \Rr)_{k \in \N} \in {\ell^q},
$$
then
$$
f := \sum_{k = 0}^{+\infty}f_k \in \Bb \quad \text{and} \quad \|f\|_{\Bb} \le C \Ll\| \Ll( 2^{\al k } \| f_k \|_{L^p_\mu} \Rr)_{k \in \N} \Rr\|_{\ell^q}.
$$
\end{lem}
\begin{rem}
\label{r:converse-series}
Let $\eta' \in \Gg$ be supported in an annulus, and $\eta'_k = 2^{kd} \eta'(2^k \cdot)$. By the same reasoning, for every $\al \in \R$ and $p,q \in [1,\infty]$, there exists $C < \infty$ such that
$$
\Ll\|\Ll(2^{\al k} \|\eta'_k \star f\|_{\Lm^p}\Rr)_{k \in \N}\Rr\|_{\ell^q} \le C \|f\|_{\Bb}.
$$
In particular, up to an equivalence of norms, the definition of Besov spaces does not depend on the choice of the functions $\td \chi$ and $\chi$.
\end{rem}
If the support of the Fourier transforms are only localized in balls instead of annuli, we get the same result provided that $\alpha > 0$ (see \cite[Lemma~2.84]{bcd}).
\begin{lem}[Series criterion II]
\label{l:crit2}
Let $\al > 0$, $p,q \in [1,\infty]$ and let $B$ be a ball. There exists $C < \infty$ such that the following holds uniformly over $\mu$. If $(f_k)_{k \in \N}$ is a sequence of functions in $\mcl S$ such that
$$
\supp \hat{f}_k \subset 2^k B \quad \text{and} \quad  \Ll( 2^{\al k } \| f_k \|_{L^p_\mu} \Rr)_{k \in \N} \in {\ell^q} ,
$$
then
$$
f := \sum_{k = 0}^{+\infty}f_k \in \Bb \quad \text{and} \quad \|f\|_{\Bb} \le C \Ll\| \Ll( 2^{\al k } \| f_k \|_{L^p_\mu} \Rr)_{k \in \N} \Rr\|_{\ell^q}.
$$
\end{lem}

For $f, g \in C^\infty_c$, we introduce the \emph{paraproduct}
$$
f \ol g = \sum_{j < k-1} \de_j f \ \de_k g = \sum_{k} S_{k-1} f \ \de_k g,
$$
and the \emph{product remainder} (or \emph{resonant term})
$$
f \oo g = \sum_{|j-k| \le 1} \de_j f \ \de_k g,
$$
and we write $f \og g = g \ol f$. We have the \emph{Bony decomposition}
\begin{equation}
\label{e:Bony}
fg = f \ol g + f \oo g + f \og g.
\end{equation}
The goal is to extend the notion of the product $fg$ to elements $f$ and $g$ of suitable Besov spaces, by showing that each of the terms in this decomposition has a natural extension. Here are the key estimates.
\begin{thm}[Paraproduct estimates]
\label{t:para}
(1) Let $\al, \al_1, \al_2 \in \R$ and $p,p_1,p_2,q \in [1,\infty]$ be such that
$$
\al_1 \neq 0, \quad \al = (\al_1\wedge 0) + \al_2 \quad \text{and} \quad \frac{1}{p} = \frac{1}{p_1} + \frac{1}{p_2},
$$
where we write $\al_1\wedge 0 = \min(\al_1,0)$. The mapping $(f,g) \mapsto f \ol g$ (defined for $f,g \in C^\infty_c$) can be extended to a continuous bilinear map from $B^{\al_1,\mu}_{p_1,\infty} \times {B^{\al_2,\mu}_{p_2,q}}$ to $\Bb$, Moreover, there exists $C < \infty$ such that uniformly over~$\mu$,
\begin{equation}
\label{e:para}
\|f \ol g\|_{\B^{\al,\mu}_{p,q}} \le C \, \|f\|_{B^{\al_1,\mu}_{p_1,\infty}} \, \|g\|_{B^{\al_2,\mu}_{p_2,q}}.
\end{equation}
(2) Let  $\al_1, \al_2 \in \R$ be such that $\al := \al_1 + \al_2 > 0$, and let $p,p_1,p_2,q$ be as above. The mapping $(f,g) \mapsto f \oo g$ can be extended to a continuous bilinear map from $B^{\al_1,\mu}_{p_1,\infty} \times {B^{\al_2,\mu}_{p_2,q}}$ to $\Bb$, Moreover, there exists $C < \infty$ such that uniformly over~$\mu$, 
$$
\|f \oo g\|_{\B^{\al,\mu}_{p,q}} \le C \, \|f\|_{B^{\al_1,\mu}_{p_1,\infty}} \, \|g\|_{B^{\al_2,\mu}_{p_2,q}}.
$$
\end{thm}

\begin{proof}
Let $f, g \in C^\infty_c$, and let $\td{\mcl{C}} = B(0,2/3) + \mcl{C}^\star$, where $\mcl{C}^\star$ was defined in \eqref{e:def:mclC}. One can check that $\td{\mcl{C}}$ is an annulus. By (\ref{chi-prop1}--\ref{chi-prop2}), for $k \ge 0$, the Fourier spectrum of $S_{k-1} f \ \de_k g$ is contained in $2^k \td{\mcl{C}}$. (The term indexed by $k = -1$ is null.) By Lemma~\ref{l:crit}, in order to estimate $\|f \ol g\|_{\Bb}$, we need to bound
$$
\Ll\| \Ll( 2^{\al k } \| S_{k-1} f \ \de_k g \|_{L^p_\mu} \Rr)_{k \in \N} \Rr\|_{\ell^q}.
$$
By H\"older's inequality,
\begin{equation}
\label{e:even-weights}
\|S_{k-1} f \ \dk g\|_{L^p_\mu} \le \|S_{k-1} f\|_{L^{p_1}_\mu} \ \| \dk g\|_{L^{p_2}_\mu},
\end{equation}
while by definition,
$$
\|\de_j f\|_{L^{p_1}_\mu} \le 2^{-\al_1 j} \, \|f\|_{\B^{\al_1,\mu}_{p_1,\infty}},
$$
so that
$$
\|S_{k-1} f\|_{\Lm^{p_1}} \le \sum_{j = -1}^{k-2}  \|\de_j f\|_{L^{p_1}_\mu} \le  \sum_{j = -1}^{k-2} 2^{-\al_1 j} \,\|f\|_{\B^{\al_1,\mu}_{p_1,\infty}}\lesssim 2^{-(\al_1\wedge 0) k} \, \|f\|_{\B^{\al_1,\mu}_{p_1,\infty}},
$$
and thus
$$
\|S_{k-1} f \ \dk g\|_{L^p_\mu} \lesssim 2^{-(\al_1\wedge 0) k} \, \|f\|_{\B^{\al_1,\mu}_{p_1,\infty}} \ \| \dk g\|_{L^{p_2}_\mu}.
$$
Multiplying both sides by $2^{\al k}$ and taking the $\ell^q$ norm, we obtain \eqref{e:para}.

For $f \oo g$, we only know that for some ball $\td{B}$, if $|j-k| \le 1$, then the Fourier spectrum of $\de_j f \, \dk g$ is contained in $2^k \td{B}$. We must thus rather use Lemma~\ref{l:crit2} instead of Lemma~\ref{l:crit}. The proof remains the same, except that we need to impose $\al > 0$. 
\end{proof}
\begin{rem}
	\label{r:change-the-weights}
	We can also distribute the weights unevenly between the two terms in the \rhs\ of \eqref{e:para}. Indeed, if $\nu_1,\nu_2 > 0$ are such that
	\begin{equation*}
		\frac{1}{p} = \frac{\nu_1}{p_1} + \frac{\nu_2}{p_2} ,  
	\end{equation*}
	then
	\begin{equation}
		\label{e:modif-Hold}
		\|f g\|_{\Lm^p} = \|f \wm^{\nu_1/p_1} \ g \wm^{\nu_2/p_2}\|_{L^p} \le \|f\|_{L_{\nu_1\mu}^{p_1}} \ \|g\|_{L_{\nu_2 \mu}^{p_2}}
	\end{equation}
	by Hölder's inequality. Using this inequality in \eqref{e:even-weights}, we see that we can replace the \rhs\ of \eqref{e:para} by
	$$
	C \, \|f\|_{\B^{\al_1,\mu_1}_{p_1,\infty}} \, \|g\|_{\B^{\al_2,\mu_2}_{p_2,q}}
	$$
	if desired, provided that $\mu_1 = \nu_1 \mu$ and $\mu_2 = \nu_2 \mu$ satisfy
	\begin{equation}
		\label{e:def:mu1mu2}
		\frac{\mu}{p} = \frac{\mu_1}{p_1} + \frac{\mu_2}{p_2} .  
	\end{equation}
	
\end{rem}
Theorem~\ref{t:para} can be turned into multiplicative inequalities, which will play a central role in our analysis. We treat separately the cases of positive and negative regularity. 
\begin{cor}[Multiplicative inequality I]
\label{c:multipl1}
Let $\al > 0$, $p,q \in [1,\infty]$, $\nu \in [0,1]$, and $\mu_0 > 0$. There exists $C < \infty$ such that for every $\mu \le \mu_0$, if
$$
\be_1 = \al + (1-\nu) \, \frac{d}{p}, \quad \be_2 = \al + \nu \, \frac{d}{p}, \quad \mu_1 = \nu \mu, \quad \mu_2 = (1-\nu)\mu
$$
and if $p_1,p_2 \in [1,\infty]$ are defined by
\begin{equation}
\label{e:def:p12}
\frac{1}{p_1} = \frac{\nu}{p} \quad \text{and} \quad \frac{1}{p_2} = \frac{1-\nu}{p},
\end{equation}
then the mapping $(f,g) \mapsto fg$ can be extended to a continuous linear map from $\B^{\al,\mu}_{p_1,q} \times  {\B^{\al,\mu}_{p_2,q}}$ to $\Bb$, and moreover,
$$
\|f g\|_{\Bb} \le C \|f\|_{\B^{\al,\mu}_{p_1,q}} \, \|g\|_{\B^{\al,\mu}_{p_2,q}} \le C^2 \, \|f\|_{\Bpq^{\be_1,\mu_1}} \, \|g\|_{\Bpq^{\be_2,\mu_2}}.
$$
\end{cor}
\begin{proof}
The first inequality follows from the decomposition~\eqref{e:Bony} and Theorem~\ref{t:para} with $\al_1 = \al_2 = \al$. The second one follows from Proposition~\ref{p:embed}.
\end{proof}
\begin{rem}
	\label{r:multipl1}
	If we assume instead that
	$$
	\frac1p = \frac1{p_1} + \frac1{p_2} \quad \text{ and }  \quad  \frac{\mu}{p} = \frac{\mu_1}{p_1} + \frac{\mu_2}{p_2} ,
	$$
	then by Remark~\ref{r:change-the-weights}, we also have
	$$
	\|f g\|_{\Bb} \le C \|f\|_{\B^{\al,\mu_1}_{p_1,q}} \, \|g\|_{\B^{\al,\mu_2}_{p_2,q}}.
	$$
\end{rem}
\begin{cor}[Multiplicative inequality II]
\label{c:multipl2}
Let $\al < 0 < \be$ be such that $\al + \be > 0$, $p, q \in [1,\infty]$,  $\nu \in [0,1]$ and $\mu_0 > 0$. There exists $C < \infty$ such that for every $\mu \le \mu_0$, if
$$
\al' = \al + (1-\nu) \, \frac{d}{p}, \quad \be' = \be + \nu \, \frac{d}{p}, \quad \mu_1 = \nu \mu, \quad \mu_2 = (1-\nu) \mu,
$$
and if $p_1,p_2 \in [1,\infty]$ are defined by \eqref{e:def:p12}, then the mapping $(f,g) \mapsto fg$ can be extended to a continuous linear map from $\B^{\al,\mu}_{p_1,q} \times  {\B^{\be,\mu}_{p_2,q}}$ to $\Bb$, and moreover,
\begin{equation}
\label{e:multipl2}
\|f g\|_{\Bb} \le C \|f\|_{B^{\al,\mu}_{p_1,q}} \, \|g\|_{B^{\be,\mu}_{p_2,q}} \le C^2 \, \|f\|_{\Bpq^{\al',\mu_1}} \, \|g\|_{\Bpq^{\be',\mu_2}}.
\end{equation}
\end{cor}
\begin{proof}
We get from Theorem~\ref{t:para} that
$$
\|f \ol g\|_{\Bpq^{\al,\mu}} \lesssim \|f \ol g\|_{\Bpq^{\al+\be,\mu}} \lesssim \|f\|_{B^{\al,\mu}_{p_1,q}} \, \|g\|_{B^{\be,\mu}_{p_2,q}},
$$
and the same estimate holds for $f \oo g$ instead of $f \ol g$. Moreover,
$$
\|f \og g \|_{\Bb} \lesssim \|f\|_{B^{\al,\mu}_{p_1,q}} \, \|g\|_{B^{\be,\mu}_{p_2,q}}.
$$
These estimates lead to the first inequality. The second one follows as before using Proposition~\ref{p:embed}.
\end{proof}
\begin{rem}
	\label{r:multipl2}
	If we assume instead that
	$$
	\frac1p = \frac1{p_1} + \frac1{p_2}  \quad \text{ and }  \quad  \frac{\mu}{p} = \frac{\mu_1}{p_1} + \frac{\mu_2}{p_2} ,
	$$
	then by Remark~\ref{r:change-the-weights}, we also have
	$$
	\|f g\|_{\Bb} \le C \|f\|_{B^{\al,\mu_1}_{p_1,q}} \, \|g\|_{B^{\be,\mu_2}_{p_2,q}} .
	$$
\end{rem}

\subsection{Duality}
Recall that $(\cdot,\cdot)_{\mu}$ denotes the scalar product in $L^2_{\mu}$.
\begin{prop}[duality]
\label{p:dual}
Let $\al \in [0,1)$ and $\mu_0 \in \R$. There exists $C < \infty$ such that the following holds. Let $p,q \in [1,\infty]$, $p'$ and $q'$ be their respective conjugate exponents, and $\mu \le \mu_0$. The mapping $(f,g) \mapsto (f,g)_\mu$ (defined for $f,g \in C^\infty_c$) extends uniquely to a continuous bilinear form on $\Bpq^{\al,\mu} \times \B^{-\al,\mu}_{p',q'}$, and moreover,
\begin{equation}
\label{e:dual}
\Ll|(f, g)_{\mu}\Rr| \le C \,  \|f\|_{\Bpq^{\al,\mu}} \, \|g\|_{\B^{-\al,\mu}_{p',q'}}.
\end{equation}
\end{prop}
\begin{proof}
It suffices to show that for every $f,g \in C^\infty_c$,
\begin{equation}
\label{e:dual-goal}
\sum_{k,l\ge -1}\Ll|(\dk f, \de_l g)_\mu\Rr| \lesssim   \|f\|_{\Bpq^{\al,\mu}} \, \|g\|_{\B^{-\al,\mu}_{p',q'}}.
\end{equation}
We decompose the proof into three steps.

\medskip 

\noindent \emph{Step 1.} Let $\ov{\mcl{C}}$ denote the annulus $B(0,10/3) \setminus B(0,1/12)$. Let $\phi \in \Gg$ be such that $\phi(0) = 0$ and $\phi = 1$ on the annulus $\ov{\mcl{C}}$, let $\phi_k = \phi(\cdot/2^k)$ and let $\check{\phi}_k = \F^{-1}(\phi_k)$. In this step, we show that for every $l \ge -1$ and every $k \ge l+2$,
\begin{equation}
\label{e:dual-step1}
(\dk f, \de_l g)_\mu = \int \dk f \, \de_l g \, (w_\mu \star \check{\phi}_k)
\end{equation}
By Parseval's identity,
$$
{(2\pi)^d}(\dk f, \de_l g)_\mu = \int \widehat{\dk f} \, \Ll( \widehat{\de_l g} \star \widehat{w_\mu} \Rr) = \int \widehat{\dk f}(\ze_1) \, \widehat{\de_l g}(\ze_2) \, \widehat{w_\mu}(\ze_1-\ze_2) \, \d \ze_1 \, \d \ze_2.
$$
For $l \ge 0$, the integrand on the \rhs\ is zero unless $\ze_1 \in 2^k \mcl{C}^\star$ and $\ze_2 \in 2^l \mcl{C}^\star$, where we recall that $\mcl{C}^\star$ was defined in \eqref{e:def:mclC}. Since we assume that $k \ge l+2$, the integrand being non-zero implies that $\ze_1-\ze_2 \in 2^k\ov{\mcl{C}}$. This conclusion remains valid when $l = -1$. Hence, multiplying the integrand by $\phi_k(\ze_1-\ze_2)$ does not change the value of the integral, and using Parseval's identity again, we obtain \eqref{e:dual-step1}.

\medskip

\noindent \emph{Step 2.} Let $\al' \in (\al,1)$. We now show that
\begin{equation}
\label{e:dual-step2}
\Ll|w_\mu \star \check{\phi}_k\Rr| \lesssim 2^{-\al'k} w_\mu.
\end{equation}
By a change of variables,
$$
w_\mu \star \check{\phi}_k (x) = \int w_\mu(x-y) \check{\phi}_k(y) \, \d y = \int w_\mu \Ll( x-\frac{y}{2^k} \Rr)  \check{\phi}(y) \, \d y.
$$
Since $\phi(0) = 0$, we have $\int \check{\phi} = 0$, and thus
$$
w_\mu \star \check{\phi}_k (x) = \int \Ll(w_\mu \Ll( x-\frac{y}{2^k} \Rr) -w_\mu(x) \Rr) \check{\phi}(y) \, \d y.
$$
Recall that by Proposition~\ref{p:decay}, $\check{\phi}(y) \lesssim e^{-c|y|^{1/\th}}$ for some $c > 0$, and that
$$
\frac{w_\mu \Ll( x-\frac{y}{2^k} \Rr)}{w_\mu(x)} \le w_{-\mu} \Ll( \frac{y}{2^k} \Rr) .
$$
We first analyse the integral over $|y| \ge 2^{l_0}$, where $l_0 \ge 0$ will be chosen later, and decompose it as the union of $2^l \le |y| < 2^{l+1}$, $l \ge l_0$. We obtain
\begin{multline}
\label{e:dual-step22}
\Ll|\int_{|y|\ge 2^{l_0}} \Ll(w_\mu \Ll( x-\frac{y}{2^k} \Rr) -w_\mu(x) \Rr) \check{\phi}(y) \, \d y \Rr| \\
\lesssim w_\mu(x) \sum_{l = l_0}^{+\infty} 2^{dl} \exp \Ll( 2 \mu 2^{\de(l+1-k)} - c2^{l/\th} \Rr) .
\end{multline}
The summand in the sum above is bounded by
$$
\exp \Ll( 2 \mu 2^{\de l} - c 2^{l/\th} + dl \log(2) \Rr) ,
$$
so if we choose $l_0$ to be the smallest integer such that $2^{l_0} \ge k^{2\theta}$, then the sum in the \rhs\ of \eqref{e:dual-step22} tends to $0$ super-exponentially fast as $k$ tends to infinity. In order to show \eqref{e:dual-step2}, it thus suffices to show that
\begin{equation}
\label{e:dual-step23}
\Ll|\int_{|y|\le 2k^{2\th}} \Ll(w_\mu \Ll( x-\frac{y}{2^k} \Rr) -w_\mu(x) \Rr) \check{\phi}(y) \, \d y \Rr| \lesssim 2^{-\al'k} w_\mu(x).
\end{equation}
By the definition of $\wm$, we have $|\wm(x-z) - \wm(x)| \lesssim |z|  \wm(x)$ uniformly over $x \in \R^d$ and $z$ such that $|z| \le 1$. The \lhs\ of \eqref{e:dual-step23} is thus bounded by a constant times
$$
\frac{k^{2\th}}{2^k} \, \wm(x) \, \|\check{\phi}\|_{L^1}
$$
with $\|\check{\phi}\|_{L^1} < \infty$, so the proof of \eqref{e:dual-step2} is complete.

\medskip 

\noindent \emph{Step 3.} Combining the results of the two previous steps and H\"older's inequality, we obtain that for every $l \ge -1$ and every $k \ge l+2$,
\begin{align}
\Ll|(\dk f, \de_l g)_\mu\Rr| & \lesssim  2^{-\al' k} \int|\dk f \,\de_l g| \, w_\mu \notag \\
& \lesssim  2^{-\al' k} \|\dk f\|_{\Lm^p} \, \|\de_l g\|_{\Lm^{p'}}.
\label{e:even-weight}
\end{align}
For any $k$ and $l$, we also have $\Ll|(\dk f, \de_l g)_\mu\Rr| \le \|\dk f\|_{\Lm^p} \, \|\de_l g\|_{\Lm^{p'}}$, so for any $k$ and $l$, we have
\begin{equation}
\label{e:dual-step3-0}
\Ll|(\dk f, \de_l g)_\mu\Rr| \lesssim 2^{-\al'|k-l|}\, \|\dk f\|_{\Lm^p} \, \|\de_l g\|_{\Lm^{p'}}.
\end{equation}
Let us write $u_k = \1_{k\ge -1} 2^{-\al k} \, \|\dk g\|_{\Lm^{p'}}$ , $v_k =  2^{-(\al'-\al)|k|}$, and $u\star v$ for the convolution of the sequences $u$ and $v$ (indexed by $\Z$).
By H\"older's inequality,
\begin{eqnarray*}
\sum_{k,l\ge -1} \Ll|(\dk f, \de_l g)_\mu\Rr| & \lesssim & \sum_{k\ge -1} 2^{\al k} \|\dk f\|_{\Lm^p} \sum_{l \ge -1} 2^{-\al l} \|\de_l g\|_{\Lm^{p'}}\,  2^{\al(k-l) - \al'|k-l|}\\
& \lesssim & \Ll\| \Ll( 2^{\al k} \|\dk f\|_{\Lm^p} \Rr)_{k \ge -1}  \Rr\|_{\ell^q} \, \Ll\| u \star v \Rr\|_{\ell^{q'}} .
\end{eqnarray*}
By the standard Young inequality for sequences, $\Ll\| u \star v \Rr\|_{\ell^{q'}}\le \Ll\| u \Rr\|_{\ell^{q'}}\Ll\| v \Rr\|_{\ell^{1}}$, and $\Ll\| v \Rr\|_{\ell^{1}} < \infty$, so we obtain \eqref{e:dual-goal}.
\end{proof}
\begin{rem}
	\label{r:dual}
	Although this will not be used below, note that one can spread the weights unevenly in \eqref{e:dual} if desired. Using the modified Hölder inequality \eqref{e:modif-Hold} in \eqref{e:even-weight} shows that for $\mu_1$, $\mu_2 > 0$ such that 
	$$
	\mu = \frac{\mu_1}{p} + \frac{\mu_2}{p'},
	$$
	one has
	$$
	\Ll|(f, g)_{\mu}\Rr| \le C \,  \|f\|_{\Bpq^{\al,\mu_1}} \, \|g\|_{\B^{-\al,\mu_2}_{p',q'}}.
	$$
\end{rem}

\subsection{Besov and Sobolev norms}

We now show that for every $\al \in (0,1)$, one can estimate the Besov norm $\| \, \cdot \,\|_{\B^{\al,\mu}_{1,1}}$ in terms of $\|f\|_{\Lm^1}$ and $\|\nabla f \|_{\Lm^1}$.

\begin{prop}[Estimate in terms of $\na f$]
\label{p:estimate}
Let $\alpha \in (0,1)$ and $\mu_0 > 0$. There exists $C < \infty$ such that for every $\mu \le \mu_0$, 
$$
C^{-1} \| f\|_{\B^{\al,\mu}_{1,1}} \le  \| f \|_{\Lm^1}^{1- \alpha} \, \| \na f \|_{\Lm^1}^\al  + \| f \|_{\Lm^1}\;.
$$
\end{prop}
\begin{rem}
	\label{r:estimate}
	In particular, there exists $C < \infty$ such that
	$$
	C^{-1} \|f\|_{\B^{\alpha,\mu}_{1,1}} \le  \|\nabla f\|_{L^1_\mu} + \|f\|_{L^1_\mu}.
	$$
\end{rem}
\begin{proof}[Proof of Proposition~\ref{p:estimate}]
It suffices to prove the result for $f \in C^\infty_c$. For $\ell \ge 0$, we define the projectors
\begin{align*}
 \Pl f = \sum_{-1 \leq k \leq \ell} \delta_k f \qquad \text{and} \qquad \Plp f = \sum_{k \ge \ell +1} \delta_k f \;,
\end{align*}
so that $f = \Pl f + \Plp f$, and by the triangle inequality,
\begin{align*}
\| f\|_{\B^{\al,\mu}_{1,1}} \leq  \| \Pl f \|_{\B^{\al,\mu}_{1,1}} + \| \Plp f \|_{\B^{\al,\mu}_{1,1}} \;.
\end{align*}
For the first term, recalling \eqref{e:def:etak}, we get 
\begin{align}
\| \Pl f \|_{\B^{\al,\mu}_{1,1}} &= \sum_{-1 \leq k \leq \ell } 2^{k \alpha} \| \delta_k f \|_{L^1_\mu}  \notag
=  \sum_{-1 \leq k \leq \ell } 2^{k \alpha}  \| \eta_k \star f \|_{\Lm^1}\\
&\le  \sum_{-1 \leq k \leq \ell } 2^{k \alpha}  \| \eta_k \|_{L_{-\mu}^1} \, \|  f \|_{\Lm^1} \ls 2^{\ell \alpha} \| f\|_{\Lm^1}\;, \label{e:goodBound1}
\end{align}
where we used Theorem~\ref{t:young} (Young's inequality) and then Lemma~\ref{l:scaling}.
On the other hand, using the fact that for $k \geq 0$, the function $\eta_k$ has vanishing integral, we get
\begin{align*}
\| \Plp f \|_{\B^{\al,\mu}_{1,1}} &= \sum_{ k\geq \ell+1 } 2^{k \alpha} \| \delta_k f \|_{L^1_\mu} \\
&=  \sum_{k\ge \ell+1 } 2^{k \alpha}  \int_{\R^d} \Big|\int_{\R^d} \eta_k(y) \big(f(x-y) -f(x) \big) \,dy \Big| \, e^{-\mu |x|_*^\de} \, dx  \\
&  \leq  \sum_{k\ge \ell+1} 2^{k \alpha}  2^{-k} \int_{\R^d} | 2^{k} y|  \eta_k(y) \Big( \int_{\R^d} \frac{|f(x-y) -f(x)|}{|y|} \,e^{-\mu |x|_*^\de} \, dx \Big)  \, dy  \;.
\end{align*}
By \eqref{e:w-moderate},
$$
\int f(x-y) e^{-\mu |x|^\de} \, dx \le e^{\mu |y|_*^\de} \, \|f\|_{\Lm^1},
$$
and as a consequence,
$$
\int_{|y| > 1} | 2^{k} y|  \eta_k(y) \Big( \int_{\R^d} \frac{|f(x-y) -f(x)|}{|y|} \,e^{-\mu |x|_*^\de} \, dx \Big)  \, dy  \; \le  2 \| \, |2^k \cdot  | \, \eta_k \|_{L_{-\mu}^1} \, \|f\|_{\Lm^1},
$$
with $\| \, |2^k \cdot  | \, \eta_k \|_{L_{-\mu}^1} \ls 1$ by Remark~\ref{r:scaling-poly}.
Moreover, for all $x,y \in \R^d$,
\begin{align*}
\frac{|f(x-y) -f(x)|}{|y|}  =& \frac{1}{|y|} \Big|\int_{0}^1 \nabla f(x -ty) \cdot y \,  dt\Big| \\
\leq& \int_{0}^1 \big| \nabla f(x -ty)\big|  \,  dt,
\end{align*}
so using that $\| \, |2^k \cdot  | \, \eta_k \|_{L^1} \lesssim 1$, 
$$
\int_{|y| \le 1} | 2^{k} y|  \eta_k(y) \Big( \int_{\R^d} \frac{|f(x-y) -f(x)|}{|y|} \,e^{-\mu |x|_*^\de} \, dx \Big)  \, dy  \; \ls   \, \|\na f\|_{\Lm^1}.
$$
To sum up, we have shown that
\begin{align}
\| \Plp f \|_{\B^{\al,\mu}_{1,1}} & \ls  \sum_{k\ge \ell+1 } 2^{k (\alpha-1)}    \left( \| f \|_{\Lm^1 } + \|\na f\|_{\Lm^1} \right) \ls 2^{\ell (\alpha-1)} \left( \| f \|_{\Lm^1 } + \|\na f\|_{\Lm^1}. \right), \label{e:goodBound2}
\end{align}
and thus
$$
\| f\|_{\B^{\al,\mu}_{1,1}} \ls 2^{\ell \al} \| f \|_{\Lm^1 } + 2^{\ell (\alpha-1)} \left( \| f \|_{\Lm^1 } + \|\na f\|_{\Lm^1} \right).
$$
The result then follows by optimizing over $\ell$.
\end{proof}

\subsection{Compact embedding} 
\label{ss:comp}
Finally, we prove a compactness embedding between Besov spaces. An efficient way to proceed is to show that it follows from a similar result in unweighted Besov spaces. We first need to ``project'' onto unweighted spaces.

\begin{prop}[Projection on unweighted Besov spaces]
\label{p:embed-to-unweighted}
Let $\al \in \R$, $p,q \in [1,\infty]$ and $\ph \in C^\infty_c$. The mapping $f \mapsto \ph f$ extends to a continuous linear map from $\Bb$ to $\B^{\al,0}_{p,q}$.
\end{prop}
\begin{proof}
The proof is similar to that of Corollaries~\ref{c:multipl1} and \ref{c:multipl2}. We need a modification of Theorem~\ref{t:para} that allows the destination space to be unweighted. 

Recall the definition of $\eta$ and $(\eta_k)_{k \ge 0}$ in \eqref{e:def:etak}.
By Proposition~\ref{p:decay}, 
$$
|\eta|(x) \lesssim e^{-c |x|^{1/\th}}.
$$
Let $M$ be such that the support of $\ph$ is contained in $[-M/2,M/2]^d$. Since for every $k \ge 0$,
$$
\de_k \ph(x) = 2^{kd} \int \ph(y) \eta(2^{k}(x-y)) \, \d y
$$
with $\ph$ compactly supported, we obtain that for every $k \ge 1$ and $|x| \ge M$ (redefining $c$ to be smaller),
\begin{equation}
\label{e:dec-dek}
|\de_k \ph|(x) \lesssim 2^{kd} e^{-c |2^k x|^{1/\th}},
\end{equation}
and thus for every $|x| \ge M$,
$$
|S_{k} \ph|(x) \le \sum_{j < k} |\de_j \ph|(x) \lesssim e^{-c |x|^{1/\th}}.
$$
By Lemma~\ref{l:smooth}, we also know that $\|S_k \ph\|_{L^\infty}$ is bounded uniformly over $k$. Hence, we can mimic the proof of Theorem~\ref{t:para} but with \eqref{e:even-weights} replaced by
$$
 \|S_{k-1} \ph \ \dk f\|_{L^p} \lesssim \| \dk f\|_{L^{p}_\mu},
$$
where the implicit constant is uniform over $k$. We thus obtain that the mapping $f \mapsto \ph \ol f$ extends to a continuous linear map from $\Bb$ to $\B^{\al,0}_{p,q}$. Similarly, for every $\beta \in \R$,
$$
 \|S_{k-1} f \ \dk \ph\|_{L^p} \lesssim 2^{-\beta k} \| S_{k-1} f\|_{L^{p}_\mu},
$$
where we used \eqref{e:dec-dek} and the fact, ensured by Lemma~\ref{l:smooth}, that $\|\dk \ph\|_{L^\infty} \lesssim 2^{-\beta k}$ (and where the implicit constants depend on~$\beta$). Hence, for every $\be \in \R$, the mapping $f \mapsto f \ol \ph$ extends to a continuous linear map from $\Bb$ to $\B^{\al+\be,0}_{p,q}$. Moreover, the same conclusion also holds for the mapping $f \mapsto \ph \oo f$,  so the proof is complete.
\end{proof}

We will also need the following result.

\begin{lem}[Besov norms on large scales]
\label{l:compact-support} Let $\ph \in C^\infty_c$, and for $m \ge 1$, let $\ph_m = \ph(\cdot/m)$. For every $\al \in \R$, $p,q \in [1,\infty]$ and $\mu > 0$, we have
$$
\sup_{m \ge 1} \|\ph_m\|_{\Bb} < \infty.
$$
Moreover, if the support of $\ph$ is contained in an annulus and $p \neq \infty$, then
$$
\|\ph_m\|_{\Bb} \xrightarrow[m \to \infty]{} 0. 
$$
\end{lem}
\begin{proof}
The identity  $\| \ph_m \|_{L^\infty} = \| \ph \|_{L^\infty}$ is obvious, and 
for $p < \infty$ we have
\begin{align*}
\| \ph_m \|_{L^p_\mu} = \int_{\R^d} |\ph(x/m)|^p \, w_\mu(x) \d x \leq \| \ph \|_{L^\infty} \int_{\R^d} w_{\mu}(x) \d x \;,
\end{align*}
and in particular, the left-hand side is bounded uniformly over $m$. Similarly, for any multi-index $k \in \N^d$, we have $\| \partial^k \ph_m \|_{L^\infty} = \frac{1}{m^{|k|}}\| \partial^k \ph \|_{L^\infty}$, so that $\|\partial^k \ph_m \|_{\Lm^p}$ is bounded uniformly over $m$. 

If in addition, $p < \infty$ and the support of $\ph$ does not intersect $B(0,\eps)$ (for some $\eps > 0$), we get in the same way that
\begin{align*}
\| \partial^k \ph_m \|_{L^p_\mu} \leq \frac{1}{m^{|k|}} \| \partial_k \ph \|_{L^\infty}   \int_{|x| \ge \eps m} w_{\mu}(x) \d x \xrightarrow[m \to \infty]{} 0 \;.
\end{align*}
By \eqref{bound-dek-smooth}, Proposition~\ref{p:decay} and Theorem~\ref{t:young}, we have
$$
\|\de_k \ph_m\|_{L^p_\mu} \lesssim 2^{k(d-2n)} \|(-\Delta)^n \ph_m\|_{\Lm^p},
$$
and we thus obtain the announced results.
\end{proof}
\begin{prop}[Compact embedding]
\label{p:compact}
Let $\al > \al' \in \R$, $p,q \in [1,\infty]$ with $p \neq \infty$, and $\mu' > \mu > 0$. The embedding $\Bb \subset \B^{\al',\mu'}_{p,1}$ is compact.
\end{prop}

\begin{proof}
Let $(f_n)_{n \in \N}$ be a sequence that is bounded in $\Bb$. We wish to find a subsequence that converges in $\B^{\al',\mu'}_{p,1}$. 

Let $\ph \in C^\infty_c$ be such that
$$
\ph = 1 \text{ on } B(0,1), \quad \ph = 0 \text{ outside of } B(0,2),
$$
let $\ph_m = \ph(\cdot /m)$, $\al_1 < \al$ and $\mu_1 > \mu$. For every fixed $m$, up to extracting a subsequence (which we keep denoting by $(f_n)$ for convenience), there exists $f^{(m)} \in \B^{\al_1,0}_{p,1}$ such that 
$$
\|\ph_m f_n - f^{(m)}\|_{\B^{\al_1,0}_{p,1}} \xrightarrow[n \to \infty]{} 0.
$$
Indeed, this follows from Lemma~\ref{p:embed-to-unweighted} and \cite[Theorem~2.94]{bcd}. In particular, we have $f^{(m)} \in \B^{\al_1, \mu_1}_{p,1}$ and
$$
\|\ph_m f_n - f^{(m)}\|_{\B^{\al_1,\mu_1}_{p,1}} \xrightarrow[n \to \infty]{} 0.
$$
The convergence along a subsequence holds for any given $m$, and thus up to a diagonal extraction, it holds for every $m$ simultaneously. 
By the boundedness assumption on the sequence $(f_n)$ and the multiplicative inequalities (Corollaries~\ref{c:multipl1} and \ref{c:multipl2}), 
$$
\| f^{(m)} \|_{\B^{\al_1,\mu_1}_{p,1}} = \lim_{n \to \infty} \| \ph_m f_n \|_{\B^{\al_1,\mu_1}_{p,1}}  
\lesssim \|\ph_m\|_{\B^{\td \al,\td \mu}_{p,1}} \, \limsup_{n \to \infty} \|f_n\|_{\B^{\al,\mu}_{p,1}}\lesssim \|\ph_m\|_{\B^{\td \al,\td \mu}_{p,1}},
$$
for some $\td \al$ and $\td \mu >0$ depending only on $\al_1, \al, \mu$ and $\mu_1$. By Lemma~\ref{l:compact-support}, $\|\ph_m\|_{\B^{\td \al,\td \mu}_{p,1}}$ is bounded uniformly over $m$, so that the same holds for $\| f^{(m)} \|_{\B^{\al_1,\mu_1}_{p,1}}$.

It is clear that for $m \le k$, the distributions $f^{(m)}$ and $f^{(k)}$ coincide on $B(0,m)$. In particular, for $m \le k$,
$$
f^{(m)} - f^{(k)} = (1-\ph_{m/2} )(f^{(m)} - f^{(k)}).
$$
Letting $\al_2 < \al_1$ and $\mu_2 > \mu_1$, we can use the multiplicative inequalities to obtain, for every $m \le k$,
$$
\| f^{(m)} - f^{(k)} \|_{\B^{\al_2,\mu_2}_{p,1}} \lesssim \|1-\ph_m\|_{\B^{\ov \al,\ov \mu}_{p,1}} \, \|f^{(m)} - f^{(k)} \|_{\B^{\al_1,\mu_1}_{p,1}} \lesssim \|1-\ph_m\|_{\B^{\ov \al,\ov \mu}_{p,1}} ,
$$
for some $\ov \al$ and $\ov \mu > 0$ depending only on $\al_1, \al_2, \mu_1$ and $\mu_2$ (where we used the fact that $\| f^{(m)} \|_{\B^{\al_1,\mu_1}_{p,1}}$ is uniformly bounded in the last step). By Lemma~\ref{l:compact-support}, it follows that the sequence $(f^{(m)})$ is a Cauchy sequence in $\B^{\al_2,\mu_2}_{p,1}$, and thus converges to some $f \in \B^{\al_2,\mu_2}_{p,1}$. 
We have the diagram
\begin{displaymath}
\label{diagram}
\begin{array}{cccc}
\ph_m f_n & \xrightarrow[n \to \infty]{} &  f^{(m)} & \\
\downarrow &  & \downarrow & (m \to \infty) \\
f_n & \xrightarrow[n \to \infty]{} &  f, &
\end{array}
\end{displaymath}
where every convergence holds in $\B^{\al_2,\mu_2}_{p,1}$, and the bottom horizontal arrow is the convergence we want to show. This convergence is proved by observing that the convergence of $\ph_m f_n$ to $f_n$ as $m$ tends to infinity is uniform in $n$. Indeed, it suffices to observe that
$$
\| f_n - \ph_m f_n \|_{\B^{\al_2,\mu_2}_{p,1}} \lesssim \|1-\ph_m\|_{\B^{\ov \al,\ov \mu}_{p,1}} \, \|f_n \|_{\B^{\al_1,\mu_1}_{p,1}} \lesssim  \|1-\ph_m\|_{\B^{\ov \al,\ov \mu}_{p,1}} ,
$$
and to use again the fact that $\|1-\ph_m\|_{\B^{\ov \al,\ov \mu}_{p,1}} \to 0$ as $m$ tends to infinity.
\end{proof}

\section{Other scales of weights}
\label{s:DifferentWeights}
We now indicate how the above results carry over to polynomial weights, and to flat weights on cubes for periodic functions. Throughout the paper, when we refer to a result of Section~\ref{s:besov} in the context of polynomial weights or in the periodic context, it is understood that the result is adapted according to the rules we now explain. In the previous sections, we were careful to give estimates that hold uniformly over $\mu \le \mu_0$, for some given $\mu_0$. Such uniformity will only be used for exponential weights, so the reader may leave this aspect aside.

\subsection{Polynomial weights}
For $\si \in \R$, we let
\begin{equation}
\label{e:def:wh-w}
\wh w_\si (x) = |x|_*^{-\si},
\end{equation}
and write $\hL^p$ for the space $L^p(\R^d, \wh w_\si(x) \d x)$. From now on, we will always assume that $\si \ge 0$. It is clear that 
$$
\frac{\wh w_\si (x+y)}{\wh w_\si (y)} \lesssim |x|_*^\si = \wh w_{-\si},
$$
and $\wh w_{-\si}$ is dominated by $w_{-1}$. Moreover, the weight $\wh w_\si$ is integrable if and only if $\si > d$. In Section~\ref{s:ineq}, there is only one place where we assumed the weight to be integrable (that is, $\mu \neq 0$), namely in \eqref{e:r2}. We simply need to replace the assumption that $\mu \neq 0$ by the assumption that $\si > d$. That is, if $\si > d$ and $p_1 \le p_2$, then there exists $C < \infty$ such that
\begin{equation}
\label{e:r2-poly}
\| f \|_{\hB^{\al,\si}_{p_1,q}} \le C \| f \|_{\hB^{\al,\si}_{p_2,q}},
\end{equation}
which is the analogue of \eqref{e:r2}. All the other estimates of Section~\ref{s:ineq} are valid for every $\si \ge 0$ if one replaces $\Lm^p$ by $\hL^p$ throughout. As an example, the conclusion of Lemma~\ref{l:bernstein} becomes
$$
\supp \hat{f} \subset \lambda B \quad \Rightarrow \quad \|\partial^k f\|_{\hL^p} \le C \,  \lambda^{|k|+d\Ll(\frac{1}{q}- \frac{1}{p}\Rr)} \,  \|f\|_{\wh L^q_{\si q/p}}.
$$
For any $\alpha \in \R$, $\si > 0$, $p,q \in [1,+ \infty]$ and $f \in C^\infty_c$, we define 
\begin{equation}
\label{e:def:hB}
\| f \|_{\hB^{\al,\si}_{p,q}}:=  \Ll\| \Ll( 2^{\al k } \| \dk f \|_{\hL^p} \Rr)_{k \ge -1} \Rr\|_{\ell^q}.
\end{equation}
The Besov space $\hB^{\al,\si}_{p,q}$ is the completion of $C^\infty_c$ with respect to this norm. The results of Section~\ref{s:besov} then follow seamlessly, simply replacing $\Bb$ by $\hB^{\al,\si}_{p,q}$ throughout. As an example, the conclusion of Proposition~\ref{p:embed} becomes 
$$
\|f\|_{\hB^{\al,\si}_{p,q}} \lesssim \|f\|_{\hB^{\al,\si r/p}_{r,q}}.
$$
There is one more place where we used the property that the weight is integrable, namely in Subsection~\ref{ss:comp} on the compactness embedding (in Lemma~\ref{l:compact-support}). Again, we simply need to add the requirement that $\si > d$ to obtain the result for polynomial weights. Explicitly, the compactness embedding becomes:
\begin{prop}[Compact embedding]
\label{p:compact-poly}
Let $\al > \al' \in \R$, $p, q \in [1,\infty]$ with $p \neq \infty$ and $\si' > \si > d$. The embedding $\hB^{\al,\si}_{p,1} \subset \hB^{\al',\si'}_{p,1}$ is compact.
\end{prop} 

\subsection{Flat weights on cubes} 
We will also need to use Besov spaces that are tailored for the periodic setting. One natural option would be to base the definition of these Besov spaces on discrete Fourier series (those adapted to the torus under consideration). We choose a slightly different route here: we prefer to view functions on the torus as periodic functions on the full space, and to use the continuous Fourier transform of these functions instead. The latter approach has two advantages. First, the construction is closer to what was done for the weighted Besov spaces above, so the extension of the previous results to these spaces is more transparent. Secondly, we will build solutions on the full space as limits of solutions on increasingly large tori, so viewing functions on a torus as periodic functions on the full space is more natural for our later purpose.

We say that a function $f : \R^d \to \R$ is $M$-periodic if for every $z \in M\Z^d$, we have $f(z+\cdot)=f$. We write $\td{L}_M^p$ for the $L^p$ space of $M$-periodic functions equipped with the Lebesgue measure on $[-M/2,M/2]^d$, and $\|\cdot\|_{\tL^p}$ for the associated norm. For every $M$-periodic function $f$, we define
$$
\|f\|_{\tBb} := \Ll\|  \Ll( 2^{\al k} \| \dk f \|_{\td{L}_M^p}  \Rr)_{k \ge -1}  \Rr\|_{\ell^q},
$$
and let $\tBb$ be the completion of the set of $M$-periodic functions in $C^\infty$ with respect to the norm $\|  \cdot  \|_{\tBb}$. We note that if $f,g: \R^d \to \R$ are measurable functions with $f$ being $M$-periodic, then for $r,p,q$ as in Theorem~\ref{t:young}, we have
$$
\|f \star g \|_{\tL^r} \le C \|f\|_{\tL^p} \, \|g\|_{L^q}
$$
(recall that $\|\cdot\|_{L^q}$ denotes the $L^q$ norm with respect to the Lebesgue measure on the full space). From this form of the Young inequality, one can then proceed as before. For instance, the analogue of Lemma~\ref{l:bernstein} reads: if $f$ is an $M$-periodic function with $\supp \hat f \subset \la B$, then $\|\dr^k f\|_{\tL^p} \le C \,  \lambda^{|k| + d\Ll(\frac{1}{q}- \frac{1}{p}\Rr)} \,  \|f\|_{\tL^q}$ (we view $f$ as an element of $\mcl S'$ to make sense of $\hat f \in \mcl S'$). In short, since the weight is now flat, its powers are trivial and the inequalities simplify. 

For every $M$-periodic function $f$, we define
$$
\|f\|_{\tBb} := \Ll\|  \Ll( 2^{\al k} \| \dk f \|_{\td{L}_M^p}  \Rr)_{k \ge -1}  \Rr\|_{\ell^q},
$$
and let $\tBb$ be the completion of the set of $M$-periodic functions in $C^\infty$ with respect to the norm $\|  \cdot  \|_{\tBb}$. Section~\ref{s:besov} then carries over with obvious changes. For instance, for $\al,\be,p$ and $r$ as in Proposition~\ref{p:embed}, we have $\|f\|_{\tBb} \le C \|f\|_{\tB^{\be,M}_{r,q}}$.

%

%
%
%
%
%
%
%

\section{The stochastic heat equation and its Wick powers}
\label{sec:Probability}
Now that the necessary facts about weighted Besov spaces are established, the remaining sections are devoted to the actual construction of the solution of the stochastic quantisation equation \eqref{e:eqX}. In this section, we provide some probabilistic preliminaries and perform the construction of solution of the stochastic heat equation and its Wick powers in a weighted Besov space. 

We start by recalling the definition of  a space-time white noise $\xi$. Formally, $\xi$ is a centred Gaussian distribution on $\R \times \R^d$ with covariance 
\begin{equation}\label{e:delta}
\E \xi(t,x) \, \xi(t^{\prime}, x^{\prime}) = \delta(t - t^{\prime}) \, \delta^d(x - x^{\prime}) \;,
\end{equation}
where $\delta( \cdot) $ denotes the Dirac delta function and $\delta^d( \cdot) $ is the Dirac delta function over $\R^d$. Of course, \eqref{e:delta} does not make sense as it stands, but it can be made rigorous easily, by integrating against a test function $\ph \colon \R \times \R^d \to \R$ and postulating that
\begin{equation}\label{e:delta1}
\E\; \xi(\ph)^2 = \| \ph \|_{L^2(\R \times \R^d)}^2 \;. 
\end{equation}
\begin{defi}\label{def:whitenoise}
Space-time white noise is a family of centred Gaussian random variables $\{ \xi(\ph),  \, \ph \in L^2(\R \times \R^d)\}$, such that \eqref{e:delta1} holds.
\end{defi}
The existence of space-time white noise in the sense of Definition~\ref{def:whitenoise} on some probability space $(\Omega, \mathcal{F}, \mathbb{P})$ follows immediately from the Kolmogorov extension theorem. Furthermore, the mapping $\ph \mapsto \xi(\ph)$ is automatically linear  on $\Hh := L^2(\R \times \R^d)$ (in the sense that for every $\ph_{1},\ph_2$ we have almost surely $\xi(\ph_1) + \xi(\ph_2) = \xi(\ph_1 +\ph_2)$). Below, we will often use the somewhat formal notation
\begin{align*}
\xi(\ph)=: \int_{\R \times \R^d} \ph (z) \, \xi(\d z) \;,
\end{align*}
although $\xi$ is almost surely not a measure. In particular, for a given $\ph \in L^2(\R \times \R^d)$, the random variable $\xi(\ph)$ is only defined outside of a set of measure zero, and a priori this set depends on the choice of $\ph$. 

We will need a periodised (in space) white noise, defined on the same probability space. For every $M>0$ and $\ph \in L^2(\R \times \R^d)$, we set 
\begin{align}\label{e:xiM}
\xi_M(\ph) = \int_{\R \times \R^d} \ph(z) \, \xi_M(\d z)  := \xi(\ph_M )\,\; ,
\end{align}
where 
\begin{align}\label{e:phiM}
\ph_M(t,x) =  \mathbf{1}_{[-\frac{M}{2} ,\frac{M}{2}]^d}(x) \sum_{y \in M\Z^d} \ph(t,x +  y) \;.
\end{align}
For every $\ph \in L^2(\R \times \R^d)$, we have
\begin{align*}
\| \ph_M  \|_{L^2(\R \times \R^d)} \leq \| \ph \|_{L^2(\R \times \R^d)} \;,
\end{align*}
so that \eqref{e:xiM} makes sense. Furthermore, for every $\ph$ with support contained in $\R \times [\frac{M}{2}, \frac{M}{2} ]^d$, we have $\xi(\ph) = \xi_M(\ph)$ almost surely. 

From now on, we assume that $\Ff$ is the completion of the sigma-algebra generated by $\{ \xi(\ph) \colon \ph \in \Hh\}$. Under this assumption, it is well-known that we have the orthogonal decomposition
\begin{equation*}
 L^2(\Omega, \Ff, \mathbb{P}) = \bigoplus_{k \geq 0} \Hhk \;,
 \end{equation*}
where $\Hhk$ denotes the $k$-th homogeneous Wiener chaos. More precisely, $\Hh^{(0)} = \R$, $\Hh^{(1)} =  \{ \xi(\ph) \colon \ph \in \Hh\}$, and for every $k \geq 2$, we have
\begin{equation*}
\Hhk =\Big\{ I_k(\ph) \colon \, \ph \in \Hh^{\otimes_{\mathrm{sym}} k }    \Big \} \;.
\end{equation*}
Here $\Hh^{\otimes_{\mathrm{sym}} k }  $ denotes the set of square integrable, symmetric kernels $\ph \colon (\R  \times \R^d)^k \to \R$, and $I_k$ denotes the $k$-fold iterated stochastic integral (see \cite[Sec 1.1.1 and 1.1.2]{Nualart}). Below we will often use the notation
\begin{align*}
I_k(\ph) = \int_{(\R \times \R^d)^k} \ph (z_1, \ldots, z_k) \, \xi( \d z_1) \, \ldots \,\xi(\d z_k) \;.
\end{align*}
Let us recall that for every kernel $\ph \in  \Hh^{\otimes_{\mathrm{sym}} k }$, we have the following isometry
\begin{align}
\E \Big( \int_{(\R \times \R^d)^k} & \ph (z_1, \ldots, z_k) \, \xi( \d z_1) \, \ldots \xi(\d z_k) \Big)^2 \notag\\
& = k! \int_{(\R \times \R^d)^k} \ph(z_1, \ldots, z_k)^2 \, \d z_1 \, \ldots \d z_k \;.\label{e:Isometry}
\end{align}
We also recall Nelson's estimate, which states that for every $k \geq 1$ and $p \geq 2$, there exists a $C_{k,p}$ such that for every $X \in \Hh^{(k)}$,
\begin{align}\label{e:Nelson}
\big( \E|X|^p \big)^{\frac{2}{p}} \leq  C_{k,p} \, \E X^2 \;. 
\end{align}

Below we will also need the following \textit{Kolmogorov lemma}. Recall that  according to \eqref{e:def:etak}, for $k \geq -1$, we have $\eta_k = \F^{-1} \chi_k$.  
\begin{lem}\label{le:Kolmogorov}
Let $(t,\ph) \mapsto \cZ(t, \ph)$ be a map from $\R \times L^2(\R^d)\to  L^2(\Omega, \mathcal{F},\P)$ which is linear and continuous in $\ph$. Assume  that for some  $p >1$, $\alpha \in \R$ and $\kappa> \frac{1}{p}$ and all $T >0$, there exists a function $K_T \in L^\infty(\R^d)$, such that for all $k \geq -1$, $x \in \R^d$ and $s,t \in [-T,T]$,
\begin{align}
\E |\cZ(t, \eta_k(\cdot - x))|^p &\leq K_T(x)^p  \, 2^{-k\alpha p} \label{e:Kolmog0}\\
\E |\cZ(t, \eta_k(\cdot - x))- \cZ(s, \eta_k(\cdot - x))|^p &\leq K_T(x)^p \, 2^{-k(\alpha-\ka) p} |t-s|^{\ka p}\;. \label{e:Kolmo2}
\end{align}
Then there exists a random distribution $\tZ$ which takes values in $C(\R, \hB_{p,p}^{\alpha', \si} )$ for any $\alpha'<  \alpha - \ka$ and $\si>2$,  and which satisfies for all $t \in \R$ and $\ph \in \mathcal{S}(\R^d)$ 
\begin{align}\label{e:Kolmog2}
\cZ(t, \ph) = ( \tZ(t), \ph) \qquad \text{almost surely.} 
\end{align}
Furthermore, for every $T>0$, we have
\begin{align}\label{e:Kolmog3}
\E \sup_{-T \leq t \leq T} \| \tZ(t, \cdot) \|_{\hB_{p,p}^{\alpha', \si}}^p \leq C(T, \alpha, \alpha',p) \int_{\R^d} K_T(x)^p \, \wh w_\si(x)  \d x  \;.
\end{align}
\end{lem}
We stated this lemma in the case of polynomial weights, but the same statement holds if $\wh w_\si$ is replaced by another integrable weight, e.g. stretched exponential. In particular, the following  modification of Lemma~\ref{le:Kolmogorov} follows by the same method, and we omit the proof.
\begin{lem}\label{le:KolmogorovBis}
Let $\cZ$ satisfy the assumptions \eqref{e:Kolmog0} and \eqref{e:Kolmo2} of Lemma~\ref{le:Kolmogorov}.
If in addition, $\cZ$ has the property that for some fixed $M>0$ and all $t\in \R$, $\ph \in \mathcal{S}( \R^d)$ and $z \in M\Z^d$,
\begin{align}
\cZ(t, \ph) = \cZ(t, \ph( \cdot - z)) \qquad \text{almost surely,} \label{e:conditionPeriodic}
\end{align}
then $\tZ$ takes values in the Besov space  $\tilde{B}^{\alpha',M}_{p,p}$ of periodic functions and
\begin{align}\label{e:Kolmog3A}
\E \sup_{-T \leq t \leq T} \| \tZ(t) \|_{\tilde{B}_{p,p}^{\alpha', M}}^p \leq C(T, \alpha, \alpha',p) \int_{[-\frac{M}{2}, \frac{M}{2}]^d} K_T(x)^p \,  \d x  \;.
\end{align}
\end{lem}

\begin{proof}[Proof of Lemma~\ref{le:Kolmogorov}]
For fixed $k \geq -1$ and $t \in \R$, we start by defining a random function $\cZ_k(t, \cdot)$ which will play the role of $\delta_k \tZ(t, \cdot)$. The natural definition would be 
\begin{align}\label{e:Kolm1}
 \cZ_k(t,x) = \cZ(t, \eta_k(\cdot - x))\;.
\end{align}
Unfortunately, $\cZ(t, \eta_k(\cdot -x))$ is only defined outside of a set of measure zero which depends on $x$  (and $t$), so that it is not clear if \eqref{e:Kolm1} can be made sense of for all $x$ simultaneously. We circumvent this problem by the following trick (which is motivated by the Shannon sampling theorem \cite[Section~1.4]{Meyer}). Set $\eps_k = \frac{1}{4}\frac{3 \pi}{2^{k+3}}$, and let $\tc_k \in \mathcal{S}(\R^d)$ be real valued, symmetric, such that $\tc_k$ is constantly equal to $1$  on the annulus $B(0, 2^k \frac{8}{3})\setminus B(0,2^k \frac{3}{4}) $, and such $\tc_k$ vanishes outside the annulus $ B(0,2^k \frac{16}{3}) \setminus B(0,2^k \frac{3}{8}) $ (for $k = -1$ we assume that $\tilde{\chi}_{-1}$ is $1$ on the ball $B(0, \frac{4}{3})$ and vanishes outside of $B(0, \frac{8}{3})$). Finally, define $\te_k = \F^{-1} \tc_k$.

For $x \in \eps_k \Z^d$, we define $ \cZ_k(t,x)$ as in \eqref{e:Kolm1}, and for $x \in \R^d \setminus \eps_k \Z^d$, we set 
\begin{align}\label{e:Kolm2}
\cZ_k(t,x) = \sum_{y \in \eps_k \Z^d} \eps_k^d \cZ_k(t,y) \,\te_k(x-y) \;.
\end{align}
As $\te_k$ is a Schwartz function and the assumption \eqref{e:Kolmog0} implies that e.g. 
\begin{equation*}
\sum_{y \in \eps_k \Z^d} \eps_k^d (1+|y|)^{-d-1} \cZ_k(t,y) < \infty
\end{equation*}
 almost surely, it follows that \eqref{e:Kolm2} defines a smooth random function on a set of measure $1$. Furthermore, we have for $k \geq 0$
\begin{align}\label{e:Kolm3}
\supp \Ff \big(\cZ_k(t, \cdot)\big) \subseteq B\Big(0,2^k \frac{16}{3}\Big) \setminus B\Big(0,2^k \frac{3}{8}\Big) \;,
\end{align} 
and $\supp \Ff \big( \cZ_{-1}(t, \cdot) \big) \subseteq B(0, \frac{8}{3})$.
Next, we set $\tZ(t) = \sum_{k \geq -1} \cZ_k(t, \cdot)$. Equation \eqref{e:Kolm3} implies that this sum converges almost surely when tested against a test function with smooth and compactly supported Fourier transform. 
In the estimates below, we show that $\tZ$ is a well-defined random distribution.

We claim that for fixed $x \in \R^d$,  $\tZ$ satisfies \eqref{e:Kolmog2}, at least if $\ph$ has Fourier transform in $C_c^{\infty}$. Indeed, we have for such   $\ph$
\begin{align*}
( \tZ(t), \ph) &= \sum_{k \geq -1} \big( \cZ_k(t, \cdot),  \ph \big)= \sum_{k \geq -1}\sum_{y \in \eps_k \Z^d} \eps_k^d \int_{\R^d} \cZ_k(t,y) \,\te_k(x-y) \ph(x) \, \d x \;.
\end{align*}
As a next step, we use the fact that both $\te_k$ and $\eta_k$ have a Fourier transform supported in the annulus $ B(0,2^k \frac{16}{3}) \setminus B(0,2^k \frac{3}{8}) $ which implies for our choice of $\eps_k$ that 
\begin{align*}
\sum_{y \in \eps_k \; \Z^d}\eps_k^d \eta_k(y - \cdot) \, \big(\te_k \ast \ph (y)\big) = \int_{\R^d} \eta_k (y -\cdot)\, \big( \te_k \ast \ph(y) \big) \, \d y = \delta_k \ph(\cdot) \,,
\end{align*} 
where the convergence takes place in $L^2(\R^d)$. Hence, by continuity of $\cZ(t, \cdot)$ we can conclude that
\begin{align*}
( \tZ(t), \ph) = \sum_{k \geq -1} \cZ(t, \delta_k \ph) = \cZ(t, \ph) \;.
\end{align*}
In particular, this implies that for every $x$, \eqref{e:Kolm1} holds almost surely.

Now we can estimate for $t \in [-T,T]$
\begin{align*}
\E \|\cZ_k(t, \cdot)  \|_{\wh{L}^p_\si}^p &=
  \int_{\R^d}  \,  \E | \cZ_k(t, x)|^p \; \wh w_\si(x) \d x \\
& = \int_{\R^d}  \,\E| \cZ(t,\eta_k (\cdot -x)) |^p \; \wh w_\si(x)  \d x \leq  2^{-k \alpha p} \int_{\R^d} K_T(x)^p \,\wh w_\si(x)  \d x \;.
\end{align*}
This bound in conjunction with \eqref{e:Kolm3} and the series criterion, Lemma~\ref{l:crit}, implies that for $\alpha' < \alpha$
\begin{align*}
\E\| \tZ(t,\cdot) \|_{\hB_{p,p}^{\alpha',\mu}}^p = \sum_{k \geq -1} 2^{k\alpha^{\prime} p}\, \E \|\cZ_k(t, \cdot)  \|_{\wh L^p_\si}^p \leq C(\alpha, \alpha^{\prime},p)  \int_{\R^d} K_T(x)^p \,\wh w_\si(x) \, \d x \;.
\end{align*} 
In particular, almost surely $\tZ$ is indeed a well-defined distribution in $\hB_{p,p}^{\alpha',\mu}$. Furthermore, the assumption that $\ph$ has compactly supported and smooth Fourier transform can be removed from \eqref{e:Kolmog2} and can be replaced by $\ph \in \mathcal{S}$.

In the same way, we can bound for $s,t \in [-T,T]$
\begin{align*}
\E \|\cZ_k(t, \cdot) - \cZ_k(s, \cdot)  \|_{\wh{L}^p_\si}^p 
& = \int_{\R^d} \,\E| \cZ(t,\eta_k (\cdot -x)) - \cZ(s,\eta_k (\cdot -x))  |^p \,\wh w_\si(x)  \d x \\
&\leq  |t-s|^{\ka p}2^{-k (\alpha - \ka) p} \int_{\R^d} K_T(x)^p \, \wh w_\si(x)  \d x \;.
\end{align*}
So after summing over $k$, the  (usual) Kolmogorov criterion permits to pass to a modification of the process $t \mapsto \tZ(t)$ which satisfies \eqref{e:Kolmog2} and \eqref{e:Kolmog3}.
\end{proof}

With these preliminary results in hand, we can start the construction of the solution of the stochastic heat equation and its Wick powers in weighted Besov spaces, in the case when $d = 2$. The space-time white noise $\xi$  does not satisfy the conditions of Lemma~\ref{le:Kolmogorov}, and indeed it can only be realised as a space-time distribution, and not as a continuous function in time taking values in a space of distributions. 

However, we do not need to perform this construction. Instead, we start directly by analysing the properties of the solution of the stochastic heat equation
\begin{equation}
\label{e:eqZ2}
\Ll\{
\begin{array}{ll}
\dr_t \cZ = \Delta \cZ  + \xi, \qquad \text{on } \R_+ \times \R^2, \\
\cZ(0,\cdot) = 0.
\end{array}
\Rr.
\end{equation}
There are many equivalent ways to interpret this equation. We choose to simply postulate Duhamel's principle, and to define for every $\ph \in L^2(\R^2)$ and every $t \geq 0$
\begin{align}\label{e:defZ}
\cZ(t,\ph) &= \int_{0}^t  \!\! \int_{\R^2} \big(  \ph,  K(t-r, \cdot-y) \big)  \, \xi(\d r, \d y) \;.
\end{align}
Here $K$ denotes the standard Gaussian heat kernel, i.e. for $t >0$ and $x\in \R^2$
\begin{align*}
K(t,x) = \frac{1}{4 \pi } \exp\Big(- \frac{|x|^2}{4t} \Big) \;.
\end{align*}
We also define $\ZM$, the solution of the equation with periodised noise, as in \eqref{e:defZ} with $\xi$ replaced by $\xi_M$.
Furthermore, for any $n \in \N$, we define the Wick powers
\begin{align}
&\cZ^{: n :} (t,\ph)  \notag\\
&=  \int_{([0,t]\times \R^2)^{n}} \big( \ph, \prod_{j=1}^{n} K(t-r_j, \cdot-y_j) \big) \, \xi(\d r_1, \d y_1) \, \ldots \, \xi(\d r_{n}, \d y_{n}) \;, \label{e:Wick}
\end{align}
and as before $\ZM^{: n:}$ is defined by replacing every occurrence of $\xi$ in \eqref{e:Wick} by $\xi_M$.

\begin{thm}\label{thm:BoundsOnZ}
For every  $M \geq 1$ and every integer $n \geq 1$ , there exist modifications $\tZ^{: n :}$ and $\tZM^{: n :}$ (in the sense of Lemma~\ref{le:Kolmogorov}) of $\cZ^{: n : }$ and $\ZM^{: n : }$. 

For every $T>0$, $\al>0$ and  $p>\frac{2}{\al}$,  there exists a constant $C=C(T,p,\alpha)$ such that for all $M \geq 1$ and $\si>2$,
\begin{align}
\E \sup_{0 \leq t \leq T} \|\tZ^{: n :}(t) \|_{\hB_{p,p}^{-\alpha, \si} }^p &\leq C \int_{\R^2} \wh w_\si(x) \, \d x\;,  \label{e:ZZZ1}\\
\E \sup_{0 \leq t \leq T} \|\tZM^{: n :}(t) \|_{\hB_{p,p}^{-\alpha, \si} }^p &\leq C \int_{\R^2} \wh w_\si(x) \, \d x  \,. \label{e:ZZZ2}
\end{align}
Furthermore, $\tZM$ is is $M$-periodic and 
\begin{align}
\E \sup_{0 \leq t \leq T} \|\tZM^{: n :}(t) \|_{\tilde B_{p,p}^{-\alpha, M} }^p &\leq C M^2 \;.\label{e:ZZZ3}
\end{align}
Finally, we have
\begin{align}
\E \sup_{0 \leq t \leq T} \|\tZ^{: n:}(t) - \tZM^{: n :}(t) \|_{\wh B_{p,p}^{-\alpha, \si} }^p &\leq C M^{2-\si}  \;. \label{e:ZZZ4}
\end{align}
\end{thm}
Throughout the proof of Theorem~\ref{thm:BoundsOnZ} (including the two lemmas at the end), we use the following conventions: for $x>0$, we set  $\log_+(x) = \log(x) \vee 0$. Furthermore, we set
$$
|x|_M = \Ll|
\begin{array}{ll}
\inf \{|x + y|;\; y \in M\Z^2 \} & \text{if } M < \infty, \\
|x| & \text{if } M = \infty.
\end{array}
\right.
$$
Sometimes it is convenient to write $\cZ_{\infty}^{: n :} = \cZ^{: n :}$.
\begin{proof}
Equation \eqref{e:Isometry} and polarisation show that for any  $n$,   $\ph_1, \ph_2 \in \Ss(\R^2)$,  $t_1, t_2 >0$, and for any $M\in [1,\infty]$, we have
\begin{align*}
\E ( \ZM^{: n :}(t_1)&, \ph_1) ( \ZM^{: n :}(t_2), \ph_2)  =\\
& n!  \int_{\R^2} \int_{\R^2} \ph(x_1)   \ph(x_2) \big( \mathscr{K}_M(t_1,t_2;x_1 - x_2) \big)^n \; \d x_1\, \d x_2 \;.
\end{align*}
For $M=\infty$, the generalised covariance in the last expression is given by 
\begin{align}
\mathscr{K}_{\infty}(t_1,t_2;x_1 -x_2) &=\int_{0}^{t_1 \wedge t_2}\int_{\R^2}  K(t_1-r, x_1 -z) \,  K(t_2-r, x_2-z) \; \d r \, \d z\notag \\
& = \frac{1}{8\pi} \int_{|t_1-t_2 |}^{t_1 + t_2 } \frac{1}{\ell}  \exp\Big(-\frac{ |x_1 -x_2|^2}{4\ell}\Big) \; \d \ell \;,\label{e:defKinfty}
\end{align}
while for finite $0<M<\infty$, we get
\begin{align}
\mathscr{K}_{M}(t_1,t_2;x_1 -x_2) 
& = \frac{1}{8\pi} \int_{|t_1-t_2 |}^{t_1 + t_2 } \frac{1}{\ell} \sum_{y \in M\Z^2 } \exp\Big(-\frac{ |x_1 -x_2 - y |^2}{4\ell}\Big) \; \d \ell \;.\label{e:defKM}
\end{align}
 Pointwise bounds on the kernels $\mathscr{K}_M$ are derived in Lemma~\ref{le:GenCov} below.

From the bounds provided in Lemma~\ref{le:GenCov} we get for any $k \geq -1$ any $x \in \R^d$, any $t \leq T$ and for $1 \leq M \leq \infty$
\begin{align}
\E |&\ZM^{: n :}(t,\eta_k(\cdot - x))|^2 = \E |\ZM^{: n :}(t,\eta_k)|^2 \notag \\ 
&=n! \int_{\R^2}\int_{\R^2} \eta_k(x_1) \eta_k(x_2)\, \big( \mathscr{K}_M(t,t, x_1 - x_2) \big)^n \, \d x_1 \, \d x_2 \notag \\
&\stackrel{\eqref{e:Kbound1}}{\ls_{T,n}}    \int_{\R^2}\int_{\R^2} \eta_k(x_1) \eta_k(x_2) \big( 1 +\log_+(|x_1 -x_2|_M^{-1})\big)^n \, \d x_1 \, \d x_2 \notag\;.
\end{align}
The integral involving the kernel $1$ poses no problems, because the $\eta_k$ are uniformly bounded in $L^1$.   We  split the integral over the logarithmic kernel  into an integral over $\{ |x|_1, |x_2| \leq \frac{M}{4}\}$ and an integral over the remaining part of $\R^2 \times \R^2$. For the first integral, we get for $k \geq 0$
\begin{align}
 \int_{|x_1| \leq \frac{M}{4}} & \int_{|x_2| \leq \frac{M}{4}} \eta_k(x_1) \eta_k(x_2) \log_+(|x_1 -x_2|_M^{-1})^n \, \d x_1 \, \d x_2\notag \\
&\ls_{n}   \int_{\R^2}  \int_{\R^2}\eta_0(x_1) \eta_0(x_2) \big( \log_+(|x_1 -x_2|^{-1}) + k \log(2) \big)^n \, \d x_1 \, \d x_2 \notag\\
& \ls_{n} (1+ k^n) \;.\label{e:CalcLogScaling}
\end{align}
Here we have used the fact that $\eta_{k}(\cdot) = 2^{2k}\eta_0(2^k \cdot)$. The integral in the second line converges because $\eta_0$ is a Schwartz function. 

The integral over $\R^2 \times \R^2 \setminus \{ |x_1|, |x_2| \leq \frac{M}{4} \}$ can then easily be seen to be uniformly bounded in $k$ and $M$ using the decay of the $\eta_k$. More precisely, one uses the fact that 
 for $k \geq 0$ and any $m>2$ we have $\eta_k(x) = 2^{2k} \eta(2^k x) \ls_m \frac{1}{2^{m-2}|x|^m}$.
The integral in the case $k = -1$ can easily be checked  to converge as well. Hence, summarising these calculations and applying Nelson's estimate \eqref{e:Nelson} for any $p \geq 2$, we get uniformly over $t \leq T$, $x \in \R^2$, $M \in [1,\infty]$ and $k \geq -1$
\begin{align}\label{e:Zconclusion1}
\E |&\ZM^{: n :}(t,\eta_k(\cdot - x))|^p \ls_{T,n,p} 1+ |k|^{\frac{np}{2}}\;.
\end{align}

In the same way, for $0 \leq t_1,t_2 \leq T$ we have for $M \in [1,\infty]$
\begin{align*}
\E |&\ZM^{: n :}(t_1,\eta_k(\cdot - x)) - \ZM^{: n :}(t_2, \eta_k(\cdot - x))  |^2 \\
&=n! \int_{\R^2}\int_{\R^2} \eta_k(x_1) \eta_k(x_2)\, \mathscr{Q}_M^n(t_1,t_2, x_1 - x_2) \, \d x_1 \, \d x_2 \;,
\end{align*}
where the kernel $\mathscr{Q}^n$  satisfies for $0 < \la \leq 1$ 
\begin{align*}
\mathscr{Q}^n_{\delta}(t_1,t_2, x)  &= \big( \mathscr{K}_M(t_1,t_1, x) \big)^n + \big( \mathscr{K}_M(t_2,t_2, x) \big)^n - 2 \big( \mathscr{K}_M(t_1,t_2, x) \big)^n \\
 &\stackrel{\eqref{e:Kbound1} \eqref{e:Kbound2}}{\ls_{T,n,\la}} \frac{|t_1 - t_2|^{\la} }{|x|_M^{2\la}} \big(1+ \log_+(|x|_M^{-1})^{n-1})\;.
\end{align*}
Then performing a similar calculation to \eqref{e:CalcLogScaling} using the fact that $\eta_{k}(\cdot) = 2^{2k}\eta_0(2^k \cdot)$ and that $\eta_0$ is a Schwartz function, as well as Nelson's estimate, we get
\begin{align}
\E |&\cZ^{: n :}(t_1, \eta_k(\cdot - x)) - \cZ^{: n :}(t_2, \eta_k(\cdot - x))  |^p \notag\\
 &\qquad \qquad\qquad\qquad \ls_{T,n,\la,p}  |t_1 - t_2|^\frac{\la p}{2} \, 2^{p k \la} \,|k|^{\frac{p (n-1)}{2}} \;. \label{e:Zconclusion2}
\end{align}
The bounds \eqref{e:Zconclusion1} and \eqref{e:Zconclusion2} permit to invoke Lemma~\ref{le:Kolmogorov} (and Lemma~\ref{le:KolmogorovBis}) to conclude that \eqref{e:ZZZ1}, \eqref{e:ZZZ2}, and \eqref{e:ZZZ3} hold.

We proceed to bound the difference $\DM  := \ZM^{: n :} - \cZ^{: n :}$. Unlike the preceding calculation, we need to make use of the decay of the weight in an essential way. Indeed, for $|x| \geq \frac{M}{16}$ we simply bound 

\begin{align}
\E |&\DM(t,\eta_k(\cdot - x)) |^p\notag\\
& \qquad \ls_p  \E |\ZM^{: n :}(t,\eta_k(\cdot - x))|^p + \E|   \cZ^{: n :}(t, \eta_k(\cdot - x))  |^p \stackrel{\eqref{e:Zconclusion1}}{\ls_{T,n,p}} 1+ |k|^{\frac{np}{2}} \;, \label{e:DZM1}
\end{align}
and in the same way
\begin{align}
\E \big|& \DM (t_1,\eta_k(\cdot - x)) - \DM (t_2,\eta_k(\cdot - x))  \big|^p \stackrel{\eqref{e:Zconclusion2}}{\ls_{T,n,\la, p}}   |t_1 - t_2|^\frac{\la p}{2} \, 2^{p k \la} \,|k|^{\frac{p (n-1)}{2}} \;. \label{e:ZZZZ}
\end{align}
Hence, for such $x$ the difference $\DM$ satisfies the bounds required to apply Lemma~\ref{le:Kolmogorov} for a function $K_T$ that does not depend on $x$. If $|x| \leq \frac{M}{16}$ we write
\begin{align*}
\E |& \DM(t,\eta_k(\cdot - x))  |^2\\
&\qquad \qquad =n! \int_{\R^2}\int_{\R^2} \eta_k(x_1) \eta_k(x_2)\, \mathscr{R}_M^n(t; x-x_1, x-x_2) \, \d x_1 \, \d x_2 \;.
\end{align*}
The kernel appearing in this expression is given by 
\begin{align*}
&\mathscr{R}_M^n(t; x_1, x_2) \\
&\big( \mathscr{K}(t,t, x_1 - x_2) \big)^n + \big( \mathscr{K}_M(t,t, x_1 - x_2) \big)^n - 2 \big( \mathscr{K}_{M,\infty}(t; x_1, x_2) \big)^n\;,
\end{align*}
 where 
 \begin{align}
 &\mathscr{K}_{M,\infty}(t,x_1,x_2) \notag\\
&= \sum_{y \in M \Z^2}\int_{0}^{t}\int_{[-\frac{M}{2},\frac{M}{2}]^2}  K(t-r, x_1 - x_2 -z) \,  K(t-r, -z - y) \; \d r \, \d z \;. \label{e:KMI}
 \end{align}
Estimates for the kernel $\mathscr{K}_{M,\infty}$  are collected in Lemma~\ref{le:kernelMInftyBounds}. These bounds yield for $|x-x_1|, |x-x_2| \leq \frac{M}{8}$ and any $m >1$
\begin{align*}
\big| \mathscr{R}_M^n(t; x-x_1, x-x_2) \big| \ls_{T,m}  (1 + \log_+(|x_1 -x_2|^{-1})^{n-1}) \frac{1}{M^{2m}}\;.
\end{align*}
We get
\begin{align*}
\int_{|x_1| \leq \frac{M}{16}}&\int_{|x_2 \leq \frac{M}{16}} \eta_k(x_1) \eta_k(x_2)\, \mathscr{R}_M^n(t; x-x_1, x-x_2) \, \d x_1 \, \d x_2\\
& \ls_{T,m} \frac{1}{M^{2m}} \int_{\R^2} \int_{\R^2} \eta_k(x_1) \eta_k(x_2)  (1 + \log_+(|x_1 -x_2|^{-1})^{n-1}) \; \d x_1 \, \d x_2 \\
&\stackrel{\eqref{e:CalcLogScaling}}{\ls_{T,m}} \frac{1}{M^{2m}} (1+ |k|^{n-1})\;.
\end{align*}
If one of the $x_i$ satisfies $|x_i| \geq \frac{1}{M^{2m}}$, we can use the fact that the $\eta_k$ are rescaled Schwartz functions to get the same polynomial decay of arbitrary order  in $M$ also in the integral over the rest of $\R^2$. Applying Nelson's estimate once more and merging with \eqref{e:DZM1} we get 
\begin{align*}
\E |&\DM (t,\eta_k(\cdot - x)) |^p \ls_{T,m,p} (1+ |k|^{\frac{p n}{2}}) \; K_T(x) \;,
\end{align*}
where $K_T(x) = 1$ for $|x| \geq \frac{M}{16}$ and $K_T(x) \ls M^{-2m}$ else. Interpolating this bound with \eqref{e:ZZZZ}, it is easy to obtain a similar bound for the time increments. Hence we can conclude by observing that for any $m$ large enough,
\begin{align*}
\int_{\R^2} K_T(x) \; \wh w_\sigma(x) dx \ls_m M^{-2m +2} + \int_{|x| \geq \frac{M}{16}} \wh w_{\si} (x) dx \ls_{\si} M^{2 - \si} \;.
\end{align*}
\end{proof}

\begin{lem}\label{le:GenCov}
Let $\mathscr{K}_M$ be defined by \eqref{e:defKinfty} and \eqref{e:defKM}.
Then for every $T>0$ we get uniformly for $x \in \R^2$ and $0 \leq t_1,t_2 \leq T$ and $M \in [1,\infty]$
\begin{align}
  \mathscr{K}_M(t_1,t_2;x) &\leq C(T) \big(1 +  \log_+  (|x|^{-1}_M  \big)\big)\;. \label{e:Kbound1}
\end{align} 
Furthermore, for any $\lambda \in (0,1]$ and $0 \leq t_1,t_2 \leq T$
\begin{align}
\big|&  \mathscr{K}_M(t_1,t_1;x) - \mathscr{K}_M(t_1,t_2;x)  \big|  \leq C(T,\la) \frac{ |t_1 -t_2|^\la}{|x|_M^{2\la}  } \;.\label{e:Kbound2}
 \end{align}
\end{lem}
\begin{proof}
We first show \eqref{e:Kbound1}.  For $M<\infty$, we can assume without loss of generality that the infimum $\inf\{|x+y| \colon y \in M\Z^2 \}$ is realised for $y =0$ which implies that $|x| \leq \frac{M}{\sqrt{2}}$.

Then the term corresponding to $y =0$ in \eqref{e:defKM} can be bounded by
\begin{align}
\frac{1}{8\pi}& \int_{|t_1-t_2 |}^{t_1 + t_2 } \frac{1}{\ell}  \exp\Big(-\frac{ |x|^2}{4\ell}\Big) \; \d \ell\notag\\
 & \ls \int_0^{2T}  \frac{1}{\ell}  \exp\Big(-\frac{ |x|^2}{4\ell}\Big) \; \d \ell 
 \ls \int_0^1  \frac{1}{\ell}  \exp\Big(-\frac{ 1}{4\ell}\Big) \; d \ell +  \int_1^{1 \vee \frac{2T}{|x|^2}}  \frac{1}{\ell}  \; \d \ell \notag\\
 & \ls_{T} 
 1 + \log_+ \big( |x|^{-1} \big) \,.\notag
\end{align}
This calculation already shows the desired bound \eqref{e:Kbound1} in the case $M=\infty$. For $M<\infty$ for $z = |x+ y|$ we use the bound
\begin{align*}
\frac{1}{8\pi}& \int_{|t_1-t_2 |}^{t_1 + t_2 } \frac{1}{\ell}  \exp\Big(-\frac{ |z|^2}{4\ell}\Big) \; \d \ell
  \ls_m \int_0^{2T}  \frac{1}{\ell} \Big( \frac{\ell}{|z|^2} \Big)^m  \; \d \ell  \ls_{m} \frac{T^m}{|z|^{2m}}\;,
\end{align*}
valid for every $m >0$. To bound  the sum over $y \neq 0$ in \eqref{e:defKM}  we choose an $m>1$ and obtain  
\begin{align}
\frac{1}{8\pi} &\int_{|t_1-t_2 |}^{t_1 + t_2 } \frac{1}{\ell} \sum_{y \in M\Z^2 \setminus \{0\} } \exp\Big(-\frac{ |x - y |^2}{4\ell}\Big) \; \d \ell \notag\\
&\ls_{m,T} \sum_{y \in M\Z^2 \setminus \{0\}} \frac{1}{|x+y|^{2m}} \ls_{m,T} \sum_{y \in M\Z^2 \setminus \{0\}} \frac{1}{|y|^{2m}}  \ls_{m,T} M^{-{2m}} \;.\notag
\end{align}
Here we have made use of the fact that $|x| \leq \frac{M}{\sqrt{2}}$ implies that $|x+y| \geq (1- \frac{1}{\sqrt{2}} ) |y|$ for all $y \in M \Z^2$. Hence, \eqref{e:Kbound1} is established. 

To see \eqref{e:Kbound2}  we can again assume without loss of generality that in the case $M<\infty$ we have $|x|_M =|x|$ and hence  $|x| \leq \frac{1}{\sqrt{2}}M$.
We get  for any $M \in [1,\infty)$
\begin{align}\label{e:KRechnung}
\big| \mathscr{K}_M(t_1,t_1;x)  - \mathscr{K}_M(t_1,t_2;x) \big| &\ls \int_{I_1 \cup I_2} \frac{1}{\ell} \sum_{y \in M\Z^2}\exp\Big(\frac{|x+y|^2}{4\ell} \Big) \d \ell \;.
\end{align}
Here the intervals $I_1$ and $I_2$ are given by
\begin{align*}
I_1 =& [0, |t_1 - t_2 | ] \quad \text{and} \quad
I_2 = [t_1 + t_2 , 2 t_1 ]_+ \; 
\end{align*}
with the convention $[a,b]_+ = [a,b]$ if $a \leq b$ and $[a,b]_+ = [b,a]$ else. In particular, we have $|I_1|= |I_2| = |t_1-t_2|$ and both intervals are contained in $[0,2T]$. In the case $M=\infty$ the sum in  \eqref{e:KRechnung} has to be replaced by the single term corresponding to $y=0$.

For each term in \eqref{e:KRechnung} we get, setting $z = x+y$   and choosing $m >0$ appropriately
\begin{align*}
\int_{I_1 \cup I_2} \frac{1}{\ell} \exp\Big(\frac{|z|^2}{4\ell} \Big) \d \ell \ls_m \int_{I_1 \cup I_2} \frac{1}{\ell}   \Big( \frac{\ell}{|z|^2} \Big)^m \, \d \ell \ls_{m,T}  \frac{|t_1 -t_2|^{m \wedge 1}}{|z|^{2m}}  \;.
\end{align*}
For the term involving $y=0$ we choose $m = \la \in (0,1]$ in this bound. As above this already establishes \eqref{e:Kbound2} in the case $M=\infty$. For $M<\infty$ we use an arbitrary $m >1$ to bound the terms for $y \neq 0$. We then obtain, using as above that $|x +y | \geq (1 - \frac{1}{\sqrt{2}}) |y|$
\begin{align*}
\sum_{y \in M \Z^2 \setminus \{0 \}}  \int_{I_1 \cup I_2} \frac{1}{\ell} \exp\Big(\frac{|x+y|^2}{2\ell} \Big) \d \ell &\ls _{m,T} |t_1 - t_2|  \sum_{y \in M \Z^2 \setminus \{0 \}}  \frac{1}{|y|^{2m}} \\
& \ls_{m,T}  |t_1 - t_2| M^{-2m}  \;.
\end{align*}
So, \eqref{e:Kbound2} follows as well. 
\end{proof}
\begin{lem}\label{le:kernelMInftyBounds}
Let $\mathscr{K}_{M,\infty}$ be defined by \eqref{e:KMI} and $\mathscr{K}_\infty$, $\mathscr{K}_M$ by \eqref{e:defKinfty}, \eqref{e:defKM}. For $0 \leq t \leq T$,  $1 \leq M < \infty$ and $x_1,x_2 \in \R^2$ with $|x_1|, |x_2| \leq \frac{M}{8}$ we have  
\begin{align}
\mathscr{K}_{M,\infty}(t; x_1, x_2)  \leq C(T) (1 + \log_+(|x_1 -x_2|^{-1}) \;. \label{e:KMIBoundA}
\end{align}
Furthermore, under the same assumptions on $t,M,x_1,x_2$ we have for every $m >1$
\begin{align}
\big| \mathscr{K}_{M,\infty}(t; x_1, x_2) - \mathscr{K}_\infty(t,t;  x_1- x_2) \big| \leq C(T,m) \frac{1}{M^{2m}} \;, \label{e:KMIBoundB}\\
\big| \mathscr{K}_{M,\infty}(t; x_1, x_2) - \mathscr{K}_M(t,t;  x_1- x_2) \big| \leq C(T,m) \frac{1}{M^{2m}}\;. \label{e:KMIBoundC}
\end{align}
\end{lem}
\begin{proof}
We start by establishing \eqref{e:KMIBoundA}. The  term in \eqref{e:KMI} corresponding to $y=0$  can be bounded easily
\begin{eqnarray*}
\int_{0}^{t}\int_{[-\frac{M}{2},\frac{M}{2}]^2}  K(t-r, x_1 - x_2 -z) \,  K(t-r, -z ) \; \d r \, \d z \\
\leq \mathscr{K}_\infty(t,t;x_1 - x_2) \stackrel{\eqref{e:Kbound1}}{\ls_T} (1 + \log_+(|x_1 - x_2|^{-1}) \;,
\end{eqnarray*}
where in the last inequality we have used the fact that $|x_1|$, $|x_2| \leq \frac{M}{8}$ implies that $|x_1 - x_2|_M = |x_1 - x_2|$.
For the remaining terms we use the fact that for $z \in [-\frac{M}{2}, \frac{M}{2}]^2$ and $y \in M\Z^2 \setminus \{0 \}$ we have $|z+y| \geq |y| - |z| \geq (1- \frac{1}{\sqrt{2}}) |y|$. We obtain for every $m>1$
\begin{align}
 \sum_{y \in M\Z^2 \setminus\{0\}}K(t-r, -z - y) \ls_m  \sum_{y \in M\Z^2\setminus\{0\}}  \frac{1}{t-r} \Big(\frac{t-r}{|z+y|^2}\Big)^m \ls_{m} \frac{(t-r)^{m-1}}{M^{2m}}\;, \label{e:MalWieder}
\end{align}
which implies that
\begin{align}
 &\sum_{y \in M \Z^2 \setminus \{ 0 \} } \int_{0}^{t}\int_{[-\frac{M}{2},\frac{M}{2}]^2}  K(t-r, x_1 - x_2 -z) \,  K(t-r, -z - y) \; \d r \, \d z  \notag\\
 & \quad \ls_m \frac{1}{M^{2m}}\int_{0}^{t}\int_{\R^2} (t-r)^{m-1}  K(t-r, x_1 - x_2 -z) \,   \; \d r \, \d z \ls_{m,T} \frac{1}{M^{2m}} \;. \label{e:mittenInDerRechnung}
\end{align}
So, \eqref{e:KMIBoundA} follows.

To see \eqref{e:KMIBoundB} we write
\begin{align*}
 &\mathscr{K}_{M,\infty}(t;x_1,x_2) - \mathscr{K}_\infty(t,t; x_1 - x_2) \\
 &=  \sum_{y \in M \Z^2 \setminus \{ 0 \} } \int_{0}^{t}\int_{[-\frac{M}{2},\frac{M}{2}]^2}  K(t-r, x_1 - x_2 -z) \,  K(t-r, -z - y) \; \d r \, \d z \\
&- \int_{0}^{t}\int_{\R^2 \setminus [-\frac{M}{2},\frac{M}{2}]^2}  K(t-r, x_1 - x_2 -z) \,  K(t-r, -z ) \; \d r \, \d z \;. 
\end{align*}
We have already seen above in \eqref{e:mittenInDerRechnung} that the first term on the right hand side is bounded by $\frac{C(m,T)}{M^{2m}}$ for any $m>1$.
For the second term, our assumption $|x_1| , |x_2| \leq \frac{M}{8}$ (which implies that $|x_1 - x_2 | \leq \frac{M}{4}$) enters, because it implies that for any $z \notin [\frac{M}{2} \frac{M}{2}]$ and for $m>1$ we have 
\begin{align*}
 K(t-r, x_1 - x_2 -z) \ls \frac{1}{t-r} \exp\Big(- \frac{M^2}{16(t-r)} \Big) \ls_{m} \frac{1}{t-r}\Big( \frac{t-r}{M^2} \Big)^m \;.
\end{align*} 
Therefore, we can conclude that 
\begin{align*}
 &\int_{0}^{t}\int_{\R^2 \setminus [-\frac{M}{2},\frac{M}{2}]^2}  K(t-r, x_1 - x_2 -z) \,  K(t-r, -z ) \; \d r \, \d z\\
 &  \ls_m  \frac{1}{M^{2m}} \int_{0}^{t}\int_{\R^2} (t-r)^{m-1}   K(t-r, -z ) dr \, dz  \ls_{m,T} \frac{1}{M^{2m}} \;,
\end{align*}
and \eqref{e:KMIBoundB} is established. 

Finally, in a similar way we get
\begin{align}
 &\mathscr{K}_{M,\infty}(t;x_1,x_2) - \mathscr{K}_M(t,t; x_1 - x_2) \notag\\
 &=  \int_{0}^{t}\int_{[-\frac{M}{2},\frac{M}{2}]^2}  (K(t-r, x_1 - x_2 -z) - K_M(t-r, x_1 - x_2 -z)) \notag\\
 & \qquad \qquad \times   K_M(t-r, -z ) \; \d r \, \d z \;. \label{e:erstmalLetzte}
\end{align}
We use once more the fact that $|x_1 + x_2| \leq \frac{M}{4}$  and argue as in \eqref{e:MalWieder} to see that uniformly over  $z \in [-\frac{M}{2},\frac{M}{2}]^2$ we have
\begin{align*}
\big| K(t-r, x_1 - x_2 -z) - K_M(t-r, x_1 - x_2 -z)\big| \ls_m  \frac{1}{t-r}\Big(\frac{t-r}{M}\Big)^m \,.
\end{align*}
So that after integrating out the remaining kernel $K_M$ in \eqref{e:erstmalLetzte} we get the bound \eqref{e:KMIBoundC}.
\end{proof}
It remains to treat the case of  non-zero  initial condition $X_0$ for \eqref{e:eqZ2}. We will need the following lemma only in the case $d=2$, but as it causes no extra effort, we state and prove it for arbitrary spatial dimension.

\begin{lem}\label{l:periodisation1} 
Let $|\al| < 1$, $1 \leq p < \infty $ and $\si >2$. Fix  $X \in \hB_{p,p}^{\alpha, \si} $. For every $M \geq 1$, there exists an $M$-periodic distribution $X_M \in \tilde B_{p,p}^{\alpha, M} $ such that for every test function $\ph$ with compact support contained in 
$B(0, \frac{M}{4})$, we have 
\begin{equation}
\label{e:ppty_centre}
( X_M, \ph) = (X, \ph).
\end{equation}
Furthermore, the $X_M$ are bounded in  $\hB_{p,p}^{\alpha, \si}$ uniformly in $M$ and converge to $X$ in every space $\hB_{p,p}^{\td{\alpha}, \td{\si}}$ for $\td{\alpha}<\alpha$ and $\td{\si}>\si$. 
\end{lem}

\begin{proof}
Let $\ph \in C^\infty_c$ be such that 
$$
\ph = 1 \text{ on } B(0,1/4), \quad \ph = 0 \text{ outside of } B(0,1/3),
$$
and let $\ph_M = \ph(\cdot/M)$. We define
$$
\td{X}_M = \ph_M X.
$$
By Lemma~\ref{l:compact-support} and the multiplicative inequality, $\td{X}_M$ is well-defined as an element of $\hB^{\al,\si}_{p,p}$, and moreover,
$$
\sup_{M \ge 1} \|\td{X}_M\|_{\hB^{\al,\si}_{p,p}} < \infty.
$$
Since $\td{X}_M$ is supported in $B(0,M/3)$, we can define the $M$-periodic distribution
$$
X_M = \sum_{z \in M\Z^d} \td{X}_M(\cdot - z).
$$
This distribution satisfies \eqref{e:ppty_centre}. We now show that 
$$
\sup_{M \ge 1} \|X_M\|_{\hB^{\al,\si}_{p,p}} < \infty.
$$
(It will be clear that the proof can be adapted to yield that for every $M$, $X_M$ is in $\tilde B_{p,p}^{\alpha, M}$.)
We start by writing, for any $k \geq -1$ and $M \geq 1$,
\begin{align*}
\dk X_M = \dXi + \dXo,
\end{align*}
where 
\begin{align*}
\dXi &= \sum_{z \in M \Z^d} \dk  \td{X}_M (\cdot - z) \mathbf{1}_{[-\frac{M}{2}, \frac{M}{2})^d}( \cdot -z) ,\\
\dXo &=  \sum_{z \in M \Z^d} \dk  \td{X}_M (\cdot - z) \big(1 - \mathbf{1}_{[-\frac{M}{2}, \frac{M}{2})^d} ( \cdot -z)\big) \;.
\end{align*}
For $\dXi$, we get
\begin{align}
\| \dXi \|_{\hat{L}_\si^p}^p& =  \sum_{z \in M\Z^d} \int \big| \dk \td{X}_M( x- z) \big|^p \mathbf{1}_{[-\frac{M}{2}, \frac{M}{2})^d}( x -z) \, \hat{w}_\si(x) \, \d x  \notag\\
& =  \int_{{[-\frac{M}{2}, \frac{M}{2})^d} } \big| \dk \td{X}_M( x) \big|^p  \sum_{z \in M\Z^d} \hat{w}_\si(x+z) \, \d x  \notag \;.
\end{align}
For $x \in [\frac{M}{2}, \frac{M}{2})^d$, we can write
\begin{align*}
\sum_{z \in M\Z^d} \hat{w}_\si(x+z)& =   \hat{w}_\si(x) + \sum_{\substack{z \in M\Z^d\\z \neq 0}} \frac{1}{|x+z|_*^{\sigma}}  \ls \hat{w}_\si(x) + \sum_{\substack{z \in M\Z^d\\z \neq 0}} \frac{1}{|z|^{\sigma}}  \\
&\ls \hat{w}_\si(x) + M^{-\si} \;,
\end{align*}
where we have used that fact that uniformly over $M \geq 1$, $x \in  [\frac{M}{2}, \frac{M}{2})^d$ and $0 \neq z \in M \Z^d$, we have $|z | \ls  |x + z| $. Observing that uniformly over $M\geq 1$ and $x \in  [\frac{M}{2}, \frac{M}{2})^d$, we have $w_\si(x) \gtrsim M^{-\si}$, we can conclude that 
\begin{align*}
\sum_{z \in M\Z^d} \hat{w}_\si(x+z) \ls \hat{w}_\si(x) \;,
\end{align*}
and hence $\| \dXi \|_{\hL^p}^p \ls \|\dk \td X_M\|_{\hL^p}^p $.

 In order to treat $\dXo$, we  
recall that $\dk \td X_M = \eta_k \star \td X_M$, where $\eta_k = 2^{2k} \eta(2^{k} \cdot)$ and $\eta \in \mcl S$. Formally, 
$$
|\dk \td X_M|(x) = \int \eta_k(x-\cdot) \td X_M = \int \frac{\eta_k(x-\cdot)}{\wh w_\si } \, \ph_{2M} \, \td X_M \, \wh w_\si ,
$$
and arguing by density, we obtain by Proposition~\ref{p:dual} that
$$
|\dk \td X_M|(x) \le \Ll\| \frac{\eta_k(x-\cdot)}{\wh w_\si}\ph_{2M} \Rr\|_{\hB^{-\al,\si}_{p,p}} \, \|\td X_M\|_{\hB^{\al,\si}_{p,p}}.
$$
Since $\eta \in \mcl S$, it is straightforward to check (using for instance \eqref{e:bound-Sobolev}) that for any $l$, there exists $C < \infty$ such that uniformly over $M \ge 1$ and $|x| \ge M$,
$$
|\dk \td X_M|(x) \le \frac{C}{|2^kx|^l} \|\td X_M\|_{\hB^{\al,\si}_{p,p}}.
$$
This bound allows to give a uniform bound on $\dXo(x)$. We can assume without loss of generality that $x \in  [\frac{M}{2}, \frac{M}{2})^d$  and write
\begin{align*}
\dXo(x) &= \sum_{\substack{z \in M \Z^d\\ z \neq 0}} \dk  \td{X}_M (x- z) \ls \|\td X_M\|_{\hB^{\al,\si}_{p,p}}  \sum_{\substack{z \in M \Z^d\\ z \neq 0}} \frac{C}{2^{kl}   | z+x|^l} \\
&\ls  \|\td X_M\|_{\hB^{\al,\si}_{p,p}} 2^{-kl}  M^{-l} \;.
\end{align*}
Integrating the $p$-th power of this uniform bound against $\hat{w}_\si$, we obtain the desired bound on $\| \dXo \|_{\hL^p}^p $, and hence the uniform in $M$ bound on  $\| X_M \|_{\hB_{p,p}^{\alpha, \si}}^p$. The fact that $\| X_M \|_{\tB_{p,p}^{\alpha, M}}^p <\infty$ follows by the same arguments.

To see the convergence of $X_M$ to $X$ as $M $ tends to infinity, we only need to recall that by Proposition~\ref{p:compact} and \eqref{e:pandq}, the embedding of $\hat{B}_{p,p}^{\alpha, \si} $ into $\hat{B}_{p,p}^{\td{\alpha}, \td{\si}} $ is compact, and that every accumulation point is identified to be $X$ by \eqref{e:ppty_centre}. 
\end{proof}

For  $M \geq 1$ and every integer $n \geq 1$, let $\cZ^{: n :}$ and $\ZM^{: n :}$ be the modifications constructed in Theorem~\ref{thm:BoundsOnZ} (we drop the tildes for notational convenience). 
We define 
\begin{align*}
V(t) = e^{\Delta t} X_0 \qquad V_M(t) = e^{\Delta t} X_{0;M} \;,
\end{align*}
where $X_{0;M}$ is the periodic distribution constructed from $X_0$ as  in Lemma \ref{l:periodisation1}. We also define
\begin{align*}
\mathfrak{c}(t) = \int_t^1 \int_{\R^2} K(r,x)^2 \, \d x = \frac{ \log(t^{-1})}{8 \pi}, \qquad \mathfrak{c}_M(t) = \int_t^1 \int_{[\frac{-M}{2}, \frac{M}{2} ]^2} K_M(r,x)^2 \, \d x.
\end{align*}
It is easy to see that uniformly in $t \in [0,1]$,
\begin{align}
\big| \mathfrak{c}(t) - \mathfrak{c}_M(t) \big| \leq 4 \| K \mathbf{1}_M \|_{L^2( [0,1] \times \R^2)} ^2 \leq \Big(\frac{2}{\pi}\Big)^{\frac32} \frac{1}{M}e^{-\frac{M^2}{2}} ,  \label{e:ct}
\end{align}
where $\mathbf{1}_M$ is used as a shorthand for the indicator function of the set $\{ (t,x) \colon x \notin [-\frac{M}{2} \frac{M}{2}]^2 \}$.

Finally, we set
\begin{align}
Z^{: 1 :}_t &= Z_t = \cZ(t) + V(t) \;, \label{e:DefZN}\\
Z^{: 2 :}_t &= \big(\cZ^{: 2 : }(t) - \mathfrak{c}(t) \big) + 2 \cZ(t) V(t) + V(t)^2 \;,\notag\\
Z^{: 3 :}_t &= \big(\cZ^{: 3 : }(t) -3 \mathfrak{c}(t) \cZ(t) \big)+ 3 \big(\cZ^{: 2 : }(t)- \mathfrak{c}(t) \big) V(t) + 3 \cZ(t) V^2(t) + V(t)^3 \;,\notag
\end{align}
and we define $Z_{t;M}^{: 1 :}$, $Z^{: 2 :}_{t;M}$, $Z^{: 3 :}_{t;M}$ by replacing all the distributions $\cZ, \cZ^{: 2 :}, \cZ^{: 3 :}, V$ and $\mathfrak{c}$ by  $\ZM, \ZM^{: 2 :}, \ZM^{: 3 :}, V_M, \mathfrak{c}_M$ in the definitions above. 
\begin{rem}\label{rem:RenConst}
We introduce the additional constant $\mathfrak{c}(t)$ to be consistent with the smooth approximation referred to in Remark~\ref{rem:SmoothApproximation} (and with \cite{dPD, Martin1} as well as our companion paper \cite{JCH}).

 As in Section~\ref{s:mainResult}, let $\xi_\delta$ be a regularised space-time white noise,  and let $\cZ_\delta$ be the solution of the stochastic heat equation \eqref{e:eqZ2} with $\xi$ replaced by $\xi_\delta$. For every positive $\delta$, $\cZ_\delta$ is a smooth function, and hence arbitrary powers of it can be defined without ambiguity. Furthermore, It\^o's formula for iterated stochastic integrals (see \cite[Prop 1.1.3]{Nualart}) shows that 
\begin{align*}
\cZ_\delta  \to \cZ,   \qquad \cZ_\delta^2 - \overline{\mathfrak{c}}_\delta(t) \to \cZ^{:2:}, \qquad \cZ_\delta^3 - 3 \overline{\mathfrak{c}}_\delta(t) \cZ_\delta \to \cZ^{:3:} \;,
\end{align*}
where $\overline{\mathfrak{c}}_\delta(t) = \| K \mathbf{1}_{[0,t] \times \R^2} \ast \rho_\delta \|_{L^2(\R \times \R^2)}^2$. Indeed, \cite[Prop 1.1.3]{Nualart} immediately implies that this convergence holds in $L^2(\Omega, \mathcal{F}, \mathbb{P})$ for fixed time $t$ and test function~$\ph$. Adapting the argument in Theorem~\ref{thm:BoundsOnZ}, it is also possible to establish that this convergence holds in weighted Besov spaces (see \cite[Section~10]{Martin1} for calculations in this spirit).

However, above we announced that the renormalised powers arise as limit of a renormalisation procedure with time-independent normalisation constants. The time dependence of the $\overline{\mathfrak{c}}(t)$ is somewhat artificial and only arises because of our choice to start \eqref{e:eqZ2} with zero as initial datum. 

This issue can easily be removed by setting
\begin{align*}
\mathfrak{c}_\delta = \| K \mathbf{1}_{[0,1] \times \R^2} \ast \rho_\delta \|_{L^2(\R \times \R^2)}^2
\end{align*}
and defining the renormalised powers as limits of 
\begin{align*}
\cZ_\delta^2 - \mathfrak{c}_\delta  \qquad \text{and} \qquad \cZ_\delta^3 - 3 \mathfrak{c}_\delta \cZ_\delta  
\end{align*}
The $\mathfrak{c}(t)$ in the definition of $Z_t^{ n :}$ arise as the limit of $ \mathfrak{c}_\delta - \overline{\mathfrak{c}}_\delta(t) $.
\end{rem}

We summarise the results of this section in the following corollaries. The second integrability index of the weighted Besov spaces will not be important, and from now on we will always choose it as $\infty$. The first corollary will be used as input in the construction of the periodic solutions on a torus, see Section~\ref{s:torus}.
\begin{cor}\label{cor:BoundsOnZ1} 
Fix $0 <\bar{\al} \leq 1$ and $T>0$, $\si >2$ and $p \geq 1$. We assume that $X_0 \in \hB_{p,\infty}^{-\bar{\al}, \si}$.

Then for $n=1,2,3$, $\alpha = \bar{\alpha} +\frac{2}{p}$ and $\alpha' >\alpha$ we have for every $1 \leq M <\infty$
\begin{align}
\E  \sup_{0 < t \leq T}   t^{ (n-1) \alpha' p} \|Z_{t;M}^{: n :} \|_{\tB_{\infty,\infty}^{-\alpha, M }}^p < \infty \;. \label{e:ZFinalM}
\end{align}
\end{cor}
\begin{proof}
First we observe that for $M< \infty$ the bound \eqref{e:ZZZ3} implies for all $n$
\begin{align*}
 \E  \sup_{0 \leq t \leq T}    \|\ZM^{: n :}(t) \|_{\tB_{\infty,\infty}^{-\alpha, M}}^p < \infty \;.
\end{align*}
Indeed, we can chose $\bar{\alpha} = \frac{\alpha}{2}$ and  $\bar{p} = \frac{4}{\alpha} \vee p $ to get
\begin{eqnarray*}
 \E  \sup_{0 \leq t \leq T}    \|\ZM^{: n :}(t) \|_{\tB_{\infty,\infty}^{-\alpha, M} }^p   \!\!\!\!\!\!\!\!\!\!\!\!\!\!  &\leq& \!\!\!\!\!\!\!  \Big(\E  \sup_{0 \leq t \leq T}    \|\ZM^{: n :}(t) \|_{\tB_{\infty,\infty}^{-\alpha, M}}^{\bar{p}} \Big)^{\frac{p}{\bar{p}}}  \\
 &\overset{\text{Prop.} \ref{p:embed}, \eqref{e:pandq}}{\ls}&  \!\!\!\!\!\!\! \Big(\E  \sup_{0 \leq t \leq T}    \|\ZM^{: n :}(t) \|_{\tB_{\bar{p},\bar{p}}^{-\bar{\alpha}, M}}^{\bar{p}} \Big)^{\frac{p}{\bar{p}}} < \infty\;.
\end{eqnarray*}
We get for any $\be \geq -\alpha$
\begin{align*}
\| V_M(t) \|_{\tB_{\infty,\infty}^{\beta, M}} &\overset{\text{Prop.}~\ref{p:smooth-besov}}{\ls} t^{-\frac{\alpha +\beta }{2}} \| X_{0;M} \|_{\tB_{\infty,\infty}^{-\alpha, M}} \ls t^{-\frac{\alpha +\beta }{2}} \| X_{0;M} \|_{\tB_{p ,p}^{-\bar{\alpha} , M}} \\
&\overset{\text{Lem.}~\ref{l:periodisation1}}{\ls}   t^{-\frac{\alpha +\beta }{2}} \| X_0 \|_{\hB_{p,p}^{-\bar{\alpha}, \si}} \;.
\end{align*}
Furthermore, \eqref{e:ct} implies that $\mathfrak{c}_M(t) \ls 1 + (\log(t^{-1})\vee 0) $ uniformly over $t \in [0,T]$. The desired bound \eqref{e:ZFinalM} then follows from the multiplicative inequality, Corollary~\ref{c:multipl2}.
\end{proof}
The following corollary will be used together with the a priori bounds of Section~\ref{s:full1}, to show the convergence (along a subsequence) of the periodised solutions when $M$ goes to infinity (see Section~\ref{s:exist-full}).
\begin{cor}\label{cor:BoundsOnZ2}
Fix $0<\al \leq 1$, $T>0$, $\si >2$, $p \geq 3$, and assume that $X_0 \in \hB_{p,\infty}^{-\al, \si}$.
For every $\alpha'> \alpha$ and $\td{\si} > \si$, with probability one, there exists a sequence $(M_k)_{k} $ going to infinity such that for $n = 1,2,3$,
\begin{align*}
 \sup_{k} \sup_{0 < t \leq T}   t^{ (n-1) \alpha' } \|Z_{t;M_k}^{: n :} \|_{\hB_{\frac{p}{n},\infty}^{-\alpha, \si}} < \infty \;,\\
\lim_{k \to \infty}  \sup_{0 < t \leq T } t^{(n-1)\alpha' } \|Z^{:n:}_t -Z_{t;M_k}^{: n :} \|_{\hB_{\frac{p}{n},\infty}^{-\alpha', \td{\si}}} =0\;.
\end{align*}
\end{cor}

\begin{proof}
By \eqref{e:ZZZ1} and \eqref{e:ZZZ4}, with probability one, there exists a sequence $(M_k)_k$ going to infinity and such that for $n=1,2,3$,
\begin{align*}
\sup_{k} \sup_{0 \leq t \leq T} &\| \cZ_{M_k}^{: n :}(t)\|_{\wh B_{\frac{p}{n},\infty}^{-\alpha, \si} } < \infty,\\
\lim_{k \to \infty} \sup_{0 \leq t \leq T} &\|\cZ^{: n:}(t) - \cZ_{M_k}^{: n :}(t)\|_{\wh B_{\frac{p}{n},\infty}^{-\alpha, \td{\si}} } =0\;.
\end{align*}
Furthermore, according to \eqref{e:ct}, we have 
\begin{align*}
\mathfrak{c}_M(t) &\ls 1+ \big( \log (t^{-1}) \vee 0\big) \;, \\
\lim_{M \to \infty }\sup_{ 0 \leq t \leq T} \big| \mathfrak{c}_M(t) - \mathfrak{c}(t)\big| &=0 \;,
\end{align*}
where the implicit constant in the first inequality is uniform in $M$. Finally, Lemma~\ref{l:periodisation1} and Proposition~\ref{p:smooth-besov} imply that 
\begin{align*}
\| V_M(t) \|_{\hB_{p,\infty}^{\gamma, \si} } \ls t^{-\frac{\gamma- \alpha}{2}} \qquad  \lim_{M \to \infty }\sup_{0 \leq t \leq T }t^{\frac{\gamma - \alpha'}{2}}  \| V_M(t) - V(t)  \|_{\hB_{p,\infty}^{\gamma , \td{\si}} } =0\;,
\end{align*}
where  the first inequality is valid for  any $\gamma \geq - \alpha$ and the second for $\gamma \geq -\alpha'$.

This bound and the multiplicative inequalities in Corollaries~\ref{c:multipl1} and~\ref{c:multipl2} imply that  for any $\alpha'' > \alpha$ and  $\ga \geq -\alpha$ in the first inequality and for $\alpha'' > \alpha'$ and $\ga \geq -\alpha'$ in the second inequality, we get for $n=1,2,3$ that
\begin{align*}
&\| V_M^n(t) \|_{\hB_{\frac{p}{n},\infty}^{\gamma, \si} } \ls t^{-\frac{(n-1)\alpha'' +\ga + n\alpha}{2} }  \\
&  \lim_{M \to \infty }\sup_{0 \leq t \leq T }t^{\frac{(n-1)\alpha'' +\ga +(n-1)\alpha +\alpha'}{2} }  \| V_M^n(t) - V^n(t)  \|_{\hB_{\frac{p}{n},\infty}^{\gamma , \td{\si}} } =0\;.
\end{align*}
We can always choose $\alpha''$ such that $\frac{\alpha'' + \alpha}{2} < \alpha'$. Then the desired bound follows from another application of the multiplicative inequality, Corollary~\ref{c:multipl2}, using at several places the fact that according to \eqref{e:pandq} decreasing the integrability index $p$ only makes a Besov norm weaker.
\end{proof}
%
%
%
%
%
%
\section{Construction of solutions on the torus}
\label{s:torus}
The aim of this section is to show existence and uniqueness of global solutions of \eqref{e:eqY} in the periodic case.

\medskip

Since for the most part, the index $q$ in $\tBb$ will not play an important role in our analysis, we introduce the slightly lighter notation
\begin{equation}
\label{e:def:short-tB}
\tB^{\al,M}_{p} := \tB^{\al,M}_{p,\infty}.
\end{equation}

\medskip

Let $0 < \al < \al' < 1/3$ (that we think of as being small), $\beta \in (1, 2)$ (that 
we think of as being close to $2$) and $T > 0$. 
For $\un{Z} = (Z,\Zdd,\Ztt)$ in the set
$$
\mcl{C}([0,T],\tB_\infty^{-\al,M}) \times \mcl{C}((0,T],\tB_\infty^{-\al,M}) \times \mcl{C}((0,T],\tB_\infty^{-\al,M}),
$$
we write
\begin{equation}
	\label{e:def:Z-norm}
	\|\un{Z}\|_{\td{\msc{Z}}^{M}_\infty} 
	:= \sup_{0 \le t \le T} \Ll(\|Z_t\|_{\tB_\infty^{-\al,M}} \, \vee \,  t^{\al'} \|\Zdd_t\|_{\tB_\infty^{-\al,M}} \, \vee \, t^{2\al'} \|\Ztt_t\|_{\tB_\infty^{-\al,M}} \Rr)
\end{equation}
(recall that $a \vee b$ stands for the maximum between $a$ and $b$). We let $\td{\msc{Z}}^{M}_\infty$ be the set of $\un{Z}$ such that this norm is finite.
%
%
%
For $Y \in \mcl{C}([0,T],\tB^{\be,M}_\infty)$, $\un{Z} = (Z, \Zdd,\Ztt) \in \td{\msc{Z}}^{M}_\infty$ and $t \le T$, we write 
\begin{equation}
\label{e:def:Psi}
\Psi(Y_t,\un{Z}_t) = -Y_t^3 - 3 Y_t^2 Z_t - 3 Y_t \Zdd_t - \Ztt_t + a(Y_t+Z_t).
\end{equation}
The multiplicative inequalities imply that $\Psi(Y_t,\un{Z}_t)$ is well-defined for every $t > 0$ (it actually suffices that $Y$ belong to $\mcl{C}([0,T],\tB^{\al'',M}_\infty)$ for some $\al'' > \al$). For a given $Y_0 \in \tB_\infty^{\be,M}$ and $T > 0$, we wish to solve
\begin{equation}
\label{e:eq-for-Y-torus}
\Ll\{ 
\begin{array}{l}
\partial_t Y = \Delta Y + \Psi(Y,\un{Z}) \qquad \text{on } [0,T] \times \R^2, \\
Y(0,\cdot) = Y_0,
\end{array}
\Rr.
\end{equation}
which we interpret in the mild form. That is, 
we say that $Y$ solves \eqref{e:eq-for-Y-torus} if $Y \in \mcl{C}([0,T],\tB^{\be,M}_\infty)$ and if for every $t \le T$,
$$
Y_t = e^{t\Delta} Y_0 + \int_0^t e^{(t-s)\Delta} \Psi(Y_s,\un{Z}_s) \, \d s.
$$
We let $\td{S}_\infty^{T,M}(Y_0,\un{Z})$ 
be the set of solutions of \eqref{e:eq-for-Y-torus}.

\medskip

The goal of this section is to prove the following global existence and uniqueness result for \eqref{e:eq-for-Y-torus}.

\begin{thm}[global existence and uniqueness on the torus]
	\label{t:global-torus}
For every $\be \in (1,2)$, the following holds for $0 < \al < \al'$ sufficiently small.
	Let $T > 0$, $M > 0$, $\un{Z} =  (Z, \Zdd,\Ztt) \in \td{\msc{Z}}^{M}_\infty$ and $Y_0 \in \tB^{\be,M}_\infty$. There exists exactly one solution of equation~\eqref{e:eq-for-Y-torus} over the time interval $[0,T]$. In other words, the set $\td{S}_\infty^{T,M}(Y_0,\un{Z})$ is a singleton.
\end{thm}

\begin{thm}[local existence and uniqueness on the torus]
Let $p \in [1,\infty)$ be such that
\begin{equation}
\label{e:cond-local}
\frac{\al' + \be}{2} + 3\Ll(\al' + \frac{1}{p} \Rr) < 1,
\end{equation}
let $\un{Z} = (Z, \Zdd,\Ztt) \in \td{\msc{Z}}^{M}_\infty$ and $K > 0$. There exists $T^{\star}$ such that for every $Y_0 \in \tB^{\be,M}_\infty$ satisfying $\|Y_0\|_{\tL^p} \le K$, equation \eqref{e:eq-for-Y-torus} has exactly one solution over the time interval $[0,T^{\star} \wedge T]$. In other words, the set $\td{S}_\infty^{T^{\star} \wedge T,M}(Y_0,\un{Z})$ is a singleton. 
\label{t:local}
\end{thm}
\begin{rem}
The condition on $\al'$, $\be$ and $p$ displayed in \eqref{e:cond-local} can be improved. Note that it suffices to assume $\al'$ to be sufficiently small to ensure that a $p \in [1,\infty)$ exists that satisfies \eqref{e:cond-local}. 
\end{rem}
\begin{proof}[Proof of Theorem~\ref{t:local}]
An important aspect of the theorem is that we want $T^{\star}$ to depend on $Y_0$ only through the bound $K$ on $\|Y_0\|_{\tL^p}$. In order to achieve this, we split the construction of a solution into two steps. In the first step, we construct a mild solution using only the information that $Y_0 \in \tL^p$; the price to pay for this is that the solution will be defined in a larger space than anticipated. In the second step, we show using the additional information that $Y_0 \in \tBB{\be}$ that the solution thus constructed belongs to $\mcl{C}([0,T^{\star}],\tBB{\be})$. Finally, we argue about uniqueness in a third step.

\medskip

\noindent{\emph{Step 1}}. In view of \eqref{e:cond-local}, we can choose $\al'' > \al'$ such that
$$
\frac{\al'' + \be}{2} + 3\Ll(\al'' + \frac{1}{p}\Rr) < 1.
$$
We construct a solution of \eqref{e:eq-for-Y-torus} in $\mcl{C}((0,T^{\star}],\tBB{\al''})$ by a fixed point argument. Let $\gamma$ be such that
\begin{equation}
\label{e:cond:al''}
\ga > \al'' + \frac{1}{p} \quad \text{and} \quad \frac{\al'' + \be}{2} + 3\ga < 1.
\end{equation}
For notational convenience, we assume that $T \ge 1$. For any $T^{\star} \le 1$, we define the norm
$$
\tn Y \tn_{T^{\star}} = \sup_{0 \le t \le T^{\star}} t^\ga \|Y_t\|_{\tB_\infty^{\al'',M}},
$$
and the ball
$$
\BK = \{Y \in \mcl{C}((0,T],\tB_\infty^{\al'',M}) : \tn Y \tn_{T^{\star}} \le 1 \}.
$$
For $Y \in \BK$, we let
\begin{equation}
\label{e:def:MT}
\MT Y (t) = e^{t\Delta} Y_0 + \int_0^t e^{(t-s)\Delta} \Psi(Y_s,\un{Z}_s) \, \d s \qquad (t \le T^{\star}).
\end{equation}
Our aim is to show that for $T^{\star}$ small enough, the operator $\MT$ is a contraction from $\BK$ into itself. To begin with, it is clear that for every $t$, $\MT Y (t)$ is $M$-periodic. We now argue that for $T^{\star}$ sufficiently small, $\MT$ maps $\BK$ into itself. (The periodic versions of) Propositions~\ref{p:embed} and \ref{p:smooth-besov} ensure that
$$
\|e^{t\Delta} Y_0 \|_{\tB^{\al'',M}_\infty} \lesssim \|e^{t\Delta} Y_0 \|_{\tB^{\al''+\frac{2}{p},M}_p} \lesssim t^{-\Ll( \frac{\al''}{2} + \frac{1}{p} \Rr)} \|Y_0\|_{\tB^{0,M}_{p}} \lesssim t^{-\Ll( \frac{\al''}{2} + \frac{1}{p} \Rr)} \|Y_0\|_{\tL^p},
$$
where we used Remark~\ref{r:lp-embed} in the last step.
From this, one can check that $t \mapsto e^{t\Delta} Y_0$ is in $\mcl{C}((0,T],\tB^{\al'',M}_\infty)$, and moreover,
$$
\tn e^{\cdot \Delta} Y_0 \tn_{T^{\star}}
$$
can be made arbitrarily small by taking $T^{\star}$ small enough (in terms of $K$) since $\ga > \frac{\al''}{2}+\frac1p$.

Concerning the integral term in \eqref{e:def:MT}, we observe that by the multiplicative inequalities (and the trivial embedding of Besov spaces $\tBB{\al}$ as $\al$ varies),
\begin{multline}
\label{e:mult-ineq-on-Psi}
\|\Psi(Y_s,\un{Z}_s)\|_{\tBB{-\al}} \lesssim \|Y_s\|_{\tBB{\al''}}^3 + \|Y_s\|_{\tBB{\al''}}^2 \|Z_s\|_{\tBB{-\al}} \\ 
	+ \|Y_s\|_{\tBB{\al''}} \|\Zdd_s\|_{\tBB{-\al}} + \|\Ztt_s\|_{\tBB{-\al}} + \|Y_s\|_{\tBB{\al''}} + \|Z_s\|_{\tBB{-\al}} .
\end{multline}
Moreover, since $\|\un{Z}\|_{\td{\msc{Z}}^{M}_\infty}$ is finite (see \eqref{e:def:Z-norm}), if $Y \in \BK$, then
\begin{equation*}
\|\Psi(Y_s,\un{Z}_s)\|_{\tBB{-\al}} \lesssim s^{-3\ga} + s^{-2\ga} + s^{-\ga -\al'} + s^{-2\al'} \lesssim s^{-3\ga} ,
\end{equation*}
since we assume that $\ga > \al'' > \al'$. As a consequence, using Proposition~\ref{p:smooth-besov} again,
\begin{eqnarray*}
\Ll\| \int_0^t e^{(t-s) \Delta} \Psi(Y_s,\un{Z}_s) \, \d s \Rr\|_{\tBB{\al''}} & \lesssim & \int_0^t (t-s)^{-\frac{\al + \al''}{2}} \|\Psi(Y_s,\un{Z}_s) \|_{\tBB{-\al}} \, \d s \\
& \stackrel{(\al \le \al'')}{\lesssim} & \int_0^t (t-s)^{-\al''} s^{-3\ga} \, \d s \\
& \lesssim & t^{1-\al'' - 3\ga} ,
\end{eqnarray*}
where we used the fact that $\al''<1$ and $3\ga < 1$ in the last step. 
Recalling \eqref{e:cond:al''} and $\be > 1$, we see that $1-\al''-2\ga >2-\be -\al''-6\ga >  0$, so that the right-hand side above multiplied by $t^\ga$ tends to $0$ as $t$ tends to $0$. The fact that the process $t \mapsto \int_0^t e^{(t-s) \Delta} \Psi(Y_s,\un{Z}_s)  \d s$ taking values in $\tBB{\al}$ is continuous poses no additional difficulty. Hence, we have shown that the $\tn \cdot \tn_{T^{\star}}$-norm of this process can be made arbitrarily small provided that we choose $T^{\star}$ small enough. 

This concludes the proof that for $T^{\star}$ sufficiently small, the operator $\MT$ maps $\BK$ into itself. The contraction property then follows along similar lines. 

\medskip 

\noindent{\emph{Step 2}}. We now take the solution $Y \in \mcl{C}((0,T^{\star}],\tBB{\al''})$ constructed in Step 1 and show that it actually belongs to $\mcl{C}([0,T^{\star}],\tBB{\be})$.  Recall that 
$$
Y_t = e^{t \Delta} Y_0 + \int_0^t e^{(t-s)\Delta} \Psi(Y_s,\un{Z}_s) \, \d s,
$$
and that we already know that
$$
t^\ga \|Y_t\|_{\tBB{\al''}} < \infty.
$$
By Remark~\ref{r:continuity}, the function $t \mapsto e^{t \Delta} Y_0$ belongs to $\mcl{C}([0,T^{\star}],\tBB{\be})$. Similarly to what was done in the previous step, we can estimate
\begin{eqnarray*}
\Ll\| \int_0^t e^{(t-s) \Delta} \Psi(Y_s,\un{Z}_s) \, \d s \Rr\|_{\tBB{\be}} & \lesssim & \int_0^t (t-s)^{-\frac{\al + \be}{2}} \|\Psi(Y_s,\un{Z}_s) \|_{\tBB{-\al}} \, \d s \\
& \lesssim & \int_0^t (t-s)^{-\frac{\al + \be}{2}} s^{-3\ga} \, \d s \\
& \lesssim & t^{1-\frac{\al + \be}{2} - 3\ga} ,
\end{eqnarray*}
since we have $\frac{\al + \be}{2} < 1$ and $3\ga < 1$. The continuity of $t \mapsto \int_0^t e^{(t-s) \Delta} \Psi(Y_s,\un{Z}_s)  \d s$ in $\tBB{\be}$ at any point in $(0,T^{\star}]$ follows along the same lines. The continuity at time $0$ follows from the fact that (recall \eqref{e:cond:al''} and $\al < \al''$)
$$
\frac{\al + \be}{2} + 3\ga < 1.
$$

\medskip 

\noindent{\emph{Step 3}}. For $T^{\star}$ as defined by Step 1, let $Y, \td{Y} \in \td{S}_\infty^{T^{\star},M}(Y_0,\un{Z})$. We wish to show that $Y = \td{Y}$. Since $Y,\td{Y} \in \mcl{C}([0,T],\tB^{\be,M}_{\infty})$, it follows from the reasoning in Step 1 that one can find $T^{\star \star} \le T^{\star}$ depending only on 
$$
\sup_{t \le T^{\star}} \|Y_t\|_{\tB^{\be,M}_{\infty}} \quad  \text{ and } \quad \sup_{t \le T^{\star}} \|\td{Y}_t\|_{\tB^{\be,M}_{\infty}}
$$
such that both $Y$ and $\td{Y}$ belong to $\msc{B}_{T^{\star \star}}$. Since $T^{\star \star} \le T^{\star}$, the argument in Step 1 ensures that $\msc{M}_{T^{\star \star}}$ is a contraction from $\msc{B}_{T^{\star \star}}$ into itself. As a consequence, $Y$ and $\td{Y}$ coincide on $[0,T^{\star \star}]$. We can then iterate the reasoning (keeping the same $T^{\star \star}$) and thus guarantee that $Y$ and $\td{Y}$ coincide on the whole interval $[0,T^{\star}]$. (Note that this argument is easily adapted to show the uniqueness part of Theorem~\ref{t:global-torus}.)
\end{proof}

In order to upgrade the local existence result to a global one, we need a control of the $\tL^p$ norm of the solution.

\begin{thm}[A priori estimate on the torus]
\label{t:apriori-torus}
Let $p$ be an even positive integer such that
	\begin{equation}
		\label{e:conds-p}
\al'(p+2) < 1 \quad \text{and} \quad \frac{\al' + \be}{2} + 3\Ll(\al' + \frac{1}{p} \Rr) < 1,
	\end{equation}
let $T > 0$, $M > 0$ and $\un{Z} \in \td{\msc{Z}}^{M}_\infty$. There exists $C < \infty$ such that if $Y \in \td{S}_\infty^{T^{\star},M}(Y_0,\un{Z})$ for some $T^{\star} \le T$, then 
$$
\sup_{0 \le t \le T^{\star}} \|Y_t\|_{\tL^p} \le \|Y_0 \|_{\tL^p} + C.
$$
\end{thm}
At a formal level, the idea for proving this a priori bound consists in testing equation \eqref{e:eq-for-Y-torus} against $Y^{p-1}$, which leads to a useful identity concerning $\dr_t \|Y_t\|_{\tL^p}^p$. The rigorous derivation will need several preliminary steps, starting with a time regularity estimate.
\begin{prop}[Time regularity of solutions]
	\label{p:time-reg}
	Let $Y \in \td{S}_\infty^{T,M}(Y_0,\un{Z})$, and assume that \eqref{e:conds-p} holds for some $p$. For every $\ka < \be/2$, $Y$ is $\ka$-Hölder continuous as a function from $[0,T]$ to $\tL^\infty$.
\end{prop}

\begin{proof}
We show Hölder regularity of $Y$ at time $0$, the adaptation to arbitrary times in $[0,T]$ being straightforward. 
For conciseness, we write 
\begin{equation}
\label{e:def:psis}
\Psi_s:=\Psi(Y_s,\un{Z}_s).
\end{equation}
By definition of the notion of solution, we have
$$
Y_t = e^{t\Delta} Y_0 + \int_0^t e^{(t-s)\Delta} \Psi_s \, \d s.
$$
The fact that $t \mapsto e^{t\Delta} Y_0$ is $\ka$-Hölder continuous as a function $[0,T] \to \tL^\infty$ is a consequence of Proposition~\ref{p:time-reg-flow} (and Remarks~\ref{r:embed-q} and \ref{r:embed-lp}) and of the fact that $Y_0 \in \tB^{\be,M}_\infty$. For the time integral, we use \eqref{e:mult-ineq-on-Psi} again (but with the additional information that $Y \in \mcl{C}([0,T],\tB^{\be,M}_\infty)$) to obtain that
\begin{equation}
\label{e:estpsi-torus}
\|\Psi_s\|_{\tBB{-\al}} \lesssim s^{-2\al'},
\end{equation}
so that
\begin{eqnarray*}
\Ll\| \int_0^t e^{(t-s) \Delta} \Psi_s \, \d s \Rr\|_{\tBB{\al}} & \lesssim & \int_0^t (t-s)^{-\al} \|\Psi_s \|_{\tBB{-\al}} \, \d s \\
& \lesssim & \int_0^t (t-s)^{-\al} s^{-2\al'} \, \d s \\
& \stackrel{(\al \le \al')}{\lesssim} & t^{1-3\al'}.
\end{eqnarray*}
The result follows since $1-3\al' > \be/2$ by \eqref{e:conds-p}. 
\end{proof}
We denote by $\lan \, \cdot \, , \, \cdot \, \ran$ the scalar product in $\tL^2$. 
\begin{rem}
\label{r:dual-per}
By Proposition~\ref{p:dual}, for every $\ga \in \R$, the mapping $(f,g) \mapsto \lan f , g \ran$ extends to a continuous bilinear form on $\tBB{\ga} \times \tB^{-\ga,M}_{1,1}$, and thus also on $\tBB{\ga} \times \tB^{-\ga+\eps,M}_{\infty}$ for every $\eps > 0$ by Remarks \ref{r:besov-mu} and \ref{r:embed-q}. In particular, by Lemma~\ref{l:smooth}, we see that for every $\al \in \R$ and $M$-periodic $\phi \in C^\infty$, the mapping $f \mapsto \lan f,\phi \ran$ extends to a continuous linear form on $\tBB{\al}$.
\end{rem}
\begin{prop}[A mild solution is a weak solution]
	\label{p:weak1}
	If $Y \in \td{S}_\infty^{T,M}(Y_0,\un{Z})$, then for every 
	$\phi \in \tB^{1,M}_\infty$ and $t \le T$,
\begin{equation}
\label{e:weak-eq}
	\lan Y_t,\phi \ran - \lan Y_0,\phi \ran = \int_0^t \Ll[ - \lan \na Y_s,\na \phi\ran  + \lan \Psi(Y_s,\un{Z}_s),\phi \ran \Rr] \, \d s.
\end{equation}
\label{p:weak}
\end{prop}
\begin{proof}
\noindent \emph{Step 1.} Recall our notation \eqref{e:def:psis}. We first show that for every $M$-periodic $\phi \in C^\infty$ and $t \le T$,
\begin{equation}
	\label{e:weak1}
	\lan Y_t,\phi \ran - \lan Y_0,\phi \ran = \int_0^t \Ll[ \lan Y_s,\Delta \phi\ran + \lan \Psi_s,\phi \ran \Rr] \, \d s.
\end{equation}
Since $Y$ is a mild solution of \eqref{e:eq-for-Y-torus},
$$
\int_0^t \lan Y_s, \Delta \phi \ran \, \d s = \int_0^t \lan e^{s\Delta}Y_0 + \int_0^s e^{(s-u)\Delta} \Psi_u \, \d u, \Delta \phi \ran \, \d s.
$$
Observe that for every $s > 0$,
$$
\lan e^{s\Delta} Y_0,\Delta \phi \ran = \lan\Delta e^{s\Delta} Y_0,\phi\ran =  \dr_s \lan e^{s\Delta} Y_0,\phi\ran,
$$
so that
$$
\int_0^t \lan e^{s\Delta} Y_0,\Delta \phi \ran \, \d s = \lan e^{t\Delta} Y_0, \phi \ran - \lan Y_0,\phi \ran.
$$
Similarly, we compute
\begin{eqnarray*}
\int_0^t \int_0^s \Delta e^{(s-u)\Delta} \Psi_u \, \d u \d s & = & \int_0^t \int_u^{t} \Delta e^{(s-u)\Delta} \Psi_u \, \d s \d u \\
& = & \int_0^t (e^{(t-u)\Delta} -\mathrm{Id})\Psi_u \, \d u.
\end{eqnarray*}
Combining the two, we obtain
\begin{multline*}
\int_0^t \lan e^{s\Delta} Y_0,\Delta \phi \ran \, \d s \\
= \lan e^{t\Delta} Y_0 + \int_0^t e^{(t-u) \Delta} \Psi_u \, \d u, \phi \ran - \int_0^t \lan \Psi_u,\phi \ran \, \d u - \lan Y_0,\phi \ran,
\end{multline*}
which is \eqref{e:weak1}.

\noindent \emph{Step 2.} We now conclude the proof. First, the mapping $f \mapsto \lan \na f, \na \phi \ran$ is continuous over $\tBB{\be}$, so we have $\lan Y_s, \Delta \phi \ran = - \lan \na Y_s, \na \phi \ran$. Hence, for every $\phi \in \mcl C^\infty(\R^2)$,
$$
\lan Y_t,\phi \ran - \lan Y_0,\phi \ran = \int_0^t \Ll[ - \lan \na Y_s,\na \phi \ran + \lan \Psi_s,\phi \ran \Rr] \, \d s.
$$
It then suffices to argue by density using \eqref{e:estpsi-torus}, Proposition~\ref{p:derivatives} and Remark~\ref{r:dual-per}. 
\end{proof}

We are now ready to derive a rigorous version of the identity on $\dr_t \|Y_t\|_{\tL^p}^p$ alluded to before.
\begin{prop}[Testing against $Y_t^{p-1}$]
	Let $Y \in \td{S}_\infty^{T,M}(Y_0,\un{Z})$. For every even positive integer $p$ such that \eqref{e:conds-p} holds,
	\label{p:test}
	$$
	\frac{1}{p}\Ll(\|Y_t\|^p_{\tL^p} - \|Y_0\|^p_{\tL^p}\Rr) = \int_0^t \Ll[ -(p-1)\lan \na Y_s, Y_s^{p-2}\na Y_s \ran + \lan \Psi_s,Y_s^{p-1} \ran \Rr] \, \d s.
$$
\end{prop}
\begin{rem}
We could of course extend Proposition~\ref{p:weak} to allow for functions $\phi$ that depend smoothly on time. This would bring about the additional term
\begin{equation}
\label{e:Young-integral}
\int_0^t \lan Y_s, \dr_s \phi_s \ran \, \d s.
\end{equation}
However, since $t \mapsto Y_t^{p-1}$ is only Hölder continuous in time (see Proposition~\ref{p:time-reg}), it is not clear a priori how to make sense of \eqref{e:Young-integral} for $\phi = Y^{p-1}$. In essence, the argument below consists in noticing that we can make sense of Young integrals when both the integrand and the integrator are Hölder continuous for an exponent strictly above $1/2$.
\end{rem}
\begin{proof}[Proof of Proposition~\ref{p:test}]
We begin by observing that, for any $0 \le u \le v \le T$,
\begin{align*}
\|Y_v\|^p_{\tL^p} - \|Y_u\|^p_{\tL^p} & =  \lan Y_v, Y_v^{p-1} \ran - \lan Y_u, Y_u^{p-1} \ran \\
& =  \lan Y_v, Y_u^{p-1} \ran - \lan Y_u,Y_u^{p-1} \ran + \lan Y_v, Y_v^{p-1} - Y_u^{p-1} \ran,
\end{align*}
so by Proposition~\ref{p:weak},
\begin{multline*}
	\|Y_v\|^p_{\tL^p} - \|Y_u\|^p_{\tL^p} - \lan Y_v, Y_v^{p-1} - Y_u^{p-1} \ran\\
	= \int_u^v \Ll[-\lan \na Y_s, \na(Y_u^{p-1}) \ran + \lan \Psi_s,Y_u^{p-1} \ran \Rr] \, \d s.
\end{multline*}
For any subdivision $\un{t} = (t_0,\ldots,t_n)$ such that $0 = t_0 \le \cdots \le t_n = t$, we thus have
$$
	\|Y_t\|^p_{\tL^p} - \|Y_0\|^p_{\tL^p} - \mfk{S}(\un{t}) = \mfk{I}(\un{t}),
$$
where we used the shorthand notation
$$
\mfk{S}(\un{t}) = \sum_{i = 0}^{n-1} \lan Y_{t_{i+1}}, Y_{t_{i+1}}^{p-1} - Y_{t_{i}}^{p-1} \ran,
$$
$$
\mfk{I}(\un{t}) = \sum_{i = 0}^{n-1} \int_{t_i}^{t_{i+1}} \Ll[-\lan \na Y_s, \na(Y_{t_i}^{p-1}) \ran + \lan \Psi_s,Y_{t_i}^{p-1} \ran \Rr] \, \d s.
$$
Since $t \mapsto Y_t$ is a continuous function from $[0,T]$ to $\tBB{\be}$ with $\be > 1$, in particular $t \mapsto \na Y_t$ is a continuous function from $[0,T]$ to $\tL^\infty$ by Proposition~\ref{p:derivatives}. Using also \eqref{e:estpsi-torus}, one can check that as the subdivision gets finer and finer,
$$
\mfk{I}(\un{t}) \to \int_0^t \Ll[ -(p-1)\lan \na Y_s, Y_s^{p-2}\na Y_s \ran + \lan \Psi_s,Y_s^{p-1} \ran \Rr] \, \d s.
$$
We now argue that as the subdivision $\un{t}$ gets finer and finer,
\begin{equation}
\label{e:lim-mfkS}
\mfk{S}(\un{t}) \to \frac{p-1}{p} \Ll( 	\|Y_t\|^p_{\tL^p} - \|Y_0\|^p_{\tL^p} \Rr).
\end{equation}
Note that
\begin{equation}
\label{e:decomp-puiss}
Y_{t_{i+1}}^{p-1} - Y_{t_{i}}^{p-1} = (Y_{t_{i+1}} - Y_{t_{i}})(Y_{t_{i+1}}^{p-2} + \cdots + Y_{t_{i}}^{p-2}).
\end{equation}
We show that the contribution of each term in the sum above is asymptotically the same as the contribution of the first one. For instance with the last term,
\begin{multline*}
\lan Y_{t_{i+1}}, (Y_{t_{i+1}} - Y_{t_{i}}) Y_{t_{i+1}}^{p-2} \ran - \lan Y_{t_{i+1}}, (Y_{t_{i+1}} - Y_{t_{i}}) Y_{t_{i}}^{p-2} \ran \\
= \lan (Y_{t_{i+1}} - Y_{t_{i}})^2, Y_{t_{i+1}}(Y_{t_{i+1}}^{p-3} + \cdots + Y_{t_{i}}^{p-3} )\ran.
\end{multline*}
Since $t \mapsto Y_t$ is bounded in $\tL^p$, an application of Hölder's inequality leads to
\begin{multline*}
\Ll|\lan Y_{t_{i+1}}, (Y_{t_{i+1}} - Y_{t_{i}}) Y_{t_{i+1}}^{p-2} \ran - \lan Y_{t_{i+1}}, (Y_{t_{i+1}} - Y_{t_{i}}) Y_{t_{i}}^{p-2} \ran \Rr| \\
=    \Ll|\lan Y_{t_{i+1}}(Y_{t_i}^{p-3} + \cdots + Y_{t_{i+1}}^{p-3}), (Y_{t_{i+1}} - Y_{t_{i}})^2  \ran \Rr|
\lesssim \|Y_{t_{i+1}}-Y_{t_i}\|_{\tL^p}^2 \lesssim |t_{i+1}-t_i|^{2\ka}
\end{multline*}
for some $\ka > 1/2$, see Proposition~\ref{p:time-reg}. We can argue similarly for the other terms in the sum \eqref{e:decomp-puiss}. This estimate implies that as the subdivision $\un{t}$ gets finer and finer, the difference between $\mfk{S}(\un{t})$ and 
\begin{equation}
\label{e:p-1-times}
\sum_{i = 0}^{n-1} \lan Y_{t_{i+1}}, (p-1)Y_{t_{i+1}}^{p-2}(Y_{t_{i+1}} - Y_{t_i}) \ran
\end{equation}
tends to $0$. The same argument also shows that the difference between 
$$
\|Y_t\|^p_{\tL^p} - \|Y_0\|^p_{\tL^p} = \sum_{i=0}^{n-1} \Ll(\|Y_{t_{i+1}}\|^p_{\tL^p} - \|Y_{t_i}\|^p_{\tL^p} \Rr)
$$
and 
\begin{equation}
\label{e:p-times}
\sum_{i = 0}^{n-1} \lan 1, p Y_{t_{i+1}}^{p-1}(Y_{t_{i+1}} - Y_{t_i}) \ran
\end{equation}
tends to $0$. Since the quantity in \eqref{e:p-1-times} is $(p-1)/p$ times that in \eqref{e:p-times}, these two observations imply \eqref{e:lim-mfkS}.
\end{proof}

\begin{proof}[Proof of Theorem~\ref{t:apriori-torus}]
Recall that $\Psi(Y_s,\un{Z}_s)$ is defined as a sum of terms, see \eqref{e:def:Psi}. If we decompose it as $\Psi(Y_s,\un{Z}_s) = -Y_s^3 + \Psi'(Y_s,\un{Z}_s)$, we can then rewrite the identity of Proposition~\ref{p:test} as
\begin{multline}
	\label{e:test-torus}
	\frac{1}{p}\Ll(\|Y_t\|^p_{\tL^p} - \|Y_0\|^p_{\tL^p} \Rr) + \int_0^t \Ll[ (p-1)\Ll\|Y_s^{p-2} \Ll| \na Y_s \Rr|^2 \Rr\|_{\tL^1} + \|Y_s^{p+2}\|_{\tL^{1}} \Rr] \, \d s \\
	= \int_0^t \lan \Psi'(Y_s,\un{Z}_s),Y_s^{p-1} \ran  \, \d s.
\end{multline}
We will now show that the integrand on the \rhs\ is controlled by the integrand on the \lhs. For notational convenience, we write 
$$
A_s = \Ll\|Y_s^{p-2} \Ll| \na Y_s \Rr|^2 \Rr\|_{\tL^1}, \qquad B_s = \|Y_s^{p+2}\|_{\tL^{1}}.
$$
Decomposing $\Psi'$ as a sum of terms, we see that the first term we need to estimate is $\lan Y_s^2 Z_s, Y_s^{p-1} \ran = \lan Y_s^{p+1}, Z_s \ran$. By Proposition~\ref{p:dual}, we have
$$
\Ll| \lan Y_s^{p+1}, Z_s \ran  \Rr| \lesssim \|Y_s^{p+1}\|_{\tB^{\al,M}_{1,1}} \, \|Z_s\|_{\tBB{-\al}}.
$$
By Proposition~\ref{p:estimate}, we get
$$
\|Y_s^{p+1}\|_{\tB^{\al,M}_{1,1}} \lesssim \|Y_s^{p+1}\|_{\tL^1}^{1-\al} \, \|Y_s^p \, \na Y_s \|_{\tL^1}^\al + \|Y_s^{p+1}\|_{\tL^1}.
$$
and by the Cauchy-Schwarz inequality,
$$
\|Y_s^p \, \na Y_s \|^2_{\tL^1} \le \|Y_s^{p-2} \, |\na Y_s|^2\|_{\tL^{1}} \, \|Y_s^{p+2}\|_{\tL^{1}} = A_s \, B_s.
$$
Moreover, Jensen's inequality implies that $\|Y_s^{p+1}\|_{\tL^1} \lesssim B_s^{\frac{p+1}{p+2}}$ (where the implicit constant depends on $M$). Using also the fact that $\sup_{s \le T} \|Z_s\|_{\tBB{-\al}} < \infty$, we get
\begin{equation}
	\label{e:first-control}
	\Ll| \lan Y_s^{p+1}, Z_s \ran  \Rr| \lesssim A_s^{\frac{\al}{2}} \, B_s^{\frac{\al}{2} + (1-\al)\frac{p+1}{p+2} } + B_s^{\frac{p+1}{p+2} }.
\end{equation}
This term is controlled by an arbitrarily small constant times $(A_s + B_s)$. Indeed, we observe that since
$$
\frac{\al}{2} + \Ll( \frac{\al}{2} + (1-\al)\frac{p+1}{p+2} \Rr) < 1,
$$
one can define exponents $\ga_1 < 1, \ga_2 < 1$ such that 
$$
\frac{\al}{2\ga_1} + \frac{1}{\ga_2}\Ll( \frac{\al}{2} + (1-\al)\frac{p+1}{p+2} \Rr) = 1,
$$
and Young's inequality (the one for products, not for convolutions!) implies that
$$
\Ll| \lan Y_s^{p+1}, Z_s \ran  \Rr| \lesssim A_s^{\ga_1} + B_s^{\ga_2} + B_s^{\frac{p+1}{p+2} }.
$$
Since $\sup_{x \ge 0} (-x + x^{\ga}) < \infty$ for any $\ga < 1$, it is clear that the right-hand side above is bounded by $\frac{1}{10}(A_s + B_s)$ plus a universal constant.
	
Estimating the other terms coming from $\Psi'$ is similar. For instance, the second term we need to estimate is $\lan Y_s \Zdd_s, Y_s^{p-1} \ran = \lan Y_s^{p},  \Zdd_s\ran$. By the same reasoning, 
$$
\Ll| \lan Y_s^{p},  \Zdd_s \ran \Rr| \lesssim \|Y_s^p\|_{\tB^{\al,M}_{1,1}} \, \| \Zdd_s \|_{\tBB{-\al}},
$$
with 
\begin{eqnarray*}
	\|Y_s^p\|_{\tB^{\al,M}_{1,1}} & \lesssim &  \|Y_s^{p-1} \na Y_s\|_{\tL^1} + \|Y_s^p\|_{\tL^1} \\
	& \lesssim & \sqrt{A_s \|Y_s^p\|_{\tL^1} }   + \|Y_s^p\|_{\tL^1} \\
	& \lesssim & \sqrt{A_s B_s^{1-\frac{2}{p+2}} }   + B_s^{1-\frac{2}{p+2}  } \\
	& \lesssim & A_s^{1-\frac{1}{p+2}} + B_s^{1- \frac{1}{p+2}}  + B_s^{1-\frac{2}{p+2}  },
\end{eqnarray*}
where for simplicity, we used Remark~\ref{r:estimate} instead of the full strength of Proposition~\ref{p:estimate} in the first line, and then the Cauchy-Schwarz, Jensen and Young inequalities. 
Since $\|\un{Z}\|_{\td{\msc{Z}}^{M}_\infty}$ is finite (see \eqref{e:def:Z-norm}), we have $\|\Zdd_s\|_{\tBB{-\al}} \lesssim s^{-\al'}$, and thus
$$
\Ll| \lan Y_s^{p},  \Zdd_s \ran \Rr| \lesssim \Ll( A_s^{1-\frac{1}{p+2}} + B_s^{1- \frac{1}{p+2}}  + B_s^{1-\frac{2}{p+2}  } \Rr) s^{-\al'}.
$$
Observing that $\sup_{x \ge 0} (-x + a x^{\ga}) \lesssim a^{\frac{1}{1-\ga}}$, we obtain that
\begin{equation}
\label{e:control-thirda}
-\frac{A_s + B_s}{10} + \Ll| \lan Y_s^{p},  \Zdd_s \ran \Rr| \lesssim s^{-\al'(p+2)}.
\end{equation}
Similarly,
$$
 \Ll|  \lan Y_s^{p-1}, \Ztt_s \ran \Rr| \lesssim \Ll( A_s^{1-\frac{2}{p+2}} + B_s^{1- \frac{2}{p+2}  } + B_s^{1-\frac{3}{p+2}} \Rr) s^{-2\al'},
$$
so that 
\begin{equation}
\label{e:control-third}
-\frac{A_s + B_s}{10} + \Ll|  \lan Y_s^{p-1}, \Ztt_s \ran \Rr| \lesssim s^{-\al'(p+2)}.
\end{equation}
The remaining terms are treated in the same way. In short, we have shown that
$$
\frac{1}{p}\Ll(\|Y_t\|^p_{\tL^p} - \|Y_0\|^p_{\tL^p} \Rr)\lesssim \int_0^t s^{-\al'(p+2)} \, \d s.
$$
The integral on the right-hand side is finite since $\al'(p+2) < 1$, so the proof is complete.
\end{proof}
\begin{proof}[Proof of Theorem~\ref{t:global-torus}]
We assume that $\al'$ is sufficiently small that there exists a positive integer such that \eqref{e:conds-p} holds. In order to construct a global solution, we take $T^{\star}$ from Theorem~\ref{t:local} according to the a priori bound on the $\tL^p$ norm of the solution provided by Theorem~\ref{t:apriori-torus}, and then simply glue together local solutions until the time interval $[0,T]$ is covered. Uniqueness was already obtained in Step 3 of Theorem~\ref{t:local}.
\end{proof}

%
%
%
%
%
%
%
%
\section{A priori estimates}
\label{s:full1}
The goal of this section is to derive strong a priori estimates on solutions of \eqref{e:eqY}. This will enable us to show that ``periodised'' solutions of \eqref{e:eqY} constructed in the previous section all belong to suitable compact subsets of polynomially weighted Besov spaces. 

\medskip

Our setting bears similarities with that of the previous section. One main difference with what was done before is that we will work with polynomially weighted Besov spaces instead of ``periodised'' ones. We assume throughout that $\si > d = 2$, and introduce the slightly lighter notation
$$
\hB^{\al,\si}_{p} := \hB^{\al,\si}_{p,\infty}.
$$
Let $0 < \al < \al' < 1/6$, $\beta \in (1, 2)$, $T > 0$ and $p \ge 1$. 
For $\un{Z} = (Z,\Zdd,\Ztt)$ in the set
$$
\mcl{C}([0,T],\hB_{3p}^{-\al,\si}) \times \mcl{C}((0,T],\hB_{2p}^{-\al,\si}) \times \mcl{C}((0,T],\hB_{p}^{-\al,\si}),
$$
we write
\begin{equation}
	\label{e:def:Z-norm-full}
	\|\un{Z}\|_\ZZt
	= \sup_{0 < t \le T} \Ll(\|Z_t\|_{\hB_{3p}^{-\al,\si}} \, \vee \,  t^{\al'} \|\Zdd_t\|_{\hB_{2p}^{-\al,\si}} \, \vee \, t^{2\al'} \|\Ztt_t\|_{\hB_{p}^{-\al,\si}} \Rr).
\end{equation}
We let $\ZZt$ be the set of $\un{Z}$ such that this norm is finite.
%
%
%
For $Y \in \mcl{C}([0,T],\hB^{\be,\si}_{3p})$, $\un{Z} = (Z, \Zdd,\Ztt) \in \ZZt$ and $t \le T$, we write 
\begin{equation}
\label{e:def:Psi:full}
\Psi_t = \Psi(Y_t,\un{Z}_t) = -Y_t^3 - 3 Y_t^2 Z_t - 3 Y_t \Zdd_t - \Ztt_t + a(Y_t+Z_t) .
\end{equation}
The multiplicative inequalities imply that $\Psi(Y_t,\un{Z}_t)$ is well-defined for every $t > 0$, since we assume $p \ge 1$. Since we will ultimately only be interested in solutions with $Y_0 = 0$, we take advantage of this simplification right away and only consider solutions of
\begin{equation}
\label{e:eq-for-Y}
\Ll\{ 
\begin{array}{l}
\partial_t Y = \Delta Y + \Psi(Y,\un{Z}) \qquad \text{on } [0,T] \times \R^2, \\
Y(0,\cdot) = 0.
\end{array}
\Rr.
\end{equation}
Again, we will interpret this equation in the mild form. More precisely, we write $Y \in \wh S^{T,\si}_{3p}(\un{Z})$ if $Y \in \mcl{C}([0,T],\hB^{\be,\si}_{3p})$ and if for every $t \le T$,
$$
Y_t = \int_0^t e^{(t-s)\Delta} \Psi(Y_s,\un{Z}_s) \, \d s.
$$

We will assume throughout that
	\begin{align}
	& p \text{ is an even positive integer such that} \notag \\
& \al'(p+2) < 3/4 \quad \text{and} \quad \frac{\al' + \be}{2} + 3\al' < 1.
	\label{e:conds-p-full}
	\end{align}


Our strategy parallels that of the previous section in that, informally, we want to multiply \eqref{e:eq-for-Y} by $Y^{p-1}$ and integrate to get information on $\dr_t \|Y_t\|_{\hL^p}^p$. In order to carry the argument rigorously, we first need to assert that solutions of \eqref{e:eq-for-Y} have sufficient time regularity.

\begin{prop}[Time regularity of solutions]
\label{p:time-reg-full}
Let $\un{Z} \in \ZZt$ and $\ka < \be/2$. If $Y \in \wh S^{T,\si}_{3p}(\un{Z})$, then $Y$ is $\ka$-Hölder continuous as a function from $[0,T]$ to $\hL^{p}$. 
\end{prop}
The proof makes use of the following simple consequence of the multiplicative inequalities, which will be used again later on.
\begin{lem}[Estimating $\Psi$]
\label{l:estimpsi}
Let $\td{\al} > \al$. There exists $C < \infty$ such that
\begin{multline*}
C^{-1} \  \|\Psi(Y_s,\un{Z}_s)\|_{\hB^{-{\al},\si}_{p}} \le \|Y_s\|^3_{\hL^{3p}} + \|Y_s\|^2_{\hB^{\td{\al},\si}_{3p}} \, \|Z_s\|_{\hB^{-\al,\si}_{3p}} \\
+ \|Y_s\|_{\hB^{\td{\al},\si}_{3p}} \, \|\Zdd_s\|_{\hB^{-\al,\si}_{2p}} +  \|\Ztt_s\|_{\hB^{-\al,\si}_{p}} + \|Y_s\|_{\hL^{3p}} + \|Z_s\|_{\hB^{-\al,\si}_{3p}}.
\end{multline*}
\end{lem}
\begin{proof}
Recalling the decomposition of $\Psi$ in \eqref{e:def:Psi:full}, we begin by observing that, due to Remark~\ref{r:lp-embed},
\begin{equation*}
\|Y_s^3\|_{\hB^{-{{\al}},\si}_{p}} \lesssim \|Y_s^3\|_{\hL^{p}} = \|Y_s\|^3_{\hL^{3p}}.
\end{equation*}
We proceed to estimate
\begin{align*}
\|Y_s^2 Z_s\|_{\hB^{-\al,\si}_{p}} & \lesssim \|Y_s^2\|_{\hB^{\td{\al},\si}_{\frac{3p}{2}}} \, \| Z_s\|_{\hB^{-\al,\si}_{3p}} \\
& \lesssim \|Y_s\|^2_{\hB^{\td{\al},\si}_{3p}} \,\| Z_s\|_{\hB^{-{{\al}},\si}_{3p}},
\end{align*}
where we used Corollary~\ref{c:multipl2} in the first step and Corollary~\ref{c:multipl1} in the second one. Similarly,
\begin{align*}
\|Y_s \Zdd_s\|_{\hB^{-{\al},\si}_{p}} & \lesssim \|Y_s\|_{\hB^{\td \al,\si}_{2p}} \, \| \Zdd_s\|_{\hB^{-{\al},\si}_{2p}} \\
& \lesssim   \|Y_s\|_{\hB^{{\al},\si}_{3p}} \,\| \Zdd_s\|_{\hB^{-{\td{\al}},\si}_{2p}},
\end{align*}
where we used Remark~\ref{r:besov-mu} [or \eqref{e:r2-poly}] in the last step. For the same reasons,
$$
\| \Ztt_s + Y_s + Z_s \|_{\hB^{-{\al},\si}_{p}} \lesssim  \|\Ztt_s\|_{\hB^{{-\al},\si}_{p}} + \|Y_s\|_{\hL^{3p}} + \|Z_s\|_{\hB^{-\al,\si}_{3p}} .
$$
\end{proof}
\begin{proof}[Proof of Proposition~\ref{p:time-reg-full}]
As before, we only discuss Hölder regularity at time $0$. 
By Lemma~\ref{l:estimpsi} and since $\un{Z} \in \ZZt$ and $Y \in \mcl{C}([0,T],\hB^{\be,\si}_{3p})$ (with $\be \ge 1$), we have
\begin{equation}
\label{e:estimPsi}
\|\Psi_s\|_{\hB^{-{\al},\si}_{p}} = \|\Psi(Y_s,\un{Z}_s)\|_{\hB^{-{\al},\si}_{p}}   \lesssim s^{-2\al'}
\end{equation}
(where the implicit constant depends on $Y$ and $\un Z$). Hence, by Proposition~\ref{p:smooth-besov},
\begin{eqnarray*}
\Ll\| \int_0^t e^{(t-s) \Delta} \Psi_s \, \d s \Rr\|_{\hB^{\al,\si}_{p}} & \lesssim & \int_0^t (t-s)^{-\al} \|\Psi_s \|_{\hB^{-\al,\si}_{p}} \, \d s \\
& \lesssim & \int_0^t (t-s)^{-\al} s^{-2\al'} \, \d s \\
& \lesssim & t^{1-3\al'},
\end{eqnarray*}
since $\al' > \al$.
The conclusion then follows by Remark~\ref{r:embed-lp} since $1-3\al' > \be/2$ by \eqref{e:conds-p-full}.
\end{proof}
\begin{rem}
As the proof reveals, the index of Hölder regularity in Proposition~\ref{p:time-reg-full} could be improved (for instance, one can choose $\ka = \be/2$). We prefer to stick to the present version, because it corresponds to the general statement with initial condition $Y_0 \in \hB_{3p}^{\be,\si}$ (compare with Proposition~\ref{p:time-reg}).
\end{rem}
The fact that mild solutions are weak solutions remains valid in the present context (the proof being the same as that of the first step of Proposition~\ref{p:weak1}). We denote by $(\, \cdot \, , \, \cdot \,)$ the scalar product in the unweighted $L^2$ space. 
\begin{prop}[A mild solution is a weak solution]
	If $\un{Z} \in \ZZt$ and $Y \in \wh {S}_{3p}^{T,\si}(\un{Z})$, then for every $\phi \in C^\infty_c$ and $t \le T$,
\begin{equation}
\label{e:weak-eq-full}
	( Y_t,\phi ) = \int_0^t \Ll[ (Y_s,\Delta \phi )  + ( \Psi(Y_s,\un{Z}_s),\phi ) \Rr] \, \d s.
\end{equation}
	\label{p:weak-full}
\end{prop}
This can be upgraded to the following statement. We denote by $(\, \cdot \,, \, \cdot \,)_{\si}$ the scalar product in $\hL^2$, which can be extended to distributions outside of $\hL^2$ through Proposition~\ref{p:dual}. 
\begin{prop}[Weights and more general test functions]
\label{p:weak-f}
There exists $C$ such that the following holds. For every $\un{Z} \in \ZZt$,  $Y \in \wh S^{T,\si}_{3p}(\un{Z})$ and $\phi \in \hB^{1,\si}_{p'}$ (with $p'$ the conjugate exponent of $p$), 
\begin{equation}
\label{e:weak-eq-full-form}
	( Y_t,\phi )_{\si} = \int_0^t \Ll[- (\na Y_s,\na \phi )_{\si}  + ( \Psi(Y_s,\un{Z}_s),\phi )_{\si} \Rr] \, \d s + \msf{Err}(t),
\end{equation}
where the error term satisfies
\begin{equation}
\label{e:boundErr}
\Ll| \msf{Err}(t) \Rr| \le C  \int_0^t (|\na Y_s|, |\phi|)_{\si} \, \d s.
\end{equation}
\end{prop}
\begin{proof}
We first argue that \eqref{e:weak-eq-full-form} holds for smooth, compactly supported $\phi$, and then conclude by density.

\medskip 

\noindent \emph{Step 1.} For any $\phi \in C^\infty_c$,
$$
(Y_s,\Delta \phi) = - (\na Y_s,\na \phi )
$$ 
in the sense of distributions. In fact, since $Y \in \mcl{C}([0,T],\hB^{\be,\si}_{3p})$ with $\be > 1$, in particular $\na Y_s$ belongs to $\hL^{3p}$ (in the sense that each coordinate of $\na Y$ belongs to this space, see Proposition~\ref{p:derivatives}) and the right-hand side above can be interpreted as a space integral. We apply Proposition~\ref{p:weak-full} with $\phi \wh w_\si$ as a test function [recall that $\wh w_\si$ was defined in \eqref{e:def:wh-w}]. We observe that the gradient of this function is
$$
(\na \phi) \wh w_\si -\underbrace{\frac{\si  \na(|\cdot|_*)}{|\cdot|_*}}_{=:G} \phi \wh w_\si.
$$
We have established \eqref{e:weak-eq-full-form} for $\phi \in C^\infty_c$ and with
\begin{equation}
\label{e:defErr}
\msf{Err}(t) =  \int_0^t (\na Y_s, G\phi)_{\si} \, \d s.
\end{equation}

\medskip 

\noindent \emph{Step 2.} We now conclude by density.
In view of Proposition~\ref{p:derivatives}, we have $\na Y \in \mcl{C}([0,T],\hB^{\be-1,\si}_{3p})$ (in the sense that each coordinate of $\na Y$ belongs to this space). Hence, for smooth, compactly supported $\phi$ and $\td{\phi}$, 
$$
\int_0^t \Ll| (\na Y_s, \na \phi)_{\si} - (\na Y_s, \na \td{\phi})_{\si} \Rr| \, \d s 
 \lesssim \|\na \phi - \na \td{\phi}\|_{\hB^{1-\be,\si}_{p'',1}} \lesssim \|\phi - \td{\phi}\|_{\hB^{2-\be,\si}_{p'',1}}
$$
by Propositions~\ref{p:dual} and \ref{p:derivatives}, where $p''$ is the conjugate exponent of $3p$. 
Similarly, we infer from \eqref{e:estimPsi} that
$$
\int_0^t \Ll| (\Psi_s,  \phi)_{\si} - (\Psi_s,\td{\phi})_{\si} \Rr| \, \d s \lesssim \|\phi - \td{\phi}\|_{\hB^{\al,\si}_{p',1}}.
$$
The other terms can be treated similarly (using also the fact that $G$ is uniformly bounded). Since $\al < 1$, $2-\be < 1$ and $p'' \le p'$, we can use Remark~\ref{r:besov-mu} [cf.\ also \eqref{e:r2-poly}] and obtain by density that for every $\phi \in \hB^{1,\si}_{p'}$, 
$$
	( Y_t,\phi )_{\si} = \int_0^t \Ll[- (\na Y_s,\na \phi )_{\si}  + ( \Psi_s,\phi )_{\si} \Rr] \, \d s + \msf{Err}(t)
$$
with $\msf{Err}(t)$ given by \eqref{e:defErr}. The bound \eqref{e:boundErr} follows from the fact that $G$ is bounded.
\end{proof}
\begin{prop}[Testing against $Y_t^{p-1}$]
	Recall that we assume \eqref{e:conds-p-full}. There exists $C < \infty$ such that if $\un{Z} \in \ZZt$ and $Y \in \wh{S}_{3p}^{T,\si}(\un{Z})$, then
	\label{p:test-full}
	$$
	\frac{1}{p}\|Y_t\|^p_{\hL^p}  = \int_0^t \Ll[ -(p-1)( \na Y_s, Y_s^{p-2}\na Y_s)_{\si} + ( \Psi(Y_s,\un{Z}_s),Y_s^{p-1} )_{\si} \Rr] \, \d s + \msf{Err}_p(t),
$$
with 
$$
\Ll|\msf{Err}_p(t) \Rr| \le C \int_0^t (|\na Y_s|,|Y_s|^{p-1})_{\si} \, \d s.
$$
\end{prop}
\begin{proof}
We first check that we can use Proposition~\ref{p:weak-f} with $\phi = Y_t^{p-1}$ for a fixed $t$, that is, we check that $Y_t^{p-1} \in \hB^{1,\si}_{p'}$, where $p'$ is the conjugate exponent of $p$. By the multiplicative inequality (Corollary~\ref{c:multipl1}),
$$
\|Y_t^{p-1}\|_{\hB^{1,\si}_{p'}} \lesssim \|Y_t\|^{p-1}_{\hB^{1,\si}_{(p-1)p'}} = \|Y_t\|^{p-1}_{\hB^{1,\si}_{p}}.
$$ 
Since $Y_t \in \hB^{\be,\si}_{p}$ and $\be > 1$, we have indeed $Y_t^{p-1} \in \hB^{1,\si}_{p'}$. As a consequence, we can proceed as in the beginning of the proof of Proposition~\ref{p:test}, i.e.\ write
$$
\|Y_v\|^p_{\hL^p} - \|Y_u\|^p_{\hL^p} - ( Y_v, Y_v^{p-1} - Y_u^{p-1} )_\si
 =  ( Y_v, Y_u^{p-1} )_\si - ( Y_u,Y_u^{p-1} )_\si,
$$
and then use Proposition~\ref{p:weak-f} to obtain that for any subdivision $\un{t} = (t_0,\ldots,t_n)$ with $0 = t_0 \le \cdots \le t_n = t$,
$$
	\|Y_t\|^p_{\hL^p} - \|Y_0\|^p_{\hL^p} - \mfk{S}(\un{t}) = \mfk{I}_1(\un{t}) + \mfk{I}_2(\un{t})  + \msf{Err}(\un{t}),
$$
where
$$
\mfk{S}(\un{t}) = \sum_{i = 0}^{n-1} ( Y_{t_{i+1}}, Y_{t_{i+1}}^{p-1} - Y_{t_{i}}^{p-1} )_\si,
$$
$$
\mfk{I}_1(\un{t}) = -\sum_{i = 0}^{n-1} \int_{t_i}^{t_{i+1}} ( \na Y_s, \na(Y_{t_i}^{p-1}) )_\si  \, \d s,
$$
$$
\mfk{I}_2(\un{t}) = \sum_{i = 0}^{n-1} \int_{t_i}^{t_{i+1}}  ( \Psi_s,Y_{t_i}^{p-1} )_\si  \, \d s
$$
and 
$$
\Ll|\msf{Err}(\un{t}) \Rr| \le C \sum_{i = 0}^{n-1} \int_{t_i}^{t_{i+1}}  (|\na Y_s|, |Y_{t_i}|^{p-1})_\si  \, \d s.
$$
By Proposition~\ref{p:derivatives}, the function $t \mapsto \na Y_t$ belongs to $\mcl{C}([0,T],\hB^{\be-1,\si}_{3p})$. Moreover, a direct adaptation of the argument at the beginning of this proof shows that $t \mapsto Y_t^{p-1}$ belongs to $\mcl{C}([0,T],\hB^{1,\si}_{p'})$, and as a consequence, the function $t \mapsto \na Y_t^{p-1}$ belongs to $\mcl{C}([0,T],\hB^{0,\si}_{p'})$. By Proposition~\ref{p:dual} (and since $\be > 1$), this suffices to ensure that as the subdivision $\un{t}$ gets finer and finer,
$$
\mfk{I}_1(\un{t}) \to -\int_0^t ( \na Y_s, \na(Y_{s}^{p-1}) )_\si \, \d s ,
$$
and similarly,
$$
\sum_{i = 0}^{n-1} \int_{t_i}^{t_{i+1}}  (|\na Y_s|, |Y_{t_i}|^{p-1})_\si  \, \d s \to \int_0^t (|\na Y_s|, |Y_{s}|^{p-1})_\si  \, \d s.
$$
Now, using \eqref{e:estimPsi} and the fact that $t \mapsto \Psi_t$ is in $\mcl{C}((0,T],\hB^{-\al,\si}_{p})$, together with the fact already seen that $t \mapsto Y_t^{p-1}$ is in $\mcl C ([0,T],\hB^{1,\si}_{p'})$, we obtain that as the subdivision gets finer and finer,
$$
\mfk{I}_2(\un{t}) \to \int_0^t  ( \Psi_s,Y_{s}^{p-1} )_\si \, \d s.
$$
There remains to check that as the subdivision $\un{t}$ gets finer and finer,
$$
\mfk{S}(\un{t}) \to \frac{p-1}{p} \Ll( 	\|Y_t\|^p_{\hL^p} - \|Y_0\|^p_{\hL^p} \Rr).
$$
The proof is the same as that of the similar statement \eqref{e:lim-mfkS} in the proof of Proposition~\ref{p:test}. We only need to verify that the function $t \mapsto Y_t$ is $\ka$-Hölder continuous as a function from $[0,T]$ to $\hL^p$, for some $\ka > 1/2$. This is guaranteed by Proposition~\ref{p:time-reg-full}.
\end{proof}
\begin{prop}[A priori estimate in $\hL^p$]
\label{p:apriori}
Recall that we assume \eqref{e:conds-p-full}. For every $K < \infty$, there exists $C < \infty$ such that if $\un{Z} \in \ZZt$ satisfies $\|\un{Z}\|_\ZZt \le K$, $Y \in \wh S^{T,\si}_{3p}(\un{Z})$ and $t \le T$, then
\begin{equation}
	\label{e:apriori}
	\|Y_t\|^p_{\hL^p} + \int_0^t \Ll\|Y_s^{p-2}|\na Y_s|^2\Rr\|_{\hL^1}  \d s  \le C.
\end{equation}
\end{prop}
\begin{proof}
As in the proof of Theorem~\ref{t:apriori-torus}, the starting point is to decompose $\Psi(Y_s,\un{Z}_s)$ into $-Y_s^3 + \Psi'_s$, so that the identity derived in Proposition~\ref{p:test-full} becomes
\begin{multline}
\label{e:energy}
\frac{1}{p}\|Y_t\|^p_{\hL^p}  + \int_0^t \Ll[ (p-1) \Ll\|Y_s^{p-2}|\na Y_s|^2\Rr\|_{\hL^1} +  \|Y_s^{p+2}\|_{\hL^1} \Rr] \, \d s \\
\le \int_0^t \Ll[(\Psi_s',Y_s^{p-1} )_{\si} + C (|\na Y_s|,|Y_s|^{p-1})_{\si} \Rr] \, \d s,
\end{multline}
We write 
$$
A_s = \Ll\|Y_s^{p-2} \Ll| \na Y_s \Rr|^2 \Rr\|_{\hL^1}, \qquad B_s = \|Y_s^{p+2}\|_{\hL^1}.
$$
We will now show that the integrand in the \rhs\ of \eqref{e:energy} is bounded by a linear combination of terms of the form $ C \|\un{Z}\|_\ZZt^{\ga_1} \, s^{-\ga_2}   \, A_s^{\ga_3} \, B_s^{\ga_4}$. We will then summarize all the terms into Table~\ref{t}, analyse the value of the exponents $\ga_1, \ldots, \ga_4$, and conclude. The integrand in the \rhs\ of \eqref{e:energy} is a sum of two terms. We begin with the second one:
$$
(|\na Y_s|,|Y_s|^{p-1})_{\si} \le A_s^{1/2} \, \|Y_s^{p}\|^{1/2}_{\hL^1} \lesssim A_s^{1/2} \, B_s^{p/2(p+2)},
$$
by the Cauchy-Schwarz and Hölder's inequalities.
This estimate is reported on the first line of Table~\ref{t}.

We now move to the study of $(\Psi_s',Y_s^{p-1} )_{\si}$. We decompose $\Psi_s'$ into a sum of terms that we will analyse in turn: 
$$
\Psi' = -3 Y^2 Z - 3 Y \Zdd - \Ztt + a Z + a Y.
$$

$$
\Ll|(Y_s^2 Z_s,Y_s^{p-1})_{\si} \Rr|= \Ll|(Y_s^{p+1},Z_s)_{\si} \Rr| \lesssim \|Y_s^{p+1}\|_{\hB^{\al,\si}_{p'',1}} \ \|Z_s\|_{\hB^{-\al,\si}_{3p}},
$$
where we used Proposition~\ref{p:dual} for the inequality, and where $p''$ is the conjugate exponent of $3p$. On the one hand,
$$
\|Z_s\|_{\hB^{-\al,\si}_{3p}} \le \|\un{Z}\|_{\ZZt}.
$$
On the other hand, we would like to use Proposition~\ref{p:estimate} to estimate $\|Y_s^{p+1}\|_{\hB^{\al,\si}_{p'',1}}$. Compared with the proof of Theorem~\ref{t:apriori-torus}, a new difficulty appears since Proposition~\ref{p:estimate} gives an estimate of the norm in Besov spaces with lower indices equal to $1$, while we have $p'' > 1$ here. We solve this difficulty by appealing to an interpolation inequality, for which we now introduce some notation. Let 
\begin{equation}
\label{e:def:nu}
q = \frac{p+2}{p+1} \quad \text{ and } \quad \nu = \frac{p + 2}{3p} \in (0,1), 
\end{equation}
so that 
\begin{equation}
\label{e:param-interpol2}
\frac{1}{p''} = \frac{1-\nu}{1}+ \frac{\nu}{q}.
\end{equation}
Note that $1-\nu \ge 1/3$, while $\al <1/4$ since we assume $\al < \al'$ and $(p+2) \al' < 1$. Hence, there exists $\al_1 < 0$ such that
\begin{equation}
\label{e:param-interpol1}
\al = \frac{3(1-\nu)}{4} + \nu \al_1.
\end{equation}
The interpolation inequality (Proposition~\ref{p:interpol}) now reads
$$
\|Y_s^{p+1}\|_{\hB^{\al,\si}_{p'',1}} \le \|Y_s^{p+1}\|^{1-\nu}_{\hB^{3/4,\si}_{1,1}} \ \|Y_s^{p+1}\|^\nu_{\hB^{\al_1,\si}_{q,1}}.
$$
Since $\al_1 < 0$, by Remarks~\ref{r:embed-q} and~\ref{r:lp-embed},
$$
\|Y_s^{p+1}\|_{\hB^{\al_1,\si}_{q,1}} \lesssim \|Y_s^{p+1}\|_{\hL^{q}} =  B_s^{\frac{p+1}{p+2} }  ,
$$
while by Proposition~\ref{p:estimate},
$$
\|Y_s^{p+1}\|_{\hB^{3/4,\si}_{1,1}} \lesssim \|Y_s^{p} |\na Y_s| \|^{3/4}_{\hL^1} \, \|Y_s^{p+1}\|_{\hL^1}^{1/4} +  \|Y_s^{p+1}\|_{\hL^1}.
$$
By the Cauchy-Schwarz inequality,
$$
\|Y_s^{p} |\na Y_s| \|_{\hL^1} \le \sqrt{A_s B_s},
$$
while by Hölder's inequality,
$$
\|Y_s^{p+1}\|_{\hL^1} \lesssim B_s^{\frac{p+1}{p+2} }  .
$$
To sum up, we have shown that
\begin{align*}
\Ll|(Y_s^2 Z_s,Y_s^{p-1})_{\si} \Rr| & \lesssim \|\un{Z}\|_\ZZt \, B_s^{\nu \frac{p+1}{p+2}}\Ll(  A_s^{\frac38} \, B_s^{\frac38 + \frac14 \frac{p+1}{p+2}  } + B_s^{\frac{p+1}{p+2} }  \Rr)^{1-\nu} \\
& \lesssim \|\un{Z}\|_\ZZt    \, A_s^{\frac{3(1-\nu)}{8}} \, B_s^{\frac{3(1-\nu)}{8} + \frac{p+1}{p+2}\frac{1 +3\nu}{4} }  \\
& \qquad +   \|\un{Z}\|_\ZZ   \, B_s^{\frac{p+1}{p+2}}.
\end{align*}
We summarize this computation by two lines of Table~\ref{t}. 

We now turn to the evaluation of
$$
\Ll|(Y_s \Zdd_s,Y_s^{p-1})_{\si} \Rr| \lesssim \|Y_s^p\|_{\hB^{\al,\si}_{\ov p,1}} \ \|\Zdd_s\|_{\hB^{-\al,\si}_{2p}}, 
$$
where $\ov p$ denotes the conjugate exponent of $2p$. 
We first note that
$$
\|\Zdd_s\|_{\hB^{-\al,\si}_{2p}} \le \|\un{Z}\|_{\ZZt} \, s^{-\al'}.
$$
We then prepare the ground for the adequate interpolation inequality by setting
$$
\td{q} = \frac{p+2}{p} \quad \text{ and } \quad \td{\nu} = \frac{p+2}{4p},
$$
so that
$$
\frac{1}{\ov p} = \frac{1-\td{\nu}}{1} + \frac{\td{\nu}}{\td{q}}.
$$
Since $1-\td{\nu} \ge 1/2$ and $\al < 1/4$, there exists $\td{\al}_1 < 0$ such that
$$
\al = \frac{1-\td{\nu}}{2} + \td{\nu}\al_1,
$$
and the interpolation inequality is
$$
\|Y_s^{p}\|_{\hB^{\al,\si}_{\ov p,1}} \le \|Y_s^{p}\|^{1-\td{\nu}}_{\hB^{1/2,\si}_{1,1}} \ \|Y_s^{p}\|^{\td{\nu}}_{\hB^{\td{\al}_1,\si}_{\td{q},1}}.
$$
We have
$$
\|Y_s^{p}\|_{\hB^{\td{\al}_1,\si}_{\td{q}}} \lesssim B_s^{\frac{p}{p+2} }, 
$$
while for simplicity, we can now use Proposition~\ref{p:estimate} in the form given by Remark~\ref{r:estimate}:
$$
\|Y_s^{p}\|_{\hB^{1/2,\si}_{1,1}} \lesssim \|Y_s^{p-1} |\na Y_s| \|_{\hL^1} + \|Y_s^p\|_{\hL^1}.
$$
The estimation is completed by the following two observations:
$$
\|Y_s^{p-1} |\na Y_s| \|_{\hL^1} \le \sqrt{A_s \, \|Y_s^p\|_{\hL^1}},
$$
$$
\|Y_s^p\|_{\hL^1} \lesssim  \, B_s^{\frac{p}{p+2} }. 
$$
To sum up, we have shown the estimate
$$
\Ll|(Y_s \Zdd_s,Y_s^{p-1})_{\si} \Rr| \lesssim \|\un{Z}\|_{\ZZt} \, s^{-\al'} \, B_s^{\td{\nu} \frac{p}{p+2} } \Ll( A_s^{1/2} \, B_s^{\frac{p}{2(p+2)}}  +   B_s^{\frac{p}{p+2} } \Rr)^{1-\td{\nu}}, 
$$
which is summarized on the corresponding two lines of Table~\ref{t}.

The same analysis can be performed for
$$
\Ll|( \Ztt_s,Y_s^{p-1})_{\si} \Rr|,
$$
by setting 
$$
\ov{q} = \frac{p+2}{p-1}, \qquad \ov{\nu} = \frac{p+2}{3p},
$$
so that 
$$
\frac 1 {p'} = \frac{1-\ov \nu}{1} + \frac{\ov \nu}{\ov q},
$$
and proceeding as before. This leads to the estimate
$$
\Ll|( \Ztt_s,Y_s^{p-1})_{\si} \Rr| \lesssim \|\un{Z}\|_{\ZZt} \, s^{-2\al'} \, B_s^{\ov{\nu} \frac{p-1}{p+2} } \Ll(  A_s^{1/2} \, B_s^{\frac{p-1}{2(p+2)}}  +  B_s^{\frac{p-1}{p+2} } \Rr)^{1-\ov{\nu}},
$$
which we report again in Table~\ref{t}. The same argument also leads to
$$
\Ll| (Z_s, Y_s^{p-1})_{\si} \Rr| \lesssim \|\un{Z}\|_{\ZZt} \,  B_s^{\ov{\nu} \frac{p-1}{p+2} } \Ll(  A_s^{1/2} \, B_s^{\frac{p-1}{2(p+2)}}  + B_s^{\frac{p-1}{p+2} } \Rr)^{1-\ov{\nu}},
$$
whose contribution can be absorbed into that of the previous term (so we do not report it in the table). Finally, we have
$$
 (Y_s, Y_s^{p-1})_{\si} = \|Y_s^p\|_{\hL^1} \lesssim B_s^{\frac{p}{p+2}} , 
$$
and we have finished to fill the table.

{\small
\begin{table}
\centering
\renewcommand{\arraystretch}{1.5}
\begin{tabular}{cccccc}
\toprule
 & $\ga_1$ & $\ga_2$ & $\ga_3$ & $\ga_4$  & $1-\ga_3 - \ga_4$ 
\\
\midrule
$(|\na Y_s|,|Y_s|^{p-1})_{\si}$   & $0$ & $0$ & $\frac12$    & $\frac{p}{2(p+2)} $ & $\frac1{p+2}$ 
\\
$\Ll|(Y_s^2 Z_s,Y_s^{p-1})_{\si} \Rr|$  & $1$ & $0$ & $\frac{3(1-\nu)}{8}  $ & $\frac{3(1-\nu)}{8} + \frac{p+1}{p+2}\frac{1 +3\nu}{4}  $ & $> \frac{1}{4(p+2)}$
\\
	  & $1$ & $0$ & $0$ & $\frac{p+1}{p+2}$ & $  \frac1{p+2}$
\\
$\Ll|(Y_s \Zdd_s,Y_s^{p-1})_{\si} \Rr|$ & $1$ & $\al'$ & $\frac{1-\td{\nu}}{2} $ & $ \frac{p}{p+2} \frac{1+\td{\nu}}{2}      $ & $> \frac1{p+2}$ 
\\
	 & $1$ & $\al'$ & $0 $ & $ \frac{p}{p+2}   $ & $  \frac2{p+2}$  
\\
$\Ll|( \Ztt_s,Y_s^{p-1})_{\si} \Rr|$  & $1$ & $2\al'$ & $\frac{1-\ov{\nu}}{2} $ & $ \frac{p-1}{p+2} \frac{1+\ov{\nu}}{2}     $ & $> \frac{3}{2(p+2)}$ 
\\
	 & $1$ & $2\al'$ & $0 $ & $ \frac{p-1}{p+2}    $ & $  \frac3{p+2}$ 
\\
$ (Y_s, Y_s^{p-1})_{\si} $  & $0$ & $0$ & $0 $ & $ \frac{p}{p+2}   $ & $  \frac2{p+2}$   
\\
\bottomrule
\end{tabular}
\bigskip
\caption{Each term in the first column is bounded by a sum of terms of the form $C \,  \|\un{Z}\|_\ZZt^{\ga_1} \, s^{-\ga_2}   \, A_s^{\ga_3} \, B_s^{\ga_4}$ for the values of $\ga_1, \ldots, \ga_4$ displayed on the corresponding lines. Recall that $\nu$, $\td{\nu}$, $\ov{\nu} \in (0,1)$. }
\label{t}
\end{table}
}

In order to conclude the proof, we have to show how to control a term of the form $C \, \|\un{Z}\|_\ZZt^{\ga_1} \, s^{-\ga_2}   \, A_s^{\ga_3} \, B_s^{\ga_4}$ by the terms $A_s$ and $B_s$ that appear on the \lhs\ of \eqref{e:energy}. We note first that we always have $\td{\ga} := \ga_3 + \ga_4 < 1$. Moreover, Young's inequality for products ensures that
$$
A_s^{\ga_3} \, B_s^{\ga_4} \le A_s^{\td{\ga}} + B_s^{\td{\ga}}.
$$
Finally, we observe that $\sup_{x \ge 0} (-x + a x^{\td{\ga}}) \lesssim a^{\frac{1}{1-\td{\ga}}}$, and as a consequence, 
$$
-\frac{A_s}{10} + C \,  \|\un{Z}\|_\ZZt^{\ga_1} \, s^{-\ga_2}   \, A_s^{\td{\ga}} \lesssim \Ll(  \|\un{Z}\|_\ZZt^{\ga_1} \, s^{-\ga_2}\Rr)^{\frac{1}{1-\td{\ga}}},
$$
and similarly with $A_s$ replaced by $B_s$. Hence, it follows from \eqref{e:energy} that
$$
\|Y_t\|^p_{\hL^p}  + \int_0^t \Ll\|Y_s^{p-2}|\na Y_s|^2\Rr\|_{\hL^1} \, \d s \lesssim \sum \int_0^t \Ll(  \|\un{Z}\|_\ZZt^{\ga_1} \, s^{-\ga_2}\Rr)^{\frac{1}{1-\td{\ga}}} \, \d s,
$$
where the sum is over all $\ga_1, \ga_2, \td{\ga} = \ga_3 + \ga_4$ described in Table~\ref{t}. In order for the integral to be finite, we need to ensure that $\ga_2 < 1-\td{\ga}$ in all cases. This is granted by the assumption that $\al' (p+2) < 3/4$ (the critical case being the third line from the bottom in the table). 
\end{proof}
We now upgrade the $\hL^p$ estimate to an estimate on a Besov norm of the solution. We do it in two steps: in the proposition below, we derive a time-averaged estimate. 

\begin{prop}[Weak a priori estimate in Besov spaces]
\label{p:apriori-besov}
Let $p \ge 4$ be an even positive integer such that \eqref{e:conds-p-full} holds, and let $\td{\al}$ be such that
\begin{equation}
\label{e:def:tdal}
\al \le \td{\al} \le \al + \frac{1}{p}.
\end{equation}
For every $K < \infty$, there exists $C < \infty$ such that if $\un{Z} \in \ZZt$ satisfies $\|\un{Z}\|_{\ZZ} \le K$ and $Y \in \wh S^{T,\si}_{3p}(\un{Z})$, then
$$
\int_0^T \|Y_s\|^{p-1}_{\B^{\td{\al},\si}_{p/3}} \, \d s \le C.
$$
\end{prop}
\begin{proof}
Let $\nu = 2/(p-1) \in (0,1)$. Since $p\td{\al} \le \al p + 1 < \al'(p+2) + 1 < 2$, we have $\td{\al} < \nu$. Hence, there exists $\al_0 < 0$ and $\al_1 \in (0,1)$ such that
$$
\td{\al} = (1-\nu) \al_0 + \nu \al_1.
$$
By the definition of $\nu$, we also have
$$
\frac{3}{p} = \frac{1-\nu}{p} + \frac{\nu}{1}.
$$
By the interpolation inequality (Proposition~\ref{p:interpol}),
\begin{equation}
\label{e:ap-be-1}
\|Y_s\|_{\hB^{\td{\al},\si}_{p/3}} \le \|Y_s\|^{1-\nu}_{\hB^{\al_0,\si}_{p}} \ \|Y_s\|^\nu_{\hB^{\al_1,\si}_{1}}.
\end{equation}
Since $\al_0 < 0$, Remark~\ref{r:lp-embed} and Proposition~\ref{p:apriori} ensure that 
\begin{equation}
\label{e:ap-be-2}
\|Y_s\|_{\hB^{\al_0,\si}_{p}} \lesssim \|Y_s\|_{\hL^p} \lesssim 1
\end{equation}
(where the implicit constant depends in particular on $K$). 
Moreover, by Remark~\ref{r:embed-q} and Proposition~\ref{p:estimate} (which we only use in the weaker form provided by Remark~\ref{r:estimate} here), we have
$$
\|Y_s\|_{\hB^{\al_1,\si}_{1}} \lesssim \|Y_s\|_{\hL^1} + \|\na Y_s\|_{\hL^1} \lesssim \|Y_s\|_{\hL^2} + \|\na Y_s\|_{\hL^2}.
$$
Using Proposition~\ref{p:apriori} with $p = 2$ (noting that \eqref{e:conds-p-full} is clearly satisfied for $p = 2$), we obtain that
\begin{equation}
\label{e:ap-be-3}
\int_0^t \|Y_s\|_{\hB^{\al_1,\si}_{1}}^2 \, \d s \lesssim 1.
\end{equation}
Combining \eqref{e:ap-be-1}, \eqref{e:ap-be-2} and \eqref{e:ap-be-3}, we arrive at
$$
\int_0^t \|Y_s\|^{2/\nu}_{\hB^{\td{\al},\si}_{p/3}} \, \d s \lesssim 1,
$$
which is the announced result.
\end{proof}

We now conclude with a pointwise-in-time estimate of the Besov norm of solutions.

\begin{prop}[Strong a priori estimate in Besov spaces]
Let $p \ge 10$ be an even positive integer such that \eqref{e:conds-p-full} holds, and assume furthermore that
\begin{equation}
\label{e:cond:tdbe}
\frac{p-1}{p-4} \  \frac{\al+\be}{2} < 1.
\end{equation}
For every $K < \infty$, there exists $C < \infty$ such that if $\un{Z} \in \ZZt$ satisfies $\|\un{Z}\|_\ZZt \le K$ and $Y \in \wh S^{T,\si}_{3p}(\un{Z})$, then
$$
\sup_{t \le T} \|Y_t\|_{\hB^{{\be},\si}_{p/9}} \le C. 
$$
\label{p:pointwise-in-time-apriori}
\end{prop}

\begin{proof}
Let $\td{\al} \in (\al,\al+1/p]$. Recall from Lemma~\ref{l:estimpsi} that
\begin{multline*}
 \|\Psi_s\|_{\hB^{-{\al},\si}_{p/9}}  \lesssim \|Y_s\|^3_{\hB^{\td{\al},\si}_{p/3}} + \|Y_s\|^2_{\hB^{\td{\al},\si}_{p/3}} \ \|Z_s\|_{\hB^{-\al,\si}_{p/3}} + \|Y_s\|_{\hB^{\td{\al},\si}_{p/3}} \ \|\Zdd_s\|_{\hB^{-\al,\si}_{2p/9}} \\
+ \|\Ztt_s\|_{\hB^{-\al,\si}_{p/9}} + \|Y_s\|_{\hB^{\td{\al},\si}_{p/3}} + \|Z_s\|_{\hB^{-\al,\si}_{p/3}}.
\end{multline*}
From the definition of $\|\un{Z}\|_\ZZt$ in \eqref{e:def:Z-norm-full} and Remark~\ref{r:besov-mu} [or \eqref{e:r2-poly}], it follows that
$$
\|\Psi_s\|_{\hB^{-{\al},\si}_{p/9}}  \lesssim \Ll( \|Y_s\|^3_{\hB^{\td{\al},\si}_{p/3}} + s^{-2\al'} \Rr)  \Ll( \|\un{Z}\|_{\ZZt} + 1 \Rr) ,
$$
so that
\begin{equation}
\label{e:someest}
\|\Psi_s\|_{\hB^{-{\al},\si}_{p/9}}  \lesssim \Ll( \|Y_s\|^3_{\hB^{\td{\al},\si}_{p/3}} + s^{-2\al'} \Rr)  
\end{equation}
(where the implicit constant depends in particular on $K$).
By the definition of $Y \in \wh S^{T,\si}_{3p}(\un{Z})$,
$$
Y_t = \int_0^t e^{(t-s) \Delta} \Psi_s \, \d s,
$$
so by Proposition~\ref{p:smooth-besov} and \eqref{e:someest},
\begin{align*}
\|Y_t\|_{\hB^{{\be},\si}_{p/9}} & \lesssim \int_0^t (t-s)^{-\frac{\al+\be}{2}} \Ll( \|Y_s\|^3_{\hB^{\td{\al},\si}_{p/3}} + s^{-2\al'}\Rr) \ \d s \\
& \lesssim 1 + \int_0^t (t-s)^{-\frac{\al+\be}{2}} \|Y_s\|^3_{\hB^{\td{\al},\si}_{p/3}}  \ \d s ,
\end{align*}
since $\frac{\al + \be}{2} + 2\al' < 1$ by \eqref{e:conds-p-full}. By Hölder's inequality, since we assume
\begin{equation*}
\frac{1}{1-\frac{3}{p-1}  } \ \frac{\al+\be}{2} < 1,
\end{equation*}
it follows that the remaining integral is smaller than a constant times
$$
\Ll(\int_0^t \|Y_s\|^{p-1}_{\hB^{\td{\al},\si}_{p/3}} \ \d s\Rr)^{3/(p-1)},
$$
so Proposition~\ref{p:apriori-besov} enables us to conclude.
\end{proof}

\begin{rem}
\label{r:uniformly-equi}
Informally, we started from
$$
\partial_t Y = \Delta Y + \Psi(Y, \un Z),
$$
multiplied by $Y^{p-1}$ and integrated to get an estimate on the $L^p$ norm, which we then upgraded to obtain Proposition~\ref{p:pointwise-in-time-apriori}. A similar strategy enables to find an a priori estimate on the modulus of continuity of the solutions. Indeed, for any fixed $s \in [0,T]$, we write an equation for $\td{Y}_t := Y_t - Y_s$ ($s \le t \le T$), test it against $\td{Y}_t^{p-1}$, and proceed as before. We obtain that under the assumptions of Proposition~\ref{p:pointwise-in-time-apriori}, the set
$$
\{Y \in S^{T,\si}_{3p}(\un{Z}), \ \|\un{Z}\|_\ZZt \le K \}
$$
is a family of uniformly equicontinuous functions in $\mcl C([0,T], \hB^{{\be},\si}_{p/9})$.
\end{rem}

%
%
%
%
%
%
%
\section{Construction of solutions in the plane}
\label{s:exist-full}

\begin{thm}[Existence of solutions in the plane]
Let $T > 0$, $\be < 2$, $ \al >0$ be sufficiently small, $p$ be sufficiently large, and $\si > 2$. Let $X_0 \in \hB^{-\al,\si}_{3p}$, and let $\un Z = (Z,Z^{:2:},Z^{:3:})$ be as in \eqref{e:DefZN} (that is, $Z$ is the solution of \eqref{e:eqZ}, and $Z^{:n:}$ are its Wick powers). With probability one, there exists $Y \in \mcl C([0,T],\hB^{\be,\si}_{p/9})$ solving \eqref{e:eqY}.
\label{t:exist-full}
\end{thm}
\begin{proof}
Recall that we denote the periodic approximations of $\un{Z}$ by 
$$\un{Z}_{\cdot; M} = (Z_{\cdot; M},Z_{\cdot; M}^{:2:},Z_{\cdot; M}^{:3:}).
$$
By Corollary~\ref{cor:BoundsOnZ1}, for any $\al' > \al$ and every integer $M \ge 1$, the quantity
$$
 \sup_{0 \le t \le T} \Ll(\|Z_{t,M}\|_{\tB_\infty^{-\al-\frac2{3p},M}} \, \vee \,  t^{\al'+\frac2{3p}} \|Z^{:2:}_{t,M}\|_{\tB_\infty^{-\al-\frac2{3p},M}} \, \vee \, t^{2\Ll(\al'+\frac2{3p}\Rr)} \|Z^{:3:}_{t,M}\|_{\tB_\infty^{-\al-\frac2{3p},M}} \Rr)
$$
is finite almost surely. For $\al + \frac{2}{3p} < \al' + \frac{2}{3p}$ sufficiently small, Theorem~\ref{t:global-torus} ensures that there exists $Y_{\cdot;M} \in \mcl C([0,T],\tB^{\be,M}_\infty)$ such that
\begin{equation}
\label{e:eqY-weakM}
Y_{t;M} = \int_0^t e^{(t-s) \Delta} \, \Psi(Y_{s;M}, \un{Z}_{s;M}) \, \d s.
\end{equation}
In particular, $Y_{\cdot;M} \in \mcl C([0,T],\hB^{\be,\si}_p)$. 

We further impose that $p\ge 3^4$ be sufficiently large that 
$$
\frac{p-1}{p-4} \ \frac{\beta}{2} < 1,
$$
and then $0 < \al < \al'$ sufficiently small that \eqref{e:conds-p-full} and \eqref{e:cond:tdbe} hold. We learn from Corollary~\ref{cor:BoundsOnZ2} that  with probability one, there exists a subsequence $(M_k)_k$ tending to infinity and such that
$$
\sup_{k} \|\un{Z}_{\cdot;M_k}\|_\ZZt < \infty.
$$
By Proposition~\ref{p:pointwise-in-time-apriori}, it thus follows that with probability one,
\begin{equation}
\label{e:strong-bound}
\sup_{k} \  \sup_{t \le T}  \|Y_{t;M_k}\|_{\hB^{{\be},\si}_{p/9}} < \infty.
\end{equation}
By Remark~\ref{r:uniformly-equi}, with probability one, $(Y_{\cdot;M_k})_{k}$ is a family of uniformly equicontinuous functions in $\mcl C([0,T], \hB^{\be,\si}_{p/9})$. 
%
By Proposition~\ref{p:compact} (or Proposition~\ref{p:compact-poly}), for every $\si' >\si$, we can thus extract a subsequence  that converges uniformly in $\mcl C([0,T], \hB^{\be,\si'}_{p/9})$ to some $Y$; and moreover, $Y \in \mcl C([0,T], \hB^{\be,\si}_{p/9})$. It then suffices to pass to the limit in \eqref{e:eqY-weakM} to obtain \eqref{e:eqY-mild}, 
using the fact ensured by Corollary~\ref{cor:BoundsOnZ2} that for $n =1,2,3$, with probability one, $
\sup_{0 \le t \le T} t^{(n-1)\al'}\|Z^{:n:}_{t;M_k} - Z^{:n:}_t\|_{\hB_{\frac{3p}{n}}^{-\al',\si'}} \xrightarrow[k \to \infty]{} 0.
$
\end{proof}

%
%
%
%
%
%
%
\section{Uniqueness of solutions in the plane}
\label{s:uniqueness}

Consider the parabolic Anderson problem
\begin{equation}
\label{e:parab-And}
\Ll\{
\begin{array}{l}
\dr_t Y  = \Delta Y + W Y \qquad (\text{on } [0,T] \times \R^2), \\  
Y(0,\cdot) = 0,
\end{array}
\Rr.
\end{equation}
interpreted in the mild sense:
\begin{equation}
	\label{e:mild-parab-And}
	Y_t = \int_0^t e^{(t-s)\Delta} (W_s Y_s) \, \d s.
\end{equation}
 We want to find sufficient conditions on $W$ and $Y$ to guarantee that $Y = 0$.

\begin{thm}[Uniqueness for the parabolic Anderson problem]
	Assume that there exists $K < \infty$, $p \ge 1$, $\mu_0 > 0$ and $\al, \al' , b\in (0,\infty)$ such that for every $\mu \le \mu_0$ and $t \le T$,
	\begin{equation}
		\label{e:uniq-hyp1}
		\|W_t\|_{\B^{-\al,\mu}_{p}} \le K \, t^{-\al'} \, \mu^{-b}.
	\end{equation}
	Let $Y$ be a solution of \eqref{e:parab-And} such that for some $C < \infty$, $\td{\al} > \al$ and $c < \infty$, it holds for every $\mu \le \mu_0$ and $t \le T$ that
	\begin{equation}
		\label{e:uniq-hyp2}
		\|Y_t\|_{\B^{\td{\al},\mu}_p} \le C \mu^{-c}.
	\end{equation}
	If
	\begin{equation}
		\label{e:uniq-cond}
		\frac{\al + \td{\al}}{2} + \frac{1}{p} + b < 1  \quad \text{ and } \quad \frac{\al + \td{\al}}{2}  + \frac{1}{p} + \al' < 1,
	\end{equation}
	then $Y = 0$.
		\label{t:unique-anders}
\end{thm}
\begin{rem}
	\label{r:c-not-here}
	The constant $c$ does not appear in the condition \eqref{e:uniq-cond}. This is a manifestation of the fact that \eqref{e:uniq-hyp2} can be somewhat weakened if desired, as the reader can easily check from the proof.
\end{rem}

The proof relies on the following estimate.
\begin{prop}[Recursive estimate]
\label{p:recurs}
	Let $\td{\al} > \al$, $p \ge 1$, $a > 0$ and $\mu > 0$. There exists $C > 0$ such that for every $n\ge 1$,
	\begin{equation}
		\label{e:recurs}
		\|Y_t\|_{\B^{\td{\al},\frac{\mu}{n^a}}_p} \le C \int_0^t (t-s)^{-\frac{\al + \td{\al}}{2} - \frac{1}{p}} \ \|W_s\|_{\B^{-\al,{\mu_n}}_{p}} \ \|Y_s\|_{\B^{\td{\al},\frac{\mu}{(n+1)^a}}_p} \ \d s, 
	\end{equation}
	where 
	\begin{equation}
	\label{e:def:mu_n}
	\mu_n =  \frac{\mu}{n^a} - \frac{\mu}{(n+1)^a} .
\end{equation}
\end{prop}
\begin{proof}
Note that 
$$
\frac{1}{(p/2)} \frac{\mu}{2 n^a}  = \frac{\mu}{p}\Ll( \frac{1}{n^a} - \frac{1}{(n+1)^a}    \Rr)  + \frac{\mu}{p(n+1)^a}  = \frac{\mu_n}{p} + \frac{1}{p}\frac{\mu}{(n+1)^a}.
$$
It thus follows from Remark~\ref{r:multipl2} that
\begin{equation}
	\label{e:recurs-1}
	\|W_s Y_s\|_{\B^{-{\al},\frac{\mu}{2n^a} }_{p/2}} \lesssim \|W_s\|_{\B^{-\al,\mu_n}_{p}} \ \|Y_s\|_{\B^{\td{\al},\frac{\mu}{(n+1)^a}}_p}.
\end{equation}
%
We now observe that
\begin{align*}
	\|Y_t\|_{\B^{\td{\al},\frac{\mu}{n^a} }_p} & \le \int_0^t \|e^{(t-s)\Delta} (W_s Y_s)\|_{ \B^{\td{\al},\frac{\mu}{n^a} }_p} \ \d s \\
		& \lesssim \int_0^t (t-s)^{-\frac{\al + \td{\al}}{2} - \frac{1}{p}   } \, \|W_s Y_s\|_{\B^{-{\al} - \frac{2}{p} ,\frac{\mu}{n^a} }_p} \ \d s \\
		& \lesssim \int_0^t (t-s)^{-\frac{\al + \td{\al}}{2} - \frac{1}{p}} \,  \|W_s Y_s\|_{\B^{-{\al},\frac{\mu}{2 n^a} }_{\frac{p}{2} }} \ \d s,
\end{align*}
where we used Proposition~\ref{p:smooth-besov} in the second step, and Proposition~\ref{p:embed} in the third. The conclusion follows by \eqref{e:recurs-1}.
\end{proof}

We now prepare for a Gronwall-type argument via the following lemma.
\begin{lem}[Iterated integrals]
	\label{l:iter-int}
	Let $\ga_1, \ga_2 \ge 0$ be such that $\ga_1 + \ga_2 < 1$, and define recursively
	\begin{align*}
		\label{}
		I_0(t) & = 1, \\
		I_{n+1}(t) & = \int_0^t (t-s)^{-\ga_1} \, s^{-\ga_2} \,   I_n(s) \, \d s \qquad (n \in \N).
	\end{align*}
	For every $\td{\ga} > \ga_1$ and $T < \infty$, there exists $C < \infty$ such that uniformly over $n \in \N$ and $t \le T$,
	\begin{equation}
		\label{e:iter-int}
		I_n(t) \le \frac{C}{(n!)^{1-\td{\ga}}}.
	\end{equation}
\end{lem}
\begin{proof}
Let $\ga = \ga_1 + \ga_2$ and, for $n \in \N$,
$$
J_n = \int_0^1 (1-u)^{-\ga_1} u^{-\ga_2 + n(1-\ga)} \, \d u.
$$
We first show by induction on $n$ that
\begin{equation}
	\label{e:iter-int-2}
	I_n(t) = t^{n(1-\ga)} \prod_{k = 0}^{n-1} J_k.
\end{equation}
The case $n = 0$ is trivial (we understand the product as being $1$). For $n \in \N$, a change of variables gives
$$
I_{n+1}(t) = t^{1-\ga} \int_0^1 (1-u)^{-\ga_1} u^{-\ga_2} \, I_n(tu) \, \d u,
$$
so \eqref{e:iter-int-2} implies the same statement with $n$ replaced by $n+1$.

In order to conclude, it suffices to show that there exists $C < \infty$ such that for every $n$ sufficiently large,
\begin{equation}
	\label{e:iter-int-1}
	J_n  \le C \, \Ll(\frac{\log n}{n} \Rr)^{1-\ga_1}. 
\end{equation}
We consider $n$ sufficiently large that $-\ga_2 + n(1-\ga) \ge 0$. For any $\eps \in [0,1)$, we can decompose the integral defining $J_n$ along $\int_0^{1-\eps} + \int_{1-\eps}^1$ and obtain that
\[
	J_n \le \eps^{-\ga_1} (1-\eps)^{-\ga_2 + n(1-\ga)}  + (1-\ga_1)^{-1} \eps^{1-\ga_1}.
\]
Choosing $\eps = c (\log n)/n$ for some constant $c$ gives an upper bound that is asymptotically equivalent to
\[
	\Ll( \frac{n}{c\log n}   \Rr)^{\ga_1} n^{-c(1-\ga)} + (1-\ga_1)^{-1} \Ll( \frac{c \log n}{n}   \Rr)^{1-\ga_1}.
\]
It then suffices to fix $c$ sufficiently large to obtain \eqref{e:iter-int-1}, and thus conclude the proof.
\end{proof}

\begin{proof}[Proof of Theorem~\ref{t:unique-anders}]
In view of \eqref{e:uniq-cond}, there exists $a > 0$ such that 
\begin{equation}
	\label{e:cond-ab}
	(1+a)b < 1-\frac{\al + \td{\al}}{2} - \frac{1}{p}.
\end{equation}
We fix $\mu = \mu_0$. By Proposition~\ref{p:recurs} and the assumption in \eqref{e:uniq-hyp1}, there exists $C > 0$ such that for every $n \ge 1$,
$$
		\|Y_t\|_{\B^{\td{\al},\frac{\mu}{n^a}}_p} \le C \, \mu_n^{-b} \, \int_0^t (t-s)^{-\frac{\al + \td{\al}}{2} - \frac{1}{p}} \, s^{-\al'} \, \|Y_s\|_{\B^{\td{\al},\frac{\mu}{(n+1)^a}}_p} \ \d s,
$$
where $\mu_n$ is as in \eqref{e:def:mu_n}. 
We define $\ga_1 = \frac{\al + \td{\al}}{2} + \frac1{p}$, $\ga_2 = \al'$, and $I_n(t)$ as in Lemma~\ref{l:iter-int} [note that $\ga_1 + \ga_2 < 1$ by \eqref{e:uniq-cond}]. By induction, we obtain
\begin{equation}
\label{e:uniq-induct}
\|Y_t\|_{\B^{\td{\al},\mu}_p} \le C^n \ \Ll(\prod_{k=1}^{n} \mu_k\Rr)^{-b} \ I_n(t) \ \sup_{s \le t} \|Y_s\|_{\B^{\td{\al},\frac{\mu}{(n+1)^a}}_p} .
\end{equation}
As $n$ tends to infinity, we have
$$
\mu_n \sim \frac{a\mu}{n^{1+a}},
$$ 
so in particular (for some possibly larger $C < \infty$),
$$
\Ll(\prod_{k=1}^{n} \mu_k \Rr)^{-b} \le C^n (n!)^{(1+a)b}.
$$
In view of \eqref{e:cond-ab}, we can define $\td{\ga}$ such that $\ga_1 < \td{\ga} < 1$ and 
\begin{equation}
	\label{e:cond-td-ga}
	(1+a)b < 1-\td{\ga}.
\end{equation}
By Lemma~\ref{l:iter-int} and \eqref{e:uniq-hyp2}, the \rhs\ of \eqref{e:uniq-induct} is thus bounded by
$$
C^n \ \frac{(n!)^{(1+a)b}}{(n!)^{1-\td{\ga}}}  \Ll(\frac{(n+1)^a}{\mu} \Rr)^c.
$$
This quantity tends to $0$ with $n$ by \eqref{e:cond-td-ga}, so $\|Y_t\|_{\B^{\td{\al},\mu}_p} = 0$ for every $t \le T$, and thus $Y = 0$.
\end{proof}

\begin{thm}[Uniqueness of solutions]
\label{t:uniq}
Let $T > 0$, $\be < 2$, $\si > 2$, $\al >0$ be sufficiently small, and $p$ be sufficiently large (depending on $\si$). Let $X_0 \in \hB^{-\al,\si}_{p}$, and let $\un Z = (Z,Z^{:2:},Z^{:3:})$ be as in \eqref{e:DefZN} (that is, $Z$ is the solution of \eqref{e:eqZ}, and $Z^{:n:}$ are its Wick powers). With probability one, if $Y^{(1)},Y^{(2)} \in \mcl C([0,T],\hB^{\be,\si}_{p})$ are two solutions of \eqref{e:eqY}, then $Y^{(1)} = Y^{(2)}$.
\end{thm}
\begin{proof}
The process $Y := Y^{(1)} - Y^{(2)}$ solves
$$
Y_t = \int_0^t e^{(t-s)\Delta} \, W_s  Y_s  \, \d s,
$$
where 
$$
W = -(Y^{(1)})^2 - Y^{(1)}Y^{(2)} - (Y^{(2)})^2 - 3 (Y^{(1)} + Y^{(2)}) Z - 3 Z^{:2:} + a.
$$
We verify that for suitable choices of parameters, the conditions of Theorem~\ref{t:unique-anders} are satisfied. Observe that
$$
\sup_{\mu } \mu^{cp} \, e^{-\mu |x|_*^\de} \lesssim |x|_*^{-\de c p},
$$
and as a consequence, 
$$
\sup_{\mu } \mu^{c} \, \|\cdot \|_{\Lm^{p}} \lesssim \|\cdot\|_{\wh L_{\de c p}^p},
$$
\begin{equation}
\label{e:twist}
\sup_{\mu } \mu^c \, \|\cdot\|_{\B^{\td \al,\mu}_{p}} \lesssim \|\cdot\|_{\hB^{\td \al,\de c p}_p}.
\end{equation}
Since the solutions $Y^{(1)}$, $Y^{(2)}$ are in $\mcl C([0,T],\hB^{\be,\si}_{p})$, we get that for some $c < \infty$,
$$
\sup_{0 \le t \le T} \sup_{\mu\le \mu_0} \mu^c \, \|Y_t\|_{\B^{\be,\mu}_{p}} < \infty \qquad \text{a.s.}.
$$
By Remark~\ref{r:besov-mu} and up to a redefinition of $c$, we also have that for every $\td\al \le \be$,
$$
\sup_{0 \le t \le T} \sup_{\mu\le \mu_0} \mu^c \, \|Y_t\|_{\B^{\td \al,\mu}_{p/2}} < \infty \qquad \text{a.s.}.
$$
By Corollary~\ref{cor:BoundsOnZ2}, with probability one,
$$
\sup_{0 \le t \le T} \Ll(\|Z_t\|_{\hB^{-\al,\si}_{p}} \vee t^{\al'} \|Z^{:2:}_t\|_{\hB^{-\al,\si}_{p/2}} \Rr) < \infty.
$$
By the multiplicative inequalities, it follows that with probability one,
$$
\sup_{0 \le t \le T} t^{\al'}\|W_t\|_{\hB^{-\al,\si}_{p/2}} < \infty.
$$
Using \eqref{e:twist}, we see that for any given $b > 0$ and $0 < \al < \al'$, we can choose $p$ sufficiently large that
$$
\sup_{0 \le t \le T} \sup_{\mu\le \mu_0} \mu^b \, t^{\al'} \|W_t\|_{\B^{- \al,\mu}_{p/2}} < \infty.
$$
The conclusion thus follows from Theorem~\ref{t:unique-anders}.
\end{proof}

\appendix

\section{Gevrey classes}
\label{s:AppA}
We begin by recalling two classical facts about Gevrey classes: first, the stability of $\mcl G^\theta$ under multiplication; second, that the Fourier transform of a function in $\Gg$ has fast decay at infinity. We then prove a third result that was needed in Subsection~\ref{ss:heat-flow}, whose proof is in large measure a combination of the proofs of these two more classical facts.

\begin{prop}[Stability under multiplication]
	\label{p:gevrey-mult}
	For every $\th \ge 1$, the Gevrey class $\G^\th$ is stable under multiplication.
\end{prop}
\begin{proof}
Let $f,g \in \G^\th$, and let $K$ be a compact subset of $\R^d$. There exists $C < \infty$ such that for every $x \in K$ and $n \in \N^d$,
$$
|\dr^n f|(x)\ , \ |\dr^n g|(x) \le C^{|n|+1} (n!)^{\th}.
$$
We have
\begin{equation*}
	\dr^n (fg) = \sum_{m \le n} {n \choose m} \, \dr^{n-m} f \, \dr^m g,
\end{equation*}
where we use the multi-index notation ${n \choose m} = \frac{n!}{(n-m)! \, m!}$. The number of $m$'s such that $|m| \le |n|$ is ${|n|+d \choose d}$, so it suffices to show that on $K$,
$$
{n \choose m} \, \Ll|\dr^{n-m} f \, \dr^m g\Rr| \le C^{|n|+1} (n!)^{\th}.
$$
On $K$, the \lhs\ above is bounded by
$$
n! C^{|n|+1} [(n-m)! \, m!]^{\th-1},
$$
and since $(n-m)! \, m! \le n!$, the proof is complete.
\end{proof}
\begin{prop}[Decay of the Fourier transform]
\label{p:decay-a}
If $f \in \Gg$, then there exists $c > 0$ and $C < \infty$ such that
\begin{equation}
	\label{e:decay-a}
	|\hat{f}(\ze)| \le C e^{-c|\ze|^{1/\th}}.
\end{equation}
\end{prop}
\begin{proof}
It suffices to show that \eqref{e:decay-a} holds uniformly over $|\ze| \ge 1$. For any $n \in \N^d$, writing $\ze^n = \ze_1^{n_1} \cdots \ze_d^{n_d}$, we observe that
$$
|\ze^n \hat{f} (\ze)| = |\widehat{\dr^n f}(\ze)| \le \int | \dr^n f| \le C^{|n|+1} (n!)^{\theta},
$$
where we used the fact that $f$ is compactly supported and in $\G^\th$ in the last step. As a consequence, for every positive integer $m$, letting $M = \lfloor m/\th \rfloor$ (the integer part of $m/\th$) and $\ov{M} = M+1$, we have
$$
|\ze|^{m/\th} \, |\hat{f}(\ze)| \stackrel{(|\ze| \ge 1)}{\le} |\ze|^{\ov{M}} \, |\hat{f}(\ze)| \le C^{\ov{M}} \Ll( |\ze_1|^{\ov{M}} + \cdots + |\ze_d|^{\ov{M}} \Rr) |\hat{f}(\ze)| \le C^{\ov{M}+1} (\ov{M}!)^{\th}
$$
(we use $C$ as a generic constant whose value can change from an inequality to another). One can then check that the \rhs\ above is bounded by 
$$
C^{M+1} (M!)^\th \le C^{M+1} M^{\th M} \le C^{m+1} m^{m}.
$$
We have thus shown that uniformly over $|\ze| \ge 1$ and $m$, 
$$
|\hat{f}(\ze)| \le C \Ll( \frac{C m}{|\ze|^{1/\th}} \Rr)^m.
$$
Since
$$
e^{c|\ze|^{1/\th}} |\hat{f}(\ze)| = \sum_{m = 0}^{+\infty} \frac{\Ll(c|\ze|^{1/\th}\Rr)^{m}}{m!} |\hat{f}(\ze)| \le C \sum_{m = 0}^{+\infty} \frac{(cCm)^m}{m!} 
$$
and $m! \ge (m/e)^m$, it suffices to choose $c > 0$ sufficiently small that $cCe < 1$ to obtain the result.
\end{proof}
\begin{prop}[Exponential decay]
	\label{p:gevrey-ann}
	Let $\phi \in \Gg$ be supported in an annulus $\mcl{C}$, and let
	$$
	\ov{g}_{t}(x) = \int e^{i x \cdot \ze} \phi(\ze) \, e^{-t|\ze|^2} \, \d \ze.
	$$
	There exists $C < \infty$ and $c > 0$ such that uniformly over $x \in \R^d$ and $t \ge 0$,
	$$
	|\ov{g}_{t}(x)| \le C  e^{-ct - c |x|^{1/\th}}.
	$$
\end{prop}
\begin{proof}
It suffices to show that there exists $C < \infty$ and $c > 0$ such that uniformly over $\ze \in \mcl{C}$ and $n \in \N^d$,
\begin{equation}
	\label{e:gevrey-ann}
	\dr^n \Ll( e^{-t|\ze|^2} \Rr) \le C^{|n|+1} (n!)^\th e^{-ct}.
\end{equation}
Indeed, given this, the proof of Proposition~\ref{p:gevrey-mult} shows that uniformly over $\ze$,
$$
\dr^{n} \Ll( \phi(\ze) \, e^{-t|\ze|^2}  \Rr) \le C^{|n|+1} (n!)^\th e^{-ct},
$$
and then we can repeat the proof of Proposition~\ref{p:decay-a} to obtain the result.

We now show that \eqref{e:gevrey-ann} holds with $\th = 1$. Let us write $f(y) = e^{ty}$ and $g(\ze) = -|\ze|^2$. By Fa\'a di Bruno's formula, for $n \in \N$ and $1 \le i_1,\ldots,i_n \le d$,
$$
\frac{\dr^n}{\dr_{\ze_{i_1}} \cdots \dr_{\ze_{i_n}}} (f(g(\ze)) = \sum_{\pi \in \Pi} f^{(|\pi|)}(g(\ze)) \prod_{B \in \pi} \frac{\dr^{|B|}}{\prod_{k \in B} \dr_{\ze_{i_k}}} g(\ze),
$$
where $\Pi$ is the set of partitions of $\{1,\ldots,n\}$. Because of the form of $g$, the term indexed by $B$ in the last product is zero unless $|B| \le 2$. It thus suffices to focus on showing that
$$
\Ll|  \sum f^{(|\pi|)}(g(\ze))  \Rr| \le C^{|n|+1} \, n! \, e^{-ct},
$$ 
where the sum runs over partitions $\pi$ whose constituents have at most two elements.
Moreover, $f^{(|\pi|)}(y) = t^{|\pi|} e^{ty}$, and there are
$$
\frac{n!}{m_1! \, m_2! \, 2^{m_2}}
$$
partitions of $\{1,\ldots,n\}$ by $m_1$ singletons and $m_2$ sets of $2$ elements ($m_1 + 2m_2 = n$). Let $r \in (0,2)$ be such that $|\ze|^2 \le r \Rightarrow \ze \notin \mcl{C}$. It suffices to check that
\begin{equation}
	\label{e:gev-ann2}
	\Ll|  \sum_{m_1 + 2m_2 = n} \frac{n!}{m_1! \, m_2! \, 2^{m_2}} t^{m_1 + m_2} e^{-t r}\Rr| \le C^{|n|+1} \, n! \, e^{-ct}.
\end{equation}
Since $2^N = \sum_{k = 0}^N {N \choose k}$, we have $(m_1 + m_2)! \le 2^{m_1 + m_2} \, m_1! \, m_2!$, and thus, for $m_1 + 2m_2 = n$,
$$
\frac{1}{m_1! \, m_2! } t^{m_1 + m_2} \le 2^{n} \frac{t^{m_1+m_2}}{(m_1+m_2)!} =\Ll( \frac{4}{r}   \Rr)^{n} \frac{(tr/2)^{m_1+m_2}}{(m_1+m_2)!} \le \Ll( \frac{4}{r}   \Rr)^{n} e^{tr/2}.
$$
This implies \eqref{e:gev-ann2} (with $c = r/2$), and thus concludes the proof.
\end{proof}

\end{document}